\documentclass[final,hidelinks,onefignum,onetabnum]{siamart220329}

\usepackage{amssymb}
\usepackage{marginnote}
\usepackage{comment}
\usepackage{enumitem}
\usepackage{graphicx}
\usepackage{bm}

\def\Xint#1{\mathchoice
{\XXint\displaystyle\textstyle{#1}}%
{\XXint\textstyle\scriptstyle{#1}}%
{\XXint\scriptstyle\scriptscriptstyle{#1}}%
{\XXint\scriptscriptstyle%
\scriptscriptstyle{#1}}%
\!\int}
\def\XXint#1#2#3{{\setbox0=\hbox{$#1{#2#3}{%
\int}$ }
\vcenter{\hbox{$#2#3$ }}\kern-.6\wd0}}
\def\barint{\, \Xint -} 
\def\bariint{\barint_{} \kern-.4em \barint}
\def\bariiint{\bariint_{} \kern-.4em \barint}
\renewcommand{\iint}{\int_{}\kern-.34em \int} 
\renewcommand{\iiint}{\iint_{}\kern-.34em \int} 

\DeclareMathAlphabet{\mathcal}{OMS}{cmsy}{m}{n}

\newcommand{\R}{\mathbb{R}}
\newcommand{\C}{\mathbb{C}}

\newcommand{\Z}{\mathbb{Z}}
\newcommand{\T}{\mathbb{T}}

\newcommand{\PP}{\mathbb{P}}
\newcommand{\bI}{\mathbf{I}}

\newcommand{\cL}{\mathcal L}

\newcommand{\bu}{\bm{u}}

\newcommand{\bx}{\bm{x}}

\newcommand{\bk}{\bm{k}}

\newcommand{\be}{\bm{e}}
\newcommand{\bp}{\bm{p}}
\newcommand{\bv}{\bm{v}}
\newcommand{\divp}{{\rm{div}}_p}

\newcommand{\p}{\partial}
\newcommand{\la}{\langle}
\newcommand{\ra}{\rangle}
\newcommand{\les}{\lesssim}

\newcommand{\norm}[1]{\lVert #1 \rVert}

\let\div\relax
\DeclareMathOperator{\div}{div}

\let\tilde\relax
\newcommand{\tilde}[1]{\widetilde{#1}}
\let\hat\relax
\newcommand{\hat}[1]{\widehat{#1}}

\newcommand{\abs}[1]{\left\lvert #1 \right\rvert}
\newcommand{\mc}[1]{\mathcal{#1}}
\newcommand{\wh}[1]{\widehat{#1}}
\newcommand{\wt}[1]{\widetilde{#1}}

\let\Re\relax
\DeclareMathOperator{\Re}{Re}
\let\Im\relax
\DeclareMathOperator{\Im}{Im}

\usepackage{lipsum}
\usepackage{amsfonts}
\usepackage{graphicx}
\usepackage{epstopdf}
\usepackage{algorithmic}
\ifpdf
  \DeclareGraphicsExtensions{.eps,.pdf,.png,.jpg}
\else
  \DeclareGraphicsExtensions{.eps}
\fi

\newsiamremark{remark}{Remark}
\newsiamremark{hypothesis}{Hypothesis}
\crefname{hypothesis}{Hypothesis}{Hypotheses}
\newsiamthm{claim}{Claim}

\headers{On the stabilizing effect of swimming in an active suspension}{D. Albritton and L. Ohm}

\title{On the stabilizing effect of swimming in an active suspension}
\author{Dallas Albritton \and Laurel Ohm}

\usepackage{amsopn}

\ifpdf
\hypersetup{
  pdftitle={On the stabilizing effect of swimming in an active suspension},
  pdfauthor={D. Albritton and L. Ohm}
}
\fi

\begin{document}

\maketitle

\begin{abstract}
    We consider a kinetic model of an active suspension of rod-like microswimmers. 
    In certain regimes, swimming has a stabilizing effect on the suspension. We quantify this effect near homogeneous isotropic equilibria $\overline{\psi} = \text{const}$. Notably, in the absence of particle (translational and orientational) diffusion, swimming is the only stabilizing mechanism. On the torus, in the non-diffusive regime, we demonstrate linear Landau damping up to the stability threshold predicted in the applied literature. With small diffusion, we demonstrate nonlinear stability of arbitrary equilibrium values for pullers (front-actuated swimmers) and enhanced dissipation for both pullers and pushers (rear-actuated swimmers) at small concentrations. On the whole space, we prove nonlinear stability of the vacuum equilibrium due to generalized Taylor dispersion. 
\end{abstract}



\section{Introduction}

We consider a kinetic model of a dilute suspension of rod-like microswimmers proposed by Saintillan and Shelley~\cite{saintillan2008instabilitiesPOF,saintillan2008instabilitiesPRL} and, independently, Subramanian and Koch~\cite{subramanian2009critical}. The number density $\psi(\bx,\bp,t)$ of swimmers with center-of-mass position $\bx\in \Omega^d$ and orientation $\bp\in S^{d-1}$, $d=2,3$, evolves according to a Smoluchowski equation:
\begin{equation}\label{eq:smoluchowski}
    \p_t \psi + \bp\cdot\nabla_x\psi + \bu\cdot\nabla_x\psi + \div_p\left[({\bf I}-\bp\otimes\bp)(\nabla\bu \bp) \psi \right] = \nu \Delta_p\psi +\kappa\Delta_x\psi \, ,
\end{equation}
where $\nabla_p$, $\div_p$, and $\Delta_p = \div_p \nabla_p$ denote the gradient, divergence, and Laplacian, respectively, on $S^{d-1}$, see~\eqref{eq:flatgradient}--\eqref{eq:divcalc}. 
We consider the domains $\Omega^d=\T^d := \R^d/(2\pi\Z)^d$ and $\Omega^d=\R^d$. 
The velocity field $\bu(\bx,t)$ of the surrounding fluid medium satisfies the Stokes equations forced by the divergence of an active stress $\bm{\Sigma}$, which measures the local alignment of swimmers:
\begin{align}
- \Delta \bu + \nabla q &= \div \bm{\Sigma}, \quad \div\bu =0 \, , \label{eq:stokes} \\
\bm{\Sigma}(\bx,t)&=\iota \int_{S^{d-1}} \psi(\bx,\bp,t)\, \bp\otimes\bp \, d\bp \, , \quad \iota\in\{\pm\} \, . \label{eq:activestress}
\end{align}
with $\int_{\T^d}\bu\,d\bx=0$ when $\Omega=\T^d$. The coefficient $\iota$ in the active stress $\bm{\Sigma}$ corresponds to the sign of the force dipole $\iota (\bp \cdot \nabla_x)\delta \bp$ exerted by each swimmer on the fluid and depends on the swimming mechanism: $\iota=+$ for \emph{pullers} (front-actuated swimmers), such as \emph{C. reinhardtii}, and $\iota=-$ for \emph{pushers} (rear-actuated swimmers), such as \emph{E.~coli}.

It is instructive to rewrite~\eqref{eq:smoluchowski} schematically as
\begin{equation}
    \p_t \psi + \div_x (\dot \bx \psi) + \div_p (\dot \bp \psi) = \nu \Delta_p\psi +\kappa\Delta_x\psi \, ,
\end{equation}
where the fluxes
\begin{equation}
    \label{eq:xpderivs}
    \dot \bx = \bp + \bu \, , \quad \dot \bp = (\bI - \bp \otimes \bp) \nabla \bu \bp
\end{equation}
signify that (i) particles swim in the direction of their orientation, (ii) 
particles are advected by the surrounding fluid flow, and (iii) particles' orientations evolve according to a Jeffery term~\cite{jeffery1922motion} for the rotational dynamics of an elongated particle in Stokes flow.
The terms on the right-hand side of~\eqref{eq:smoluchowski} capture the orientational and center-of-mass diffusion, respectively, with coefficients $0 < \nu,\kappa \ll 1$. We refer to~\cite{saintillan2015theory} for a detailed derivation of the model, and we discuss our particular nondimensionalization in Appendix~\ref{app:nondim}.

\bigskip

The model~\eqref{eq:smoluchowski}-\eqref{eq:activestress} may be considered as a minimal model in which to study the large-scale flows generated by bacterial activity as seen in experiments. (See~\cite{saintillan2015theory} for comparisons between experimental, computational, and theoretical results.) A key difference between this model and, for example, the Doi--Edwards model for passive polymers~\cite{doi1981molecular,doi1986theory} is the presence of the swimming term $\bp\cdot\nabla_x\psi$ in \eqref{eq:smoluchowski}. This term seems to be necessary for the concentration fluctuations characteristic of certain bacterial suspensions \cite{saintillan2008instabilitiesPOF}, but its role is quite complex. The simplest setting in which to study the effects of swimming is near the uniform isotropic equilibrium $\psi\equiv \text{constant } \overline{\psi}$. 
\emph{What role does swimming play in stabilizing the uniform isotropic equilibrium?}

In this paper, we identify and quantify three near-equilibrium effects of swimming. On $\T^d$, the value of the constant $\overline\psi$ 
is a free parameter (see Appendix \ref{app:nondim} for our nondimensionalization of the model) and plays an important role in our results. Furthermore, the expected behavior of the system depends on whether we consider pullers or pushers in \eqref{eq:activestress}.
Defining the relative conformational entropy $\mc{S}(t)=\int_{x,p} (\psi/\overline\psi) \log\left(\psi/\overline\psi\right) \,  d\bp \,d\bx$, solutions of~\eqref{eq:smoluchowski}-\eqref{eq:activestress} on $\T^d$ formally satisfy the following $H$-theorem: 
\begin{equation}
    \label{eq:Htheorem}
    \begin{aligned}
    \overline\psi\frac{d\mc{S}}{dt} &= -\iota d \int_{\T^d} \abs{\nabla\bu}^2\, d\bx  - 4\int_{\T^d\times S^{d-1}}\left(\nu |\nabla_p\sqrt{\psi}|^2+ \kappa|\nabla_x\sqrt{\psi}|^2 \right)\,d\bp \,d\bx  \,,
    \end{aligned}
\end{equation}
from which we can see that pullers ($\iota=+$) always have decreasing conformational entropy, whereas pushers ($\iota=-$) may not. 

Our main results may be summarized as follows ($d=2,3$ unless stated otherwise):
\begin{itemize}[leftmargin=*]
    \item \emph{Landau damping}. Solutions of the linearized inviscid  equations ($\nu = \kappa = 0$) on $\T^d$ decay algebraically due to \emph{phase mixing} provided that $\iota=+$ or $\overline{\psi} < \overline{\psi}^*$, which is given by a suitable Penrose condition. \\
    \item \emph{Taylor dispersion}. The vacuum state $\overline{\psi} = 0$ in $\R^3$ is nonlinearly stable with respect to small perturbations due to the dispersive effect of the operator $\bp \cdot \nabla_x - \nu \Delta_p$. As a byproduct of the analysis, arbitrarily large puller $(\iota = +)$ equilibria on $\T^d$ are nonlinearly exponentially stable with respect to small perturbations. \\
    \item \emph{Enhanced dissipation}. For small $\overline{\psi} \ll \nu^{-(\frac{1}{2}+)}$, the equilibrium $\overline{\psi}$ is nonlinearly exponentially stable on $\T^d$ with respect to small perturbations due to the hypocoercive effect of the operator $\bp \cdot \nabla_x - \nu \Delta_p$. Nearby solutions converge to their $\bx$-averages $\la \psi \ra(\bp,t) := \int \psi(\bx,\bp,t) \, d\bx$ in the \emph{enhancement time} $O(\nu^{-(\frac{1}{2}+)})$. The $\bx$-averages $\la \psi \ra$ converge to $\overline{\psi}$ in the \emph{diffusive time} $O(\nu^{-1})$ and are \emph{metastable}.
\end{itemize}


Our results on stability may be viewed as complementary to the wealth of computational literature on dynamics in the unstable pusher region, where perturbations to the uniform isotropic equilibrium can be seen to give rise to the emergence of collective swimmer motion and large-scale flows~\cite{hohenegger2010stability, koch2011collective, ohm2022weakly, saintillan2007orientational, saintillan2008instabilitiesPOF, saintillan2008instabilitiesPRL, saintillan2012emergence, saintillan2013active, saintillan2015theory, subramanian2009critical}. 
In particular, our results highlight the complex role of swimming in these collective dynamics. Without swimming, the isotropic state in pusher suspensions is always unstable for $0\le \nu,\kappa\ll1$ \cite{ohm2022weakly}, and, as we see here, swimming has a clear stabilizing effect. However,
as noted in \cite{saintillan2008instabilitiesPOF}, swimming is also a necessary ingredient for the particle density fluctuations observed in simulations. The interplay between the destabilizing active stress \eqref{eq:activestress} in the pusher case and stabilizing swimming is perhaps worthy of further mathematical exploration.

\subsection{Main results} In this section, we make the above informal assertions precise.

\subsubsection{Landau damping}
We consider the linearized inviscid equations:
\begin{align}
    \p_t  f + \bp \cdot \nabla_x f - d \overline{\psi} \nabla \bu:\bp\otimes\bp &= 0 \label{eq:lin_inviscid1}\\
    - \Delta \bu + \nabla q= \div \bm{\Sigma} \, , \quad \div \bu &= 0 \, , \label{eq:lin_inviscid2}
\end{align}
with active stress
\begin{equation}\label{eq:Sigma}
   \bm{\Sigma}[f](\bx,t) = \iota \int_{S^{d-1}} \bp\otimes\bp \,f(\bx,\bp,t) \, d\bp \, , \qquad \iota \in \{ \pm \} \, .
\end{equation}
Here $\overline{\psi} \geq 0$ is the constant background solution. Solutions satisfy 
\begin{equation}
    \label{eq:linearizedhtheorem}
    \frac{1}{2}\frac{d}{dt}\norm{f}_{L^2_{x,p}}^2  + \iota d\overline\psi \norm{\nabla\bu}_{L^2_x}^2 =0 \, ,
\end{equation}
which is a linearized version of the $H$-theorem~\eqref{eq:Htheorem}.

To better understand the underlying decay mechanism, we consider $\overline{\psi} = 0$ and perform a mode-by-mode analysis for $f(\bx,\bp,t) = h(\bp,t) e^{i\bk \cdot \bx}$ and $f^{\rm in} = f(\cdot,0)$. Without loss of generality, we may take $\bk = k \be_1$ and solve for $h$ as
\begin{equation}
    h = e^{-ikp_1 t} h^{\rm in} \, .
\end{equation}
Over time, the solution develops large oscillations and is transferred to higher and higher frequencies in $\bp$. This increasingly oscillatory behavior for an isolated spatial mode can be observed in numerical simulations by Hohenegger and Shelley~\cite[Figure 3]{hohenegger2010stability}. Notably, $h$ converges to zero \emph{weakly} in $L^2$ but not strongly, and the convergence can be quantified in negative Sobolev spaces via the method of stationary phase, see Section~\ref{sec:landaudamping}. This is known as \emph{phase mixing} in kinetic theory.

It is observed in~\cite{hohenegger2010stability, saintillan2008instabilitiesPOF, subramanian2009critical} that the linearized operator about large $\overline\psi$ admits unstable eigenvalues in the case of pushers but not pullers. Away from this pusher instability, we show that the phase mixing described above persists, causing the velocity field $\bu$ to quickly decay.
We call this phenomenon \emph{Landau damping}, terminology again borrowed from kinetic theory. 


 Letting $f$ denote the solution to the linearized system, 
 we show the following: 
\begin{theorem}[Linear Landau damping]
\label{thm:landaudamping}
Let $f^{\rm in} \in L^2_x H^{d+1}_p(\T^d \times S^{d-1})$. Suppose that $\iota = +$ or $\overline{\psi} < \overline{\psi}^*$. Then the velocity field $\bu$ generated by the solution $f \in C_t L^2_{x,p}(\T^d \times S^{d-1} \times \R_+)$ to the linearized PDE \eqref{eq:lin_inviscid1}--\eqref{eq:Sigma} on $\T^d$ satisfies
\begin{equation}
    \label{eq:nableudecaythm}
    \int \|\nabla \bu(\cdot,t)\|_{L^2_x}^2 \la t \ra^{d-\varepsilon} \, dt \les_{\overline{\psi},\varepsilon} \| f^{\rm in} \|_{L^2_x H^{d+1}_p}^2 \, 
\end{equation}
for all $\varepsilon > 0$, where $\la t \ra = \sqrt{1+t^2}$.
The solution $f$ decomposes as
\begin{equation}
    f = f_{\rm lin} + g \, ,
\end{equation}
where $f_{\rm lin} = e^{- \bp \cdot \nabla_x t} f^{\rm in}$ satisfies
\begin{equation}
    \| f_{\rm lin}(\cdot,t) \|_{L^2_x H^{-(d-1)}_p} \les \la t \ra^{-\frac{d-1}{2}} \| f^{\rm in} \|_{L^2_x H^{d-1}_p} \, ,
\end{equation}
and $g = d \overline{\psi} \int_0^t e^{- \bp \cdot \nabla_x (t-s)} (\nabla \bu : \bp \otimes \bp)(\cdot,s) \, ds$ satisfies, for all $\varepsilon > 0$,
\begin{equation}
    \label{eq:gdecaythm}
    \int \| g(\cdot,t) \|_{L^2_x H^{-(d+1)}_p}^2 \la t \ra^{d-\varepsilon} \, dt \les_{\overline{\psi},\varepsilon} \| f^{\rm in} \|_{L^2_x H^{d+1}_p}^2 \, .
\end{equation}
\end{theorem}

In the pusher case ($\iota=-$), the stability threshold $\overline\psi^*$, defined in Lemma~\ref{lem:penrose}, corresponds precisely to the eigenvalue crossing studied in the computational literature \cite{hohenegger2010stability, ohm2022weakly, saintillan2008instabilitiesPRL, saintillan2008instabilitiesPOF, saintillan2013active, subramanian2009critical}. It arises in the \emph{Penrose condition} (terminology borrowed from kinetic theory) on the Volterra equation for $\nabla\bu$ in the proof of Theorem~\ref{thm:landaudamping}.

The standard setting in which to investigate Landau damping is the \emph{Vlasov-Poisson equation}
\begin{equation}
    \label{eq:vlasovpoisson}
    \p_t f + \bv \cdot \nabla_x f + \bm{E} \cdot \nabla_v f = 0 \, ,
\end{equation}
\begin{equation}
    \label{eq:electricfield}
    \bm{E}(\bx,t) =  \iota \nabla \Delta^{-1} \varrho \, , \quad \varrho = \int_v f \, d\bv \, , \quad \iota \in \{ \pm \} \, ,
\end{equation}
near homogeneous equilibria $\bar{f}(\bv)$ (typically radial or Maxwellian), where $\iota = +$ is the repulsive (ionic) case and $\iota = -$ is the attractive (gravitational) case. Landau damping in this context has been thoroughly studied in the PDE community~\cite{mouhotvillani,bedrossianmasmoudimouhot,grenier2020landau,chaturvedi2021vlasov}. We will discuss below certain important differences between our model~\eqref{eq:smoluchowski}-\eqref{eq:activestress} and the Vlasov-Poisson equations~\eqref{eq:vlasovpoisson}-\eqref{eq:electricfield}.


\subsubsection{Taylor dispersion}

For particles swimming with speed $U_0$, various sources (see, e.g.,~\cite[p. 282]{lauga2020fluid} and~\cite[p. 326]{saintillan2015theory}) predict \emph{generalized Taylor dispersion}, namely, effective $\bx$-diffusion $\left( \kappa + \frac{U_0^2}{2d \nu} \right) \Delta_x$ in the regime $0 < \nu \ll 1$ of weak orientational diffusion. The hallmark of Taylor dispersion~\cite{frankel1989foundations, aris1956dispersion, manela2003generalized, taylor1954dispersion} is the inverse dependence of the effective viscosity coefficient on $\nu$.

We begin by extracting the Taylor dispersion effect in the linearized PDE with $\overline{\psi} = 0$:
\begin{equation}
    \label{eq:thisisfordemonstration}
    \p_t f + \bp \cdot \nabla_x f = \nu \Delta_p f \, .
\end{equation}
If $f$ solves the above PDE, then $g = e^{\kappa \Delta_x t} f$ solves the analogous PDE with additional diffusion $\kappa \Delta_x$, so there is no loss of generality in setting $\kappa = 0$ at the linear level. (This is a common feature of the linearized operators in this paper.)

Let $f = h(\bp,t) e^{i \bk \cdot \bx}$ and $k = |\bk|$. In Theorem~\ref{thm:linTayDisp} ($\nu \leq k$) and Theorem~\ref{thm:Phi} ($\nu \geq k$), we prove
\begin{equation}
    \label{eq:dampingoftheting}
    \| h \|_{L^2}  \leq e^{-c_0 \lambda_{\nu,k} t} \| h^{\rm in} \|_{L^2} \, , 
\end{equation}
where
\begin{equation}
    \label{eq:dampingcoeffoftheting}
    \lambda_{\nu,k} = \begin{cases}
    \frac{k^2}{\nu} & k \leq \nu \\
    \frac{\nu^{1/2}k^{1/2}}{1+\abs{\log (\nu/k)}} & k \geq \nu \, .
    \end{cases}
\end{equation}
The inverse dependence of $\lambda_{\nu,k}$ on $\nu$ is seen in low modes $k \leq \nu$. However, as is characteristic of Taylor dispersion, this inverse dependence is not ``seen" until the diffusive time $O(\nu^{-1})$, that is, once the $k \geq \nu$ modes have been damped. The particular $\nu^{1/2}$ dependence of $\lambda_{\nu,k}$ is sometimes known as \emph{enhanced dissipation}, and we explain it in the next section. 

At the nonlinear level, the dispersive effect\footnote{Although this is called Taylor dispersion, mathematically the operator $\bp \cdot \nabla_x - \nu \Delta_p$ would not typically be called dispersive. In low modes, it creates a diffusive (parabolic) operator.} allows us to prove



\begin{theorem}[Nonlinear stability of vacuum in $\R^3$]
\label{thm:vacuumstab}
Let $d=3$, $\nu \in (0,1]$, and $\kappa > 0$. Let $0 \leq \psi^{\rm in} \in L^1_x L^2_p \cap H^2_x L^2_p(\R^3 \times S^2)$. If $0 < \varepsilon_0 \ll \min (\nu, \kappa)$ and\footnote{Throughout the paper, $a\ll b$ is used in hypotheses to mean that there exists an absolute constant $m_0$ such that $a\le m_0 b$.}
\begin{equation}
    \varepsilon := \|  \psi^{\rm in}  \|_{(L^1_x \cap H^2_x) L^2_p} \leq \varepsilon_0 \, ,
\end{equation}
then the strong solution $\psi$ to \eqref{eq:smoluchowski}--\eqref{eq:activestress} on $\R^3$ exists globally in time and satisfies the decay estimate
\begin{equation}
    \| \psi(\cdot,t)  \|_{L^\infty_x L^2_p} + \|  \psi(\cdot,t)  \|_{(\dot H^{d/2} \cap \dot H^2)_x L^2_p} \les r(\nu t) \varepsilon \, ,
\end{equation}
where
\begin{equation}
    r(s) = \la s \ra^{-\frac{d}{2}} (\log (2+s))^2 \, .
\end{equation}
\end{theorem}

For completeness, we develop a basic theory of strong solutions in Appendix~\ref{app:strongsols}, see specifically Definition~\ref{def:strongsol} and Theorem~\ref{thm:strongsol}.


Note that Theorem~\ref{thm:vacuumstab} holds for both pullers and pushers near equilibrium, which is consistent with the observation of Saintillan--Shelley~\cite{saintillan2007orientational, saintillan2012emergence} in many-particle simulations that the diffusivity of individual swimmers is not dependent on the emergence of large-scale flows in the suspension. This may be contrasted with enhanced diffusion of tracer particles in swimmer suspensions, a well-studied phenomenon experimentally, which does rely on collective motion of the swimmers and is different from the effect studied here.

Our proof of the linearized estimates (see Theorem~\ref{thm:linTayDisp}) is based on a well-chosen energy estimate for the solution and two moments. It is inspired by~\cite[Section 3]{BedrossianWang2019}, which is itself inspired by work of Guo on collisional kinetic equations, e.g.,~\cite{guolandau}. These energy estimates capture the Taylor dispersion effect (though not enhanced dissipation for modes $k \geq \nu$) in the puller setting for \emph{arbitrary} $\overline{\psi} \geq 0$ (that is, the non-local term may be non-perturbative). As a consequence, we show
\begin{theorem}[Nonlinear stability of puller equilibrium on $\T^d$]
\label{thm:pullerstab}
Let $\nu \leq 1$, $\kappa > 0$, $\iota = +$, and $\overline{\psi} \geq 0$.  Let $f^{\rm in} \in H^2_x L^2_p(\T^d \times S^{d-1})$ with $\int f^{\rm in} \, d\bx \, d\bp = 0$. If $0 < \varepsilon_0 \ll \min (\nu, \kappa)$ and
\begin{equation}
    \varepsilon := \| f^{\rm in} \|_{H^2_x L^2_p} \leq \varepsilon_0 \, ,
\end{equation}
then the strong solution $\psi = \overline{\psi} + f$  to \eqref{eq:smoluchowski}--\eqref{eq:activestress} on $\T^d$ exists globally in time and the perturbation $f$ satisfies the decay estimate
\begin{equation}
    \| f(\cdot,t) \|_{H^2_x L^2_p} \les e^{-\delta_0 \nu t} \varepsilon 
\end{equation}
for some $\delta_0>0$.
\end{theorem}
Notably, there is no size restriction on $\overline{\psi}$. 

\subsubsection{Enhanced dissipation}

We now address the $k \geq \nu$ regime in~\eqref{eq:dampingoftheting}-\eqref{eq:dampingcoeffoftheting}. Define $\lambda_{\nu} = \lambda_{\nu,1}$ in \eqref{eq:dampingcoeffoftheting}. If we consider~\eqref{eq:thisisfordemonstration} on $\T^d$, then the minimal non-zero wavenumber is $k=1$, and we see that the modes with $k \geq 1$ are damped in the enhanced dissipation time $O(\lambda_{\nu}^{-1})$, logarithmically slower than $O(\nu^{-1/2})$, when $\nu \leq 1$. The decay of the $k = 0$ mode $f_0(\bp,t) := \barint_x f(\bx,\bp,t) \, d\bx$ is not enhanced, as it solves the heat equation:
\begin{equation}
    \p_t f_0 - \nu \Delta_p f_0 = 0 \, .
\end{equation}
To extract the enhanced dissipation timescale, we define $f_{\neq} := f - f_0$. 

For sufficiently small $\overline\psi$, we demonstrate that perturbations to the equilibrium $\overline{\psi}$ exhibit nonlinear enhanced dissipation in the case of both pullers and pushers. 


\begin{theorem}[Nonlinear enhanced dissipation for small concentrations]\label{thm:nonlin_enhance}
Let \\
$\nu,\kappa>0$, $0\le f^{\rm in}\in H^2_xL^2_p(\T^d \times S^{d-1})$, and  $\overline\psi\ll\lambda_\nu$. If $0<\varepsilon_0\ll \min(\kappa^{1/2},\nu^{1/2})\lambda_\nu^{1/2}$ and 
\begin{equation}
\varepsilon:=\norm{f_{\neq}^{\rm in}}_{H^2_xL^2_p} \le \varepsilon_0 
\quad \text{and} \quad
\norm{f^{\rm in}_0}_{L^2_p}\le \varepsilon_0 \,,
\end{equation}
then the nonzero modes of the strong solution $\psi = \overline{\psi} + f$ to \eqref{eq:smoluchowski}--\eqref{eq:activestress} on $\T^d$ satisfy the enhanced decay rate
\begin{equation}
\norm{f_{\neq}(\cdot,t)}_{H^2_xL^2_p} \lesssim e^{-\delta_{\neq} \lambda_\nu t}\varepsilon
\end{equation}
for some $\delta_{\neq}>0$. Furthermore, the zero mode satisfies the bound
\begin{equation}\label{eq:f0X0est0}
\norm{f_0}_{L^2_p} \lesssim e^{-\delta_0\nu t}\bigg(\norm{f^{\rm in}_0}_{L^2_p}+ \nu^{-1/2} \lambda_\nu^{-1/2} \varepsilon^2 \bigg) 
\end{equation}
for some $\delta_0>0$.
\end{theorem}

Note that the notion of ``small'' $\overline\psi$ here is larger than might be immediately expected for a perturbative result: $\lambda_\nu$ rather than $\nu$.

\bigskip

We highlight a useful heuristic regarding the enhancement timescale $\nu^{-1/2}$. 
Consider the PDE~\eqref{eq:thisisfordemonstration}
in dimension $d=2$ and, writing $\bp=\cos\theta\be_1+\sin\theta\be_2$, consider a plane wave solution (note the time rescaling)
\begin{equation}
f(\bx,\bp,t) = h(\theta, s) e^{i \bk \cdot \bx}, \quad \bk = k \be_1, \quad s=kt \, .
\end{equation}
 Then $h$ satisfies
\begin{equation}
    \p_s h + i \cos \theta h - \frac{\nu}{k} \p_\theta^2 h = 0 \, .
\end{equation}
Heuristically, the evolution $\p_s h + i \cos \theta h = 0$, which dominates for short times, creates large gradients in $\theta$. This can be seen from applying $\p_\theta$ to the explicit formula $h=e^{-i\cos\theta s} h^{\rm in}$. Once these gradients are sufficiently large, the evolution $\p_s h - \nu \p_\theta^2 h / k = 0$ smooths them; this is the source of the ``enhanced dissipation."

The enhancement is slowest near $\p_\theta\cos \theta=0$, where the ``shearing" caused by $i \cos \theta$ is slowest. We examine the critical point $\theta=0$ more closely. The characteristic scales at $\theta = 0$ are $\hat{\theta} = (\nu/k)^{1/4}$ and $\hat{s} = (\nu/k)^{-1/2}$, in the sense that if we ``zoom in" at $\theta = 0$ by writing $ H(\Theta,S) = e^{is} h(\theta,s)$, where $\Theta = \theta/\hat{\theta}$ and $S = s/\hat{s}$, we see the following leading order behavior in $\nu/k$:
\begin{equation}
\label{eq:shearlayerpde}
    \p_S H +   \frac{i}{2}\Theta^2 H - \p_\Theta^2 H \approx 0 \, .
\end{equation}

We note a strong analogy between \eqref{eq:shearlayerpde} and the evolution of a passive scalar in a two-dimensional shear flow $\bu = (b(y),0)$:
\begin{equation}
    \label{eq:passivescalarshaerflow}
    \p_t f + b(y) \p_x f - \nu \Delta_{x,y} f = 0 \text{ on } \T^2  \times \R_+ \, .
\end{equation}
 Around a critical point $y_0$ at which $b'(y_0) = \cdots = b^{(N)}(y_0) = 0$ but $b^{(N+1)}(y_0) \neq 0$, a \emph{shear layer} where the scalar $f$ is dissipated more slowly becomes readily visible in simulations; its characteristic length scale in $y$ is $\nu^{1/(N+3)}$, and its characteristic ``enhancement" timescale is $\nu^{-(N+1)/(N+3)}$, which is known to be sharp~\cite{drivascotizelati}. Our setting corresponds to $N=1$, which is akin to Poiseuille flow, and the PDE~\eqref{eq:shearlayerpde} is approximately satisfied inside our ``shear layer" at $\theta = 0$.
 
 \bigskip
 
 Finally, we mention briefly the techniques involved in Theorem~\ref{thm:nonlin_enhance}. First, in Theorem~\ref{thm:Phi}, we prove linear enhanced dissipation for $\overline{\psi} = 0$ via the hypocoercivity method (see, for example,~\cite{beckwaynebar,jacobmicheleshear,gallay2021enhanced} in the context of shear flows~\eqref{eq:passivescalarshaerflow}). The hypocoercivity method has the advantage of being elementary, although without adjustment it produces a logarithmic loss in the exponent $\lambda_\nu$ compared to $\nu^{1/2}$; this is technical in nature, and there are methods to remove it~\cite{kolmogorovhypocoercivitywei,dongyiweiresolvent,Albritton2022}.  
 To prove the nonlinear enhanced dissipation in Theorem~\ref{thm:nonlin_enhance}, we rely crucially on the structure of the nonlinear terms. On the whole, our approach is partially inspired by~\cite{michelepoiseuille}.

\subsection{Comparison with existing literature}

\subsubsection{Classical kinetic theory}


While the kinetic model has similarities with the Vlasov-Poisson equation~\eqref{eq:vlasovpoisson} and its collisional cousins, a key difference is that the phase mixing effect of the swimming term $\bp \cdot \nabla_x$ is not as strong as that of the free-streaming term $\bv \cdot \nabla_x$. For example, solutions $e^{-it\bv \cdot \nabla_x} f^{\rm in}$ on the torus decay exponentially to their mean-in-$\bx$, provided the initial data is analytic-in-$\bv$. Meanwhile, solutions $e^{-it\bp \cdot \nabla_x} f^{\rm in}$ decay only polynomially, as $t^{-\frac{d-1}{2}}$, which is not even time-integrable. This is the difference between non-stationary versus stationary phase, see Lemma~\ref{lem:stationaryphase}. Luckily, the particular structure of the nonlocal term allows us to improve the decay rate of $\nabla \bu$ by a factor of $t^{-1}$. The enhanced dissipation rates are different as well: $e^{-c\nu^{1/2} t}$ rather than $e^{-c\nu^{1/3} t}$ observed in, e.g.,~\cite{BedrossianWang2019}. Finally, while the nonlocal term $\nabla_x \bu : \bp \otimes \bp$ in the linearization might be compared to $\bm{E} \cdot \nabla_v \bar{f}$ in~\eqref{eq:vlasovpoisson}, there is no general theory to handle these nonlocal terms, and each is treated separately. 





\subsubsection{Complex fluids}
The role of swimming near the homogeneous isotropic equilibrium was considered by \v{S}kult{\'e}ty et al. in \cite{vskultety2020swimming}. The authors work directly with the BBGKY hierarchy associated to a stochastic many-particle system for which \eqref{eq:smoluchowski}-\eqref{eq:activestress} is essentially the formal mean-field limit. They calculate spatial and temporal correlations, indicative of collective behavior, of fluctuations around the equilibrium, sSection 2.D therein, and find that they are suppressed by swimming. This involves solving a Volterra equation similar to the one we obtain in Section~\ref{sec:landaudamping} below.

We now turn our attention to the PDE literature.
Although there is a vast PDE literature on related models without swimming (see the survey~\cite{masmoudi2018equations}), 
to the authors' knowledge, there is a single PDE work on the model \eqref{eq:smoluchowski}-\eqref{eq:activestress}. In \cite{chen2013global}, Chen and Liu establish the existence of global weak entropy solutions to~\eqref{eq:smoluchowski} coupled with the Navier-Stokes equations or the Stokes equations~\eqref{eq:stokes}-\eqref{eq:activestress} for the velocity. Notably, in dimension two, the weak solutions they construct in the Stokes setting are \emph{unique}. The regularity theory of~\eqref{eq:smoluchowski}-\eqref{eq:activestress} is complicated by the absence of a maximum principle for the density $\varrho$.

The system \eqref{eq:smoluchowski}-\eqref{eq:activestress} belongs to a broader class of micro-macro models commonly used to describe passive immersed polymers. Such models couple a microscopic description of the immersed particles (in the form of a Fokker-Planck equation which depends on a configuration variable $\bp$) to a macroscopic description of the suspension in the form of a forced Stokes or Navier-Stokes equation depending only on $\bx$. 

Different microscopic descriptions lead to two main types of models. The first treats the polymers as elastic dumbbells (two beads attached by a spring). We refer to \cite{la2020diffusive, masmoudi2018equations} for an overview of well-posedness results here. 
In the second class of models (sometimes called Doi-type), to which the system \eqref{eq:smoluchowski}-\eqref{eq:activestress} belongs, the immersed polymers are treated as rigid rods. The many existence and uniqueness results for this class of models include
\cite{BaeTrivisa, chen2014existence, constantin2007regularity, constantin2008global, constantin2010global, la2019global, lions2007global, masmoudi2008global, otto2008continuity, sieber2020existence, sun2011global, zhang2008new}. Note that no swimming is included in these previous results.

Finally, we mention two works which do incorporate swimming.
Jiang et al. \cite{jiang2020coupled} provide a proof of local well-posedness for a microscopic “self-organized kinetic” model coupled with Navier-Stokes and rigorously justify the hydrodynamic limit to a macroscopic closure model. Further related work on swimmers includes Kanzler and Schmeiser \cite{kanzler2021kinetic}, who consider a kinetic transport model for myxobacteria in which the particles interact via collisions rather than through a surrounding fluid medium. They show existence, uniqueness, and decay to equilibrium for sufficiently large particle diffusivity using hypocoercivity.

\subsection{Future directions} The model~\eqref{eq:smoluchowski}-\eqref{eq:activestress} is in its early stages of development from a rigorous PDE perspective, and many questions remain. Below, we focus only on questions concerning the near-equilibrium behavior:

\subsubsection{Nonlinear Landau damping}
A natural but challenging question is whether solutions exhibit \emph{nonlinear} Landau damping in the absence of dissipation. A comparison with previous works on nonlinear Landau damping in the Vlasov-Poisson equation~\cite{mouhotvillani,bedrossianmasmoudimouhot,grenier2020landau} and inviscid damping in two-dimensional Euler near the Couette flow~\cite{bedrossianmasmoudiinviscid} indicate the potential difficulty of this question, especially since arbitrarily fast polynomial decay to equilibrium is not anticipated.

\subsubsection{Precise description of the Taylor dispersion}
It would be desirable to capture more precise asymptotics of the Taylor dispersion rather than the upper bound we prove in Theorem~\ref{thm:linTayDisp}. We expect this to be possible by homogenization or, in a different vein, the center manifold approach in~\cite{beckchaudharywayne}. There may be additional difficulties in capturing this at the nonlinear level. Furthermore, it would be interesting to prove the correct analogue of Theorem~\ref{thm:vacuumstab} on $\R^2$. In this context, the term $\bu \cdot \nabla_x \psi$ is ``critical" and \emph{a priori} could modify the leading order asymptotics. 

\subsubsection{Stable-in-$\nu$ Landau damping and the viscous Penrose condition}
Our expectation is that the Landau damping and enhanced dissipation phenomena persist for $0 < \nu \ll 1$ provided that the Penrose condition is satisfied. More specifically, we expect that whenever $\iota = +$ or $\overline{\psi} < \overline{\psi}^*$, then for sufficiently small $\nu \ll_{\overline{\psi}} 1$, the linearized equations exhibit Landau damping and enhanced dissipation.\footnote{\emph{Added in proof}. The very recent work~\cite{zelati2022orientation} has now achieved enhanced dissipation in the above regime.} See~\cite{chaturvedi2021vlasov} and~\cite{bmv1} for analogous theorems in the Vlasov--Poisson--Landau and two-dimensional Navier-Stokes settings. In our setting, the key difficulty seems to be to perturb the Penrose condition to positive viscosity. For this purpose, one requires Landau damping estimates for the viscous operator $\p_t f + \bp \cdot \nabla_x f - \nu \Delta_p f$ which appear to be unknown. This corresponds to stable-in-$\nu$ mixing estimates for shear flows with critical points. 

\subsubsection{Zero translational diffusion} The $\bx$-diffusion plays no role in our linear arguments and, in the nonlinear arguments, is only used to control the term $\bu \cdot \nabla_x f$ semilinearly. We anticipate that much of our analysis can be generalized to $\kappa = 0$, in which case the model is quasilinear-in-$\bx$. One possible approach would be to incorporate a term $\overline{\bu} \cdot \nabla_x$ into the Guo and hypocoercivity schemes, where $\overline\bu$ is a known function.


\section{Landau damping}
\label{sec:landaudamping}

In this section, we prove the linear Landau damping result of Theorem~\ref{thm:landaudamping}. We consider the linearized inviscid equations
\eqref{eq:lin_inviscid1}-\eqref{eq:Sigma} and perform a mode-by-mode analysis in $\bx$.
Writing $f(\bx,\bp,t) = h(\bp,t) e^{i\bk \cdot \bx}$, $\bk \in \R^d \setminus \{0\}$, we have that $\bu = \widehat{\bu}_k e^{i \bk \cdot \bx}$, where
\begin{equation}
    \wh{\bu}_k =  i\abs{\bm{k}}^{-1} ({\bf I} - \overline{\bk} \otimes\overline{\bk})\wh{\bm{\Sigma}}_k \overline{\bm{k}}  \, .
\end{equation}
Here $\overline{\bk} = \bk/|\bk|$,  $\bm{\Sigma}=\wh{\bm{\Sigma}}_ke^{i\bk\cdot\bx}$, and
\begin{equation}
    \wh{\nabla \bu}_k = - ({\bf I} - \overline{\bk}\otimes\overline{\bk}) \wh{\bm{\Sigma}}_k \overline{\bk} \otimes \overline{\bk} \, .
\end{equation}
Upon rotating, we may take $\bk = k \be_1$, and upon rescaling time and replacing $\overline{\psi}$ by $\overline{\psi}/k$, we may consider $\bk = \be_1$ without loss of generality. With these simplifications, observe that $\wh{\nabla \bu}$ (we omit the subscript $k$) has the following structure:
\begin{equation}\label{eq:gradu_hat}
\widehat{\nabla \bu}=-({\bf I}-\be_1\otimes\be_1)\wh{\bm{\Sigma}}\,\be_1\otimes\be_1  = -\begin{pmatrix}
0 & 0 & \cdots & 0 \\
\Sigma_{21} & 0 & \cdots & 0 \\
\vdots & \vdots & \ddots & \vdots \\
\Sigma_{d1} & 0 & \cdots & 0  \\
\end{pmatrix} \, .
\end{equation}
In particular, the matrix $\widehat{\nabla \bu}$ consists of nonzero entries only in the first column, with zero first row. The structure of this matrix, particularly the zero first entry, will be very important for achieving the decay rate in Theorem \ref{thm:landaudamping}.

The proof of Theorem \ref{thm:landaudamping} thus reduces to studying decay properties of the following equation for $h(\bp,t)$:  
\begin{equation}
    \p_t h + i p_1 h - d \overline{\psi} \wh{\nabla \bu} : \bp \otimes \bp = 0 \, .
\end{equation}
We begin by quantifying the decay when $\overline{\psi} = 0$. We study the free transport equation
\begin{equation}
\p_t h + i p_1 h = 0 \, , \quad h(\cdot,0) = h^{\rm in} \, ,
\end{equation}
whose solution 
\begin{equation}
    \label{eq:hissolution}
    h = e^{-ip_1 t} h^{\rm in} 
\end{equation}
decays due to phase mixing, which is measured in a negative Sobolev norm.

\begin{lemma}[Stationary phase estimates]
\label{lem:stationaryphase}
(i) If $h^{\rm in} \in H^{d-1}(S^{d-1})$ and $h$ is defined by~\eqref{eq:hissolution}, then
\begin{equation}
    \| h(\cdot,t) \|_{H^{-(d-1)}} \les \la t \ra^{-\frac{d-1}{2}} \| h^{\rm in} \|_{H^{d-1}} \, .
\end{equation}
(ii) If moreover $h^{\rm in} \in H^{d+1}(S^{d-1})$ and $h^{\rm in}(\be_1) = h^{\rm in}(-\be_1) = 0$, then
\begin{equation}
    \| h(\cdot,t) \|_{H^{-(d+1)}} \les\la t \ra^{-\frac{d-1}{2}-1} \| h^{\rm in} \|_{H^{d+1}} \, .
\end{equation}
Here $\la t \ra = \sqrt{1+t^2}$ is the Japanese bracket.
\end{lemma}

Recall that 
\begin{equation}
    \norm{h}_{H^k}^2 := \norm{h}_{L^2}^2+\norm{\nabla^k h}_{L^2}^2
\end{equation}
and 
\begin{equation}
    \norm{h}_{H^{-k}} := \sup_{\norm{g}_{H^k}=1}\abs{\int hg} \, , \quad k\ge 0\,.
\end{equation}

\begin{proof}
Our goal is to estimate the oscillatory integral
\begin{equation}\label{eq:osc}
    \int_{S^{d-1}} e^{-ip_1t} h^{\rm in} \phi \, d\bp \, ,
\end{equation}
where $\phi \in H^{m}(S^{d-1})$ is arbitrary and $m \in \{ d-1,d+1 \}$ depending on whether we are proving (i) or (ii). Then $h^{\rm in} \phi \in W^{m,1}(S^{d-1})$ and
\begin{equation}
    \label{eq:improvedmuhdecay}
    \| h^{\rm in} \phi \|_{W^{m,1}(S^{d-1})} \les \| h^{\rm in} \|_{H^m(S^{d-1})} \| \phi \|_{H^m(S^{d-1})} \, .
\end{equation}
If $m=d+1$ and $h^{\rm in}$ vanishes at $\be_1$ and $-\be_1$, then the function $h^{\rm in} \phi$ also vanishes at $\be_1$ and $-\be_1$. The desired estimates follow from the method of stationary phase, which describes the $t$-asymptotics of oscillatory integrals of the type \eqref{eq:osc}; see, for example,~\cite[Chapter 8]{stein1993harmonic}.
\end{proof}


We now consider $\overline{\psi} \geq 0$ and seek a closed equation for $\wh{\nabla \bu}$.\footnote{It would also be natural to study a closed equation for the active stress $\bm{\Sigma}$. However, to see the improved decay, e.g., in~\eqref{eq:improvedmuhdecay}, we will need to take advantage of the structure of $\nabla\bu$. } 
This type of argument is familiar from kinetic theory; see \cite[Chapter 3]{villani2010landau} regarding the Vlasov equation. By Duhamel's formula, we have
\begin{equation}\label{eq:hDuhamel}
    h(\bp,t) = e^{-ip_1t} h^{\rm in}(\bp) + d\overline{\psi} \int_0^t e^{-ip_1(t-s)} \widehat{\nabla \bu}(s) : \bp\otimes\bp \, ds\, .
\end{equation}
Multiplying \eqref{eq:hDuhamel} by $\iota \bp\otimes\bp$ and integrating in $\bp$, by definition of the active stress $\bm{\Sigma}$ \eqref{eq:Sigma}, we obtain 
\begin{equation}\label{eq:Sigma_eqn}
    \wh{\bm{\Sigma}}[h] = \wh{\bm{\Sigma}}[e^{-ip_1t} h^{\rm in}] + \iota d\overline{\psi}\int_0^t \bp\otimes\bp \,e^{-ip_1(t-s)}\widehat{\nabla \bu}(s) : \bp\otimes\bp \, ds \, .
\end{equation} 
Using the definition of $\wh{\nabla\bu}$ \eqref{eq:gradu_hat}, we may then multiply \eqref{eq:Sigma_eqn} on the left by $({\bf I}-\be_1\otimes\be_1)$ and on the right by $\be_1\otimes\be_1$ to obtain a Volterra equation for $\wh{\nabla\bu}$: 
\begin{equation}\label{eq:volterra}
\widehat{\nabla \bu}[h] = \widehat{\nabla \bu}[e^{-ip_1t}h^{\rm in}] - \iota d\overline{\psi}\int_0^t K(t-s)\, \wh{\nabla\bu}[h](s)\, ds \, .
\end{equation}
Here, $K(t)$ is an operator-valued kernel which acts on the tensor $\widehat{\nabla \bu}$ by
\begin{equation}\label{eq:Kgradu}
    K(t)\widehat{\nabla \bu} = ({\bf I}-\be_1\otimes\be_1)\bigg(\int_{S^{d-1}} \bp\otimes\bp \, e^{-ip_1 t} \, \widehat{\nabla \bu} : \bp\otimes\bp \, d\bp \bigg)\, \be_1\otimes\be_1 \, ,
\end{equation}
that is, for $j \geq 2$,
\begin{equation}\label{eq:Kgradu2}
    (K(t)\widehat{\nabla \bu})_{j1} = -\int_{S^{d-1}} p_1 p_j e^{-ip_1 t} \, \widehat{\nabla \bu} : \bp\otimes\bp \, d\bp \, ,
\end{equation}
and vanishes in all other components.

Notice that, due to the special structure \eqref{eq:gradu_hat} of $\wh{\nabla\bu}$, we may apply Lemma~\ref{lem:stationaryphase}, estimate (ii), to obtain the following decay estimates:
\begin{align}
    | \wh{\nabla \bu}[e^{-ip_1t} h^{\rm in}] | &\les \la t \ra^{-\frac{d-1}{2}-1} \| h^{\rm in} \|_{H^{d+1}(S^{d-1})} \, ,  \label{eq:whdatadecay}\\ 
    |K(t)\widehat{\nabla \bu} | &\les\la t \ra^{-\frac{d-1}{2}-1}  | \widehat{\nabla \bu} | \, . \label{eq:Kdecay}
\end{align}

To solve the Volterra equation \eqref{eq:volterra} for $\wh{\nabla\bu}$, as is standard, we begin by taking the Fourier-Laplace transform in time, which we denote by
\begin{equation}\label{eq:LapTrans}
    \cL[g](\lambda) = \int_0^{+\infty} e^{-\lambda t} g(t) \, dt \, , \qquad \lambda\in\C \, .
\end{equation}
\emph{A priori}, $\cL\wh{\nabla\bu}$ is only well defined for $\Re(\lambda) \gg \overline{\psi}$ when $\iota = -$, although it is automatically well defined for $\Re \lambda > 0$ when $\iota = +$. This can be seen from the energy estimate~\eqref{eq:linearizedhtheorem}. By the decay estimates~\eqref{eq:whdatadecay} and~\eqref{eq:Kdecay}, $\wh{\nabla \bu}[e^{-ip_1t} h^{\rm in}]$ and $K$ are time-integrable, and their Fourier-Laplace transforms are well defined and continuous for $\Re \lambda \geq 0$. We have
\begin{equation}
    \cL \wh{\nabla\bu}[h] = \cL \wh{\nabla\bu}[e^{-ip_1t} h^{\rm in}] - \iota d\overline{\psi}(\cL K)\cL\wh{\nabla\bu}[h] \, ,
\end{equation}
which formally can be solved for $\cL \wh{\nabla \bu}$:
\begin{equation}\label{eq:LUhat}
    \cL\wh{\nabla\bu}[h] = (I + \iota d\overline{\psi} \cL K)^{-1} \cL \wh{\nabla\bu}[e^{-ip_1t} h^{\rm in}] \, ,
\end{equation}
provided that $I + \iota d \overline{\psi} \cL K$ is invertible.
Note that the inverse in \eqref{eq:LUhat} is not a matrix inverse, since in this case $I$ and $\cL K$ are really linear operators acting on tensors of the form~\eqref{eq:gradu_hat} via \eqref{eq:Kgradu}.
As long as $(I + \iota d\overline{\psi} \cL K)^{-1}$ is finite, we can solve for $\cL\wh{\nabla\bu}$. We characterize the invertibility of $I + \iota d\overline{\psi} \cL K$ in the following lemma.

\begin{lemma}[Penrose condition]\label{lem:penrose}
The \emph{Penrose condition} 
\begin{equation}\label{eq:penrose}
    \sup_{\Re \lambda \geq 0} \| (I + \iota d\overline{\psi} \cL K)^{-1} \| \leq C < +\infty \, 
\end{equation}
is satisfied unconditionally in the case of pullers ($\iota=+$) with constant $C$ independent of $\overline{\psi}$.  For pushers ($\iota=-$), the condition \eqref{eq:penrose} is equivalent to 
\begin{equation}\label{eq:disp_rel}
\iota d\overline{\psi} \int_{S^{d-1}} \frac{p_1^2p_j^2}{\lambda+ip_1} \, d\bp =1
\end{equation}
having no solution for $\Re(\lambda)\ge 0$.
\end{lemma}

Let $\overline{\psi}^*$ denote the supremum of $\psi$ such that the Penrose condition is satisfied for all $\overline{\psi} < \psi$. We will see in the course of the proof that indeed $\overline\psi^*>0$ is well defined.

The Fourier-Laplace transforms of $\cL K$ and $\cL \wh{\nabla \bu}[e^{-ip_1\cdot} h^{\rm in}]$ are analytic in $\Re \lambda > 0$ and continuous in $\Re \lambda \geq 0$. When the Penrose condition is satisfied, the right-hand side of~\eqref{eq:LUhat} is the unique analytic continuation of $\cL \wh{\nabla \bu}$ into $\Re \lambda > 0$. Using the notation $\lambda=\sigma+i\tau$, this analytic continuation belongs to $L^\infty_\sigma L^2_\tau(\{ \sigma > 0 \})$ and is known to attain its boundary data continuously. Hence, the Paley-Wiener theorem (see~\cite[Theorem 19.2, p. 372]{RudinRealComplex}) guarantees that $\cL \wh{\nabla \bu}$ is the Fourier-Laplace transform of an $L^2(\R_+)$ function, which must be $\wh{\nabla \bu}$ by uniqueness of the Fourier-Laplace transform. See also the argument in \cite[p. 41]{villani2010landau}.

We emphasize that if the integral equation \eqref{eq:disp_rel} has no solution (that is, the operator $I+\iota d\overline{\psi} \mathcal{L}K$ is invertible pointwise), then the Penrose condition \eqref{eq:penrose} is satisfied for some constant $C$ (that is, the operator is invertible uniformly). This is because $\cL K$ is continuous and decays as $|\lambda| \to +\infty$ with $\Re \lambda \geq 0$.

\begin{remark}
Note that the integral equation \eqref{eq:disp_rel} is precisely the dispersion relation arising in the eigenvalue problem for the linearized operator \eqref{eq:lin_inviscid1}, which has been studied in detail by various authors \cite{hohenegger2010stability,ohm2022weakly,saintillan2008instabilitiesPRL,saintillan2008instabilitiesPOF,saintillan2013active,subramanian2009critical, subramanian2011stability}. In particular, any $\lambda$ satisfying \eqref{eq:disp_rel} is an eigenvalue of the linearized operator \eqref{eq:lin_inviscid1}. In the pusher case ($\iota=-$), the Penrose condition reduces to the dispersion relation having no solution for $\Re(\lambda)\ge 0$, i.e. the linearized operator has no unstable or marginally stable eigenvalue.

The implicit dispersion relation \eqref{eq:disp_rel} can be solved numerically for $\lambda$ as a function of the parameter $\overline{\psi}$ (see \cite{hohenegger2010stability,saintillan2008instabilitiesPOF,saintillan2008instabilitiesPRL,saintillan2013active}; note that a different nondimensionalization from \eqref{eq:nondim} is commonly used) from which we observe that \eqref{eq:disp_rel} has a solution with $\Re(\lambda)\ge0$ for $\overline{\psi}\ge \overline{\psi}^*$ for some $\overline{\psi}^*$. In this situation the linearized equations \eqref{eq:lin_inviscid1}-\eqref{eq:lin_inviscid2} have growing modes which give rise to pattern formation and ``bacterial turbulence" observed in numerical simulations \cite{ohm2022weakly,saintillan2008instabilitiesPOF,saintillan2012emergence,saintillan2013active,saintillan2015theory}. 

To satisfy the Penrose condition \eqref{eq:penrose} in the case of pushers, $\overline{\psi}$ must therefore be taken to be sufficiently small. The threshold value of $\overline{\psi}^*$ exactly corresponds to the eigenvalue crossing observed in numerical studies \cite{hohenegger2010stability,ohm2022weakly,saintillan2008instabilitiesPRL,saintillan2008instabilitiesPOF,saintillan2013active,subramanian2009critical}.
\end{remark}

\begin{proof}[Proof of Lemma \ref{lem:penrose} (Penrose condition)]
 To see how $\cL K$ acts on tensors $\bm{\Xi}$ of the form \eqref{eq:gradu_hat}, we first note that since $\bm{\Xi}$ is only nonzero in rows 2 through $d$ of the first column, we have
\begin{equation}
\bm{\Xi}: \bp\otimes\bp = \sum_{\ell=2}^d \Xi_{\ell1}p_\ell p_1 \, .
\end{equation}
Therefore, we may write $\cL K\bm{\Xi}$ as 
\begin{equation}\label{eq:LKXi}
\begin{aligned}
 \cL K\bm{\Xi}(\lambda) &= \int_0^\infty e^{-\lambda t}\int_{S^{d-1}} e^{-ip_1 t} \, \bm{M}(\bm{\Xi},\bp)  \, d\bp \, dt \\
 &= \int_{S^{d-1}} \frac{1}{\lambda+ip_1} \, \bm{M}(\bm{\Xi},\bp)  \, d\bp ,
 \end{aligned}
\end{equation}
where $\bm{M}(\bm{\Xi},\bp)$ is again a matrix of the form \eqref{eq:gradu_hat}, 
\begin{equation}
\bm{M}(\bm{\Xi},\bp) = \begin{pmatrix}
0 & 0 & \cdots & 0 \\
M_{21} & 0 & \cdots & 0 \\
\vdots & \vdots & \ddots & \vdots \\
M_{d1} & 0 & \cdots & 0  \\
\end{pmatrix},
\end{equation}
with 
\begin{equation}
M_{j1} = p_1^2p_j\sum_{\ell=2}^d \Xi_{\ell1}p_\ell.
\end{equation}

Since odd functions in $p_j$ integrate to zero on $S^{d-1}$ for each $j=1,\dots,d$, the integral of each $M_{j1}$ in \eqref{eq:LKXi} is only nonzero when $\ell=j$. Thus $\cL K$ acts on $\bm{\Xi}$ via entrywise multiplication: 
\begin{equation}
 (\cL K\Xi)_{j1} =  \gamma_j(\lambda)\Xi_{j1}, \qquad j\neq 1, 
\end{equation}
where 
\begin{equation}\label{eq:familiar_int}
\gamma_j(\lambda) = \int_{S^{d-1}} \frac{p_1^2p_j^2}{\lambda+ip_1} \, d\bp.
\end{equation}
In particular, we have that the operator $I + \iota d\overline{\psi} \cL K$ sends
\begin{equation}
I + \iota d\overline{\psi} \cL K : \;\; \Xi_{jm} \mapsto\begin{cases}
(1+\iota d\overline{\psi} \gamma_j)\Xi_{j1} & \text{if }m=1 \text{ and } j\neq 1 \\
\Xi_{jm} & \text{else}.
\end{cases}
\end{equation}
The invertibility of $I + \iota d\overline{\psi} \cL K$ \eqref{eq:penrose} thus comes down to verifying that 
\begin{equation}\label{eq:penrose2}
\iota d\overline{\psi} \gamma_j(\lambda)\neq -1
\end{equation}
 for all $\Re(\lambda)\ge0$. We must therefore analyze the integral $\gamma_j(\lambda)$ in greater detail.  Writing $\lambda=\sigma+i\tau$, we consider the behavior of $\Im(\gamma_j)$ and $\Re(\gamma_j)$ for $\sigma\ge 0$.

For the imaginary part, we need only to verify that $\Im(\gamma_j)$ is finite and uniformly bounded. This follows from the integrability of $K$. 
For $\sigma>0$, the real part of $\gamma_j$ is given by
\begin{equation}\label{eq:regamma}
\Re(\gamma_j) = \int_{S^{d-1}}\frac{\sigma}{\sigma^2+(\tau+p_1)^2}p_1^2p_j^2\, d\bp,
\end{equation}
which is clearly finite and nonnegative. 

In particular, $\gamma_j(\lambda)$ is well-defined with $\Re(\gamma_j)\ge 0$ for all $\lambda$ with $\Re(\lambda)\ge0$. In the case of pullers ($\iota=+$), we thus have that \eqref{eq:penrose2} is satisfied for any value of $\overline{\psi}$.
For pushers ($\iota=-$), on the other hand, in order to ensure that \eqref{eq:penrose2} is satisfied, $\overline{\psi}$ must be chosen such that
\begin{equation}
d\overline{\psi}\gamma_j(\lambda) \neq 1
\end{equation}
for any $\lambda$ with $\Re(\lambda)\ge0$. This is the Penrose condition.
\end{proof}

\begin{proof}[Proof of Theorem~\ref{thm:landaudamping} (Linear Landau damping)]
We return to our mode-by-mode analysis.
To begin, we observe that due to the time decay estimates~\eqref{eq:whdatadecay} and~\eqref{eq:Kdecay} of $\wh{\nabla \bu}[e^{-ip_1t}]$ and $K$, respectively, we have the following regularity estimates on the Fourier-Laplace transforms:
\begin{equation}
    \| \cL \wh{\nabla \bu}[e^{-ip_1\cdot} h^{\rm in}](\sigma + i\tau) \|_{H^s_\tau} \les_s \| h^{\rm in} \|_{H^{d+1}_p} \, ,
\end{equation}
\begin{equation}
    \label{eq:regularityofLk}
    \| \cL K(\sigma + i\tau) \|_{H^s_\tau} \les_{s} 1 \, ,
\end{equation}
for all $\sigma \geq 0$ and $s \in [0,d/2)$. Consider the function
\begin{equation}
    F(m) = (1+\iota d \overline{\psi} m)^{-1} - 1 \, .
\end{equation}
Since $\iota = +$ or $\overline{\psi} < \overline{\psi}^*$, the Penrose condition implies that $F \circ \gamma_j$ is bounded. Since $F$ is smooth on an open neighborhood of the range of $\gamma_j$ and $F(0) = 0$, and in light of the regularity estimate~\eqref{eq:regularityofLk}, the composition $F \circ \gamma_j(\sigma+i\tau)$ belongs to $H^s_\tau$ for all $s \in [0,d/2)$. See, for example, Theorem 2.87 in~\cite{bcd}. Since $H^s_\tau$ is an algebra for $s > 1/2$, we have
\begin{equation}
    \| (I + \iota d\overline{\psi} \cL K)^{-1} \cL \wh{\nabla\bu}[e^{-ip_1\cdot} h^{\rm in}] \|_{H^s_\tau}  \les_{s,\overline{\psi}} \| h^{\rm in} \|_{H^{d+1}_p} \, ,
\end{equation}
for all $s \in (1/2,d/2)$. Finally, the equality~\eqref{eq:LUhat} and Fourier inversion on $\sigma = 0$  yield
\begin{equation}
    \int | \wh{\nabla \bu} |^2 \la t \ra^{2s} \, dt \les_{s,\overline{\psi}} \| h^{\rm in} \|_{H^{d+1}_p}^2 \, ,
\end{equation}
for all $s \in (1/2,d/2)$. This validates~\eqref{eq:nableudecaythm}. To complete the theorem, we must justify~\eqref{eq:gdecaythm}, namely, that
\begin{equation}\label{eq:gdef}
    g(\bp,t) := d \overline{\psi} \int_0^t e^{-ip_1(t-s)} \wh{\nabla \bu} : \bp \otimes \bp \, ds
\end{equation}
satisfies
\begin{equation}
    \int \| g(\cdot,t) \|_{H^{-(d+1)}_p}^2 \la t \ra^{d-\varepsilon} \, dt \les_{s,\overline{\psi}} \| h^{\rm in} \|_{H^{d+1}_p}^2 \, .
\end{equation}
Again, we may take the Fourier-Laplace transform, this time of \eqref{eq:gdef}, and observe that $e^{-ip_1t} p_1p_j$, $j \geq 2$, has the same time decay as the kernel $K$. Since $H^s_\tau$ is a multiplicative algebra, the proof follows. 
\end{proof}


\begin{remark}[Pointwise-in-time decay]
The representation formula~\eqref{eq:LUhat} can be rewritten as
\begin{equation}
    \cL \wh{\nabla \bu} = \cL \wh{\nabla \bu}[e^{-ip_1 \cdot} h^{\rm in}] + (I + \iota d\overline\psi\cL K)^{-1} (-\iota d\overline\psi\cL K) \cL \wh{\nabla \bu}[e^{-ip_1 \cdot} h^{\rm in}] \, .
\end{equation}
In principle, one can obtain pointwise-in-time decay estimates on $\wh{\nabla \bu}$ by studying the Green's function $\cL^{-1} [(I + \iota d\overline\psi\cL K)^{-1} (-\iota d\overline\psi\cL K)]$ pointwise-in-time. A natural way to proceed is to study the optimal regularity of $\cL K(i\tau)$ in $L^1_\tau$-based spaces rather than $H^s_\tau$. One can obtain an explicit formula for $\cL K(i\tau)$ by sending $\sigma \to 0^+$. We leave its analysis to future work. 

\end{remark}

\section{Taylor dispersion}
In this section, we consider stability due to Taylor dispersion near vacuum $\overline{\psi} = 0$ in the whole space $\R^3$. Our linearized estimates will also allow us to prove stability of the puller equilibrium $\overline{\psi} \equiv \text{const.}$ on the torus $\T^d$, $d=2,3$. Notably, when $\iota = +$, stability holds regardless of the size of $\overline{\psi}$.

To begin, we study the linearized equation
\begin{equation}
    \label{eq:linearizedeq}
    \p_t  f + \bp \cdot \nabla_x f - \nu \Delta_p f - d\overline{\psi} \nabla 
    \bu : \bp \otimes \bp = 0 \, ,
\end{equation}
whose basic energy estimate is
\begin{equation}
    \label{eq:basicenergyestimate}
    \frac{1}{2} \frac{d}{dt} \| f \|_{L^2}^2 + d \iota \overline{\psi} \| \nabla \bu \|_{L^2}^2 = - \nu \| \nabla_p f \|_{L^2}^2 \, .
\end{equation}

Notice that the right-hand side of \eqref{eq:basicenergyestimate} only controls $f-\barint_{S^{d-1}}f\,d\bp$, and this by itself is not enough to prove exponential decay. In view of this, we consider the density (or concentration) $\varrho$ and momentum (or the nematic order parameter times $\varrho$) $\bm{m}$,
\begin{equation}
    \varrho(\bx,t) = \barint_{S^{d-1}} f(\bx,\bp,t) \, dp \, , \quad \bm{m}(\bx,t) = \barint_{S^{d-1}}\bp f(\bx,\bp,t)  \, d\bp \, .
\end{equation}
Here and in what follows, integration is with respect to the normalized measure on the sphere ($\barint_{S^{d-1}}\, d\bp=1$). The density and momentum satisfy the PDEs
\begin{equation}
    \label{eq:densityeq}
    \p_t \varrho + \div_x \bm{m} = 0 \, ,
\end{equation}
\begin{equation}
    \label{eq:momentumeq}
    \p_t \bm{m} + \div_x \left( \barint_{S^{d-1}} \bp\otimes \bp f\, d\bp \right) + \nu \lambda \bm{m} = 0 \, ,
\end{equation}
where $\lambda$ is the first nonzero eigenvalue of the Laplacian on the unit sphere.
To derive~\eqref{eq:densityeq}, we integrate~\eqref{eq:linearizedeq} in $\bp$ and observe that $\nabla \bu : \barint_{S^{d-1}} \bp \otimes \bp \, d\bp = d^{-1}\, {\rm tr} \; \nabla \bu = d^{-1} \, \div \bu = 0$ in the non-local term. To derive~\eqref{eq:momentumeq}, we integrate $\bp$ times \eqref{eq:linearizedeq} in $\bp$ and observe that $\barint_{S^{d-1}} \bp \Delta_p f \, d\bp = \lambda \barint_{S^{d-1}} \bp f \, d\bp = \lambda \bm{m}$ in the dissipation term and $(\nabla \bu : \bp \otimes \bp) \bp$ is odd-in-$\bp$ in the nonlocal term. Notice that~\eqref{eq:densityeq} does not contain a damping term for the density, whereas~\eqref{eq:momentumeq} already contains a damping term for the momentum. 

As mentioned in the introduction, our method is inspired by~\cite[Section 3]{BedrossianWang2019}, which itself is inspired by work of Guo. We introduce the macro-micro decomposition
\begin{equation}
    \label{eq:macromicro}
    f = \varrho + g \, .
\end{equation}
Assume that $\int_x \varrho = \iint_{x,p} f = 0$. 
Notice that
\begin{equation}
    \label{eq:mredef}
    \bm{m} = \barint_{S^{d-1}} \bp g \, d\bp \, .
\end{equation}
With the above decomposition~\eqref{eq:macromicro}, we have the following refinement of the flux in~\eqref{eq:momentumeq}:
\begin{equation}
    \barint_{S^{d-1}} \bp \otimes \bp f\, d\bp = \barint_{S^{d-1}} \bp \otimes \bp (\varrho + g) \, d\bp = \frac{1}{d} \varrho {\bf I}  + \barint_{S^{d-1}} \bp \otimes \bp g \, d\bp \, .
\end{equation}
 Thus,~\eqref{eq:momentumeq} can be rewritten as
\begin{equation}
    \label{eq:momentumeq2}
    \p_t \bm{m} + \frac{1}{d} \nabla_x \varrho + \div_x \left( \barint_{S^{d-1}} \bp \otimes \bp g \, d\bp \right) + \nu \lambda \bm{m} = 0 \, .
\end{equation}
Crucially,~\eqref{eq:momentumeq2} will produce the desired damping in $\varrho$ when integrated against $|\nabla_x|^{-2} \nabla_x \varrho$. Below, we refine the above reasoning to demonstrate

\begin{theorem}[Linear Taylor dispersion]\label{thm:linTayDisp}
Let $\nu > 0$, $\bk \in \R^d \setminus \{ 0 \}$, and  $k = |\bk|$. Suppose that $f = h(\bp,t) e^{i\bk \cdot \bx}$ solves the linearized PDE~\eqref{eq:linearizedeq} and $f^{\rm in} = h^{\rm in}(\bp) e^{i\bk \cdot \bx}$. There exists an absolute constant $c_0 > 0$ satisfying the following property: If $\iota = +$ or $\overline{\psi} \ll \mu_{\nu,k}$, then
\begin{equation}
    \| h(\cdot,t) \|_{L^2_p} \les e^{-c_0 \mu_{\nu,k} t} \| h^{\rm in} \|_{L^2_p} \, , \quad \forall t > 0 \, ,
\end{equation}
where
\begin{equation}
    \mu_{\nu,k} = \begin{cases}
    \nu & k \geq \nu \\
    \frac{k^2}{\nu} & k \leq \nu
    \end{cases} \, .
\end{equation}
\end{theorem}

\begin{proof}
Upon rotating, we may assume that $\bk = k \be_1$ without loss of generality. Furthermore, upon rescaling time and replacing $\nu$ and $\overline{\psi}$ by $\nu/k$ and $\overline{\psi}/k$, respectively, we may assume $k=1$. Thus, we consider~\eqref{eq:linearizedeq} on $\T^d \times S^{d-1} \times \R_+$ below.


Define
\begin{equation}
    G = \la |\nabla_x|^{-2} \nabla_x \varrho, \bm{m} \ra \, ,
\end{equation}
where the brackets denote the $L^2$ inner product.
Using \eqref{eq:densityeq} and \eqref{eq:momentumeq2}, we compute
\begin{equation}
    \label{eq:twistedestimate}
\begin{aligned}
    \p_t G &= - \la |\nabla_x|^{-2} \nabla_x \div_x \bm{m}, \bm{m} \ra \\
    &\qquad - \la |\nabla_x|^{-2} \nabla_x \varrho, \frac{1}{d} \nabla_x \varrho + \div_x  \barint_{S^{d-1}} \bp \otimes \bp g \, d\bp + \nu \lambda \bm{m}  \ra \\
    &= \| \, |\nabla_x|^{-1} \div_x \bm{m} \|_{L^2}^2 - \frac{1}{d} \| \varrho \|_{L^2}^2 \\
    &\qquad - \la |\nabla_x|^{-1} \nabla_x \varrho, |\nabla_x|^{-1} \div_x  \barint_{S^{d-1}} \bp \otimes \bp g \, d\bp  \ra - \nu \lambda \la |\nabla_x|^{-2} \nabla_x \varrho, \bm{m} \ra  \\
    &\leq C_d \| g \|_{L^2}^2 - \frac{1}{2d} \| \varrho \|_{L^2}^2 - \nu \lambda G   \, ,
    \end{aligned}
\end{equation}
where we recall that $\bm{m}$ is estimated by $g$, see~\eqref{eq:mredef}.

Define
\begin{equation}
    \Phi = \frac{1}{2} \| f \|_{L^2}^2 + \varepsilon G \, .
\end{equation}
Since $G$ is not inherently sign-definite, we require $\varepsilon \ll 1$ to ensure
\begin{equation}
    C^{-1} \| f \|_{L^2}^2 \leq \Phi \leq C \| f \|_{L^2}^2 \, .
\end{equation}
We seek a differential inequality for $\Phi$ by summing~\eqref{eq:basicenergyestimate} and $\varepsilon$ times~\eqref{eq:twistedestimate}:
\begin{equation}
    \frac{d\Phi}{dt} \leq - d \iota \overline{\psi} \| \nabla \bu \|_{L^2}^2 - \nu c_d \| g \|_{L^2}^2 - \frac{\varepsilon}{2d} \| \varrho \|_{L^2}^2 - \varepsilon \nu \lambda G + C_d \varepsilon \| g \|_{L^2}^2 \, ,
\end{equation}
 where we employ the Poincar{\'e} inequality $c_d \| g \|_{L^2}^2 \leq \| \nabla_p f \|_{L^2}^2$ in~\eqref{eq:basicenergyestimate}. We require $\varepsilon \ll \nu$ to ensure that $C_d \varepsilon \| g \|_{L^2}^2$ may be absorbed:
 \begin{equation}
    \frac{d\Phi}{dt} \leq - d \iota \overline{\psi} \| \nabla \bu \|_{L^2}^2 - \frac{\nu c_d}{2} \| g \|_{L^2}^2 - \frac{\varepsilon}{2d} \| \varrho \|_{L^2}^2 - \varepsilon \nu \lambda G \, .
\end{equation}
The above inequality has $O(\nu)$ damping in $g$ and $O(\varepsilon)$ damping in $\varrho$.

Due to the rescaling of $\nu$ by $k$, we have two cases to consider, corresponding to long and short wavelengths, respectively: 

\emph{Case 1. $\nu \leq 1$.} In this setting, the requirement $\varepsilon \ll \nu$ is enough to close an estimate for $\Phi$, since $\varepsilon \nu \lambda |G| \ll \nu^2 \| f \|_{L^2}^2$, and hence this term may be absorbed into the $O(\nu)$ damping in $f$ and $\varrho$.

\emph{Case 2. $\nu \geq 1$.} In this setting, we further require $\varepsilon \ll 1/\nu$, say, $\varepsilon = \delta/\nu$. Then, using Young's inequality,
\begin{equation}
    \varepsilon \nu \lambda |G| = \delta \lambda |G| \leq \eta \nu \| g \|_{L^2}^2 + C \eta^{-1} \delta^2 \nu^{-1} \| \varrho \|_{L^2}^2 \, .
\end{equation}
We choose $\eta \ll 1$ and $\delta \ll C^{-1} \eta$ to ensure that the above term may be absorbed into the $O(\nu)$ damping in $f$ and $O(\delta/\nu)$ damping in $\varrho$.

In either case, we have
 \begin{equation}
    \frac{d\Phi}{dt} \leq - d \iota \overline{\psi} \| \nabla \bu \|_{L^2}^2 - 3c_0 \mu_{\nu,1} \Phi 
\end{equation}
for some $c_0>0$. When $\iota = +$, the first term on the right-hand side is non-positive. When $\iota = -$, we require $\overline{\psi} \ll \mu_{\nu,1}$ to absorb this term. The proof is complete. \end{proof}
 
We now define the linear operator 
\begin{equation}\label{eq:Lnukappa}
    -L_{\nu,\kappa}^{\overline\psi,\iota} f := \bp \cdot \nabla_x f - d\overline{\psi} \nabla \bu[f] : \bp \otimes \bp - \nu \Delta_p f - \kappa \Delta_x f 
\end{equation}
and its associated semigroup 
\begin{equation}\label{eq:semigrp}
    S(t) = e^{tL_{\nu,\kappa}^{\overline\psi,\iota}}\,.
\end{equation}
 Note that since Fourier multipliers in $\bx$ commute, we have 
 \begin{equation}\label{eq:add_in_kappa}
     e^{t\kappa\Delta_x}e^{t L_{\nu,0}^{\overline\psi,\iota}} = e^{t L_{\nu,0}^{\overline\psi,\iota}} e^{t\kappa\Delta_x} = S(t)\, .
 \end{equation}
 We show the following:
 
 \begin{corollary}[Smoothing estimates]
    \label{cor:smoothingestf}
    Let $T>0$. Suppose that $f = h(\bp,t)e^{i \bk \cdot \bx}$ solves 
 \begin{equation}
     \p_t f - L_{\nu,\kappa}^{\overline\psi,\iota} f = \div_p \bm{g} 
 \end{equation}
 with $f^{\rm in} = 0$ for some $\bm{g} = \hat{\bm{g}}_k(\bp,t) e^{i\bk \cdot \bx}$ for $t \in (0,T)$. Suppose that $\iota = +$.
  Then
  \begin{equation}
      \| h(\cdot,T) \|_{L^2_p}^2 \les \nu^{-1} \int_0^T e^{-2c_0 \mu_{\nu,k} (T-s)} \| \hat{\bm{g}}_k \|_{L^2_p}^2 \, ds \, .
  \end{equation}
 \end{corollary}
 
 \begin{proof}
 Again, without loss of generality, $\bk = k \bm{e}_1$ with $k = |\bk|$, and we consider the PDEs on the torus. Observe that for mode-$\bk$ solutions of the initial-value problem,
    \begin{equation}
        \nu \int_0^{t} e^{2 c_0\mu_{\nu,k} s} \int_{x,p}  |\nabla_p S(s) f^{\rm in}|^2  \, d\bx \, d\bp \, ds \les \| f^{\rm in} \|_{L^2}^2 \, ,
    \end{equation}
    which is a simple corollary of the energy estimates in Theorem~\ref{thm:linTayDisp}. Consider the solution operator $f^{\rm in} \mapsto \nabla_p S(\cdot) f^{\rm in} : L^2 \to L^2_w(\T^d \times S^{d-1} \times (0,T))$, where $w$ refers to the weight $e^{c_0 \mu_{\nu,k} t}$ in the above estimate, and we restrict to mode-$\bk$ functions. The adjoint operator $L^2_{w^{-1}}(\T^d \times S^{d-1} \times (0,T)) \to L^2$ produces the $t'=0$ trace of the solution of the PDE
    \begin{equation}
        - \p_{t'} \wt f - \bp \cdot \nabla_x \wt f - d\overline{\psi} \nabla \bu[\wt f] : \bp \otimes \bp - \nu \Delta_p \wt f - \kappa \Delta_x \wt f = \div_p \bm{g}
    \end{equation}
    backward-in-time with $\wt f(\cdot,T) = 0$ and a given $\bm{g}$ in mode $\bk$. Its solution satisfies
    \begin{equation}
        \| \wt f(\cdot,0) \|_{L^2}^2 \les \nu^{-1} \int_0^{T} e^{-2 c_0 \mu_{\nu,k} s} \int_{x,p}  |\bm{g}|^2 \, d\bx \, d\bp \, ds \, .
    \end{equation}
    The desired estimate  is obtained by rewriting $t = T-t'$ and reflecting $\bp \to -\bp$. 
 \end{proof}

\begin{proof}[Proof of Theorem \ref{thm:pullerstab} (Puller stability)]
 We consider the PDE for the perturbation $f$, namely,
\begin{equation}
    \p_t f - L_{\nu,\kappa}^{\bar\psi,\iota}f +\bu\cdot\nabla_x f + \div_p [(\bI - \bp \otimes \bp) \nabla \bu \bp f ] = 0 \, . 
\end{equation}
Since $\iota=+$, we omit this from the notation. 
By Duhamel's formula, we have
\begin{equation}
    \label{eq:duhamelsformula}
    f(\cdot,t) = S(t) f^{\rm in} + B(f,f)(\cdot,t) \, ,
\end{equation}
where
\begin{equation}
    \label{eq:bdef}
    B(f,g) = B_1(f,g) + B_2(f,g) \, ,
    \end{equation}
    \begin{equation}
        \label{eq:b1def}
        B_1(f,g)(\cdot,t) = - \int_0^t S(t-s) \div_x (\bu[g] f)(\cdot,s) \, ds \, ,
    \end{equation}
\begin{equation}
    \label{eq:b2def}
    B_2(f,g)(\cdot,t) = - \int_0^t S(t-s) \div_p [(\bI-\bp \otimes \bp) (\nabla_x \bu[g] \bp)f](\cdot,s) \, ds \, .
\end{equation}
Our estimates will be in the function space
\begin{equation}
    X_T = \{ f \in C_t H^2_x L^2_p : \| f \|_{X_T} < +\infty \} \, ,
\end{equation}
where
\begin{equation}
    \| f \|_{X_T} := \sup_{t \in [0,T]} e^{\delta_0 \nu t} \| f(\cdot,t) \|_{H^2_x L^2_p} \, .
\end{equation}
Here, $\delta_0<c_0$, where $c_0$ is the decay rate from the linear theory (Theorem \ref{thm:linTayDisp} and Corollary \ref{cor:smoothingestf}). Note that since we are on $\T^d$ and $\nu\le 1$, only the case $\mu_{\nu,k}=\nu$ is relevant.

Suppose the \emph{bootstrap assumption}:
 $\| f \|_{X_T} \leq C_0 \varepsilon$, where $\varepsilon_0 > 0$ and $C_0 > 2$ will be determined in the course of the proof. We will demonstrate
\begin{equation}
    \| f \|_{X_T} \leq \frac{C_0}{2} \varepsilon \, .
\end{equation}
Hence, the bootstrap assumption can be propagated forward-in-time to complete the proof. This Gronwall-type strategy is common in the long-time behavior of nonlinear PDEs.

To begin, we estimate $B_1$:
    \begin{equation}
    \begin{aligned}
        \| B_1(f,f)(\cdot,t) \|_{H^2_x L^2_p} &\leq  \int_0^t \| S(t-s) \div_x (\bu f)(\cdot,s) \|_{H^2_x L^2_p} \, ds\\
        &\les \int_0^t (\kappa (t-s))^{-1/2} e^{-c_0 \nu (t-s)} \| \bu f(\cdot,s) \|_{H^2_x L^2_p} \, ds \, .
        \end{aligned}
    \end{equation}
    Since $H^2_x$ is an algebra, we have
    \begin{equation}
        \label{eq:b1est}
    \begin{aligned}
        \| B_1(f,f)(\cdot,t) \|_{H^2_x L^2_p} 
        &\les \kappa^{-1/2} C_0^2 \varepsilon^2 \int_0^t (t-s)^{-1/2} e^{-c_0 \nu  (t-s)} e^{-\delta_0 \nu s}  \, ds \, \\
        &\les \kappa^{-1/2}  \nu^{-1/2}  C_0^2 \varepsilon^2e^{-\delta_0 \nu t} \, .
        \end{aligned}
    \end{equation}
    
    Next, we estimate $B_2$. First, we revisit the smoothing estimate in Corollary~\ref{cor:smoothingestf}. If $\bm{g} \in X_t$, then we may conclude
    \begin{equation}
    \begin{aligned}
        \| h(\cdot,t) \|_{H^2_x L^2_p}^2 &\les \nu^{-1} \int_0^t e^{-2 c_0 \nu(t-s)} \| \bm{g} \|_{H^2_x L^2_p}^2 \, ds \\
        &\les \nu^{-1}  \| \bm{g} \|_{X_t}^2 \int_0^t e^{-2 c_0 \nu(t-s)} e^{-2\delta_0 \nu s} \, ds \\
        &\les \nu^{-2} \| \bm{g} \|_{X_t}^2 e^{-2\delta_0 \nu t} \, .
        \end{aligned}
    \end{equation}
    Plugging in $\bm{g} = (\bI-\bp \otimes \bp) (\nabla \bu \bp)f$ and employing that $H^2_x$ is an algebra, we have
    \begin{equation}
        \label{eq:b2est}
        \| B_2(f,f) \|_{X_T} \les \nu^{-1} C_0^2 \varepsilon^2  \, .
    \end{equation}
    
    Finally, the contribution from $f^{\rm in}$ is estimated as follows:
    \begin{equation}
        \label{eq:b3est}
        \| S f^{\rm in} \|_{X_T} \les \| f^{\rm in} \|_{H^2_x L^2_p} \les \varepsilon \, .
    \end{equation}
    
    To conclude the bootstrap argument, we estimate the sum of~\eqref{eq:b1est},~\eqref{eq:b2est}, and~\eqref{eq:b3est} from above. Specifically, we require that
    \begin{equation}
        \kappa^{-1/2} \nu^{-1/2} C_0^2 \varepsilon^2 + \nu^{-1} C_0^2 \varepsilon^2 + \varepsilon \ll C_0 \varepsilon \, .
    \end{equation}
    To ensure this, it is enough to require
    \begin{equation}
        \kappa^{-1/2} \nu^{-1/2} C_0 \varepsilon_0 + \nu^{-1} C_0\varepsilon_0 + C_0^{-1} \ll 1 \, .
    \end{equation}
     To conclude, we choose $C_0 \gg 1$ and $\varepsilon_0 \ll \min(\nu,\kappa)$. This completes the proof.
\end{proof}

Finally, we use the linear Taylor dispersion estimates of Theorem \ref{thm:linTayDisp} to prove Theorem \ref{thm:vacuumstab} on $\R^3$.
Our proof will utilize Besov spaces $(\dot B^s_{2,q})_x L^2_p$. We review them below, although we cannot review the whole theory here and therefore assume a certain familiarity, for example, see the presentation in~\cite[Chapter 2]{bcd}. Let $f \in L^2_{x,p}$. There exists a smooth function $\varphi$, compactly supported on the annulus $\{ 3/4 < |\xi| < 8/3 \}$, and satisfying that $\sum_j \varphi(2^{-j} \xi) = 1$ when $\xi \neq 0$. Define the Littlewood-Paley projections
\begin{equation}
    \widehat{P_j f} = \varphi(2^{-j} \cdot) \widehat{f} \, ,
\end{equation}
\begin{equation}
    P_{\leq j} := \sum_{k \leq j} P_k \, ,
\end{equation}
and $P_{< j}$, $P_{\geq j}$, $P_{> j}$ similarly. Then the Besov norms $\| \cdot \|_{(\dot B^{s}_{2,q})_x L^2_p}$, $(s,q) \in \R \times [1,\infty]$, are defined according to
\begin{equation}
    \| f \|_{(\dot B^{s}_{2,q})_x L^2_p} := \left\lVert \left( 2^{js} \| P_j f \|_{L^2_{x,p}(\R^d \times S^{d-1})} \right) \right\rVert_{\ell^q_j(\Z)} \, .
\end{equation}
Membership in the Besov space is determined by finiteness of the above norm.\footnote{When $s>d/2$ and $(s,q) \neq (d/2,1)$, the above Besov norms are merely seminorms, and corresponding Besov spaces are only well defined up to polynomials. We will not require this level of subtlety here.} When $q=2$, the Besov spaces are equal to the homogeneous Sobolev spaces.

We introduce the spaces $(\dot B^{s}_{2,q})_x L^2_p$ because they are $L^2$-based (hence, amenable to the $L^2$-based linear estimates we derived previously) yet live at different scalings and satisfy desirable embeddings, namely,
\begin{align}
    \| f \|_{L^\infty_x L^2_p} &\les \| f \|_{(\dot B^{d/2}_{2,1})_x L^2_p} \, ,\\
    \| f \|_{(\dot B^{-d/2}_{2,\infty})_x L^2_p} &\les \| f \|_{L^1_x L^2_p} \, .
\end{align}
 We mention also the real interpolation inequality
 \begin{equation}\label{eq:realinterp}
     \| f \|_{(\dot B^s_{2,1})_x L^2_p} \les \| f \|_{(\dot B^{s_1}_{2,\infty})_x L^2_p}^\theta \| f \|_{(\dot B^{s_2}_{2,\infty})_x L^2_p}^{1-\theta} \, ,
 \end{equation}
 where $s = \theta s_1 + (1-\theta) s_2$. Notably, elliptic regularity is well behaved in the above Besov spaces regardless of $q \in [1,+\infty]$. 
 
 Define 
 \begin{equation}
     Y = (\dot B^{d/2}_{2,1} \cap \dot H^2)_x L^2_p\, , \quad Z = (\dot B^{-d/2}_{2,\infty})_x L^2_p \cap Y \, .
 \end{equation}
 Note that $Y$ is a multiplicative algebra. For $T\in (0,+\infty]$, consider the function space
\begin{equation}\label{eq:rdef}
    \| f \|_{X_T} := \sup_{t \in (0,T)} r^{-1}(\nu t)  \| f(\cdot,t) \|_{(\dot B^{d/2}_{2,1} \cap \dot H^2)_x L^2_p} \, , \quad r(s):=\la s\ra^{-\frac{d}{2}}(\log(2+s))^2\,.
\end{equation}
Let $j_0 \in \Z$ be the greatest integer such that $2^{j_0} \leq \nu$; note that $j_0\leq 0$. 

To prove Theorem~\ref{thm:vacuumstab}, we require two further linear estimates, stated in Lemmas~\ref{lem:initialdataest} and \ref{lem:smoothinginpwholespacest}. For the remainder of the section, $S(t) = e^{tL^{0,\iota}_{\nu,k}}$, that is, $\overline\psi=0$.

\begin{lemma}
    \label{lem:initialdataest}
For $\nu \in (0,1]$ and $f^{\rm in} \in Z$, we have
\begin{equation}
    \label{eq:todemonstrate}
    \| S(\cdot) f^{\rm in} \|_{X_\infty} + \sup_{t \in (0,+\infty)} \| S(t)f^{\rm in} \|_{(\dot B^{-d/2}_{2,\infty})_xL^2_p} \les \| f^{\rm in} \|_{Z} 
\end{equation}
\end{lemma}
\begin{proof}
We have the propagation estimates
\begin{equation}
    \sup_{t \in (0,+\infty)} \| S(t) f^{\rm in} \|_Z \leq \| f^{\rm in} \|_Z \, ,
\end{equation}
which imply the second half of the estimate in~\eqref{eq:todemonstrate}. This estimate will also be used to control the solution for $t \leq \nu^{-1}$,
so we focus on $t > \nu^{-1}$. We use the estimates in Theorem \ref{thm:linTayDisp}.
For $j \geq j_0$, we have exponential decay:
\begin{equation}
    \label{eq:exponentialdecaypj}
    \| S(t) P_j f^{\rm in} \|_{Y} \les e^{-\delta_0 \nu t} \| P_j f^{\rm in} \|_{Y} \, .
\end{equation}
For $j < j_0$, we have diffusive decay: for any real $s_2>s_1$,
\begin{equation}
    2^{j s_2} \| S(t) P_j f^{\rm in} \|_{L^2_{x,p}} \les 2^{j (s_2-s_1)} e^{-\delta_0 2^{2j} \nu^{-1} t} 2^{j s_1} \|  P_j f^{\rm in} \|_{L^2_{x,p}} \, .
\end{equation}
In particular, summing in $j < j_0$ and using that $\int_0^{+\infty}y^\alpha e^{-y}\,dy\les_\alpha 1$ for all $\alpha>0$, we have
\begin{equation}
    \| S(t) P_{< j_0} f^{\rm in} \|_{(\dot B^{s_2}_{2,1})_x L^2_p} \les_{s_2-s_1} (t/\nu)^{(s_1-s_2)/2} \|  P_{< j_0} f^{\rm in} \|_{(\dot B^{s_1}_{2,\infty})_xL^2_p} \, , \quad \forall s_2 > s_1 \,, \;  \forall t > 0 \, .
\end{equation}
Choosing $s_1 = -d/2$ and $s_2 \in \{d/2,2\}$, we obtain for $t > \nu^{-1}$,
\begin{equation}
    \label{eq:diffusivedecaypj}
\begin{aligned}
    \| S(t) P_{< j_0} f^{\rm in} \|_{(\dot B^{d/2}_{2,1} \cap \dot H^2)_x L^2_p} &\les (t/\nu)^{-d/2} \| P_{< j_0} f^{\rm in} \|_{(\dot B^{-d/2}_{2,\infty} )_x L^2_p}  \\
    &\les \nu^{d} (\nu t)^{-\frac{d}{2}} \| f^{\rm in} \|_{(\dot B^{-d/2}_{2,\infty} )_x L^2_p} \, .
    \end{aligned}
\end{equation}
Together,~\eqref{eq:exponentialdecaypj},~\eqref{eq:diffusivedecaypj}, and standard embeddings complete the proof of~\eqref{eq:todemonstrate}.
\end{proof}

\begin{lemma}
    \label{lem:smoothinginpwholespacest}
Let $d=3$, $T > 0$, and $\nu \in (0,1]$. Consider $\bm{g} = \bm{g}(\bx,\bp,t)$ satisfying
\begin{equation}
    \| \bm{g}(\cdot,t) \|_{(\dot B^{-d/2}_{2,\infty})_x L^2_p} \les r(\nu t) N \, , \quad \forall t \in (0,T) \, ,
\end{equation}
\begin{equation}
    \| \bm{g} \|_{X_T} \les N \, ,
\end{equation}
for $r$ as in \eqref{eq:rdef} and upper bound $N \geq 0$. Then
$q := \int_0^t S(t-s) \div_p \bm{g} \, ds$ satisfies
\begin{equation}
    \label{eq:muhqest}
    \| q \|_{X_T} + \sup_{t \in (0,T)} \|q(\cdot,t)\|_{(\dot B^{-d/2}_{2,\infty})_xL^2_p} \les \nu^{-1} N \, .
\end{equation}
\end{lemma}
\begin{proof}
We require the estimates in Theorem \ref{thm:linTayDisp} and Corollary \ref{cor:smoothingestf}. To begin, we use Corollary \ref{cor:smoothingestf} and $e^{-c_0 \mu_{\nu,j}(t-s)} \leq 1$ to prove that, for all $j \in \Z$, we have
\begin{equation}
  \label{eq:trivialpropest}
\| \int_0^t P_j S(t-s) \div_p \bm{g} \, ds \|_{L^2_{x,p}}^2 \les \nu^{-1} \int_0^t \| P_j \bm{g}(\cdot,s) \|_{L^2_{x,p}}^2 \, ds \, .
\end{equation}
For the $(\dot B^{-d/2}_{2,\infty})_x L^2_p$ estimate, observe that the right-hand side of~\eqref{eq:trivialpropest} is bounded above by
\begin{equation}
\nu^{-1} 2^{jd} \int_0^t r(\nu s) \, ds \sup_{s \in (0,+\infty)} r^{-1}(\nu s) \| \bm{g}(\cdot,s) \|_{(\dot B^{-d/2}_{2,\infty})_x L^2_p}^2 \les \nu^{-2} 2^{jd} N^2 \, .
\end{equation}
Crucially, we asked that $d=3$ so that the decay rate $r$ is time-integrable.
Multiplying by $2^{-jd}$ and taking a supremum in $j$ already yields the desired estimate on the second term on the left-hand side of~\eqref{eq:muhqest}.

A different way to estimate~\eqref{eq:trivialpropest} is to multiply by $2^{2\sigma}$, $\sigma \in [0,2]$, and sum in $j$ to obtain
\begin{equation}
\| \int_0^t S(t-s) \div_p \bm{g} \, ds \|_{\dot H^\sigma_x L^2_p}^2 \les \nu^{-1} \int_0^t \| \bm{g}(\cdot,s) \|_{\dot H^\sigma_x L^2_p}^2 \, ds \les \nu^{-2} N^2 \, .
\end{equation}
Interpolating in $\sigma$ using \eqref{eq:realinterp}, we additionally have
\begin{equation}
    \| \int_0^t S(t-s) \div_p \bm{g} \, ds \|_{(\dot B^{d/2}_{2,1})_x L^2_p} \les \nu^{-1} N \, .
\end{equation}
For $t \leq \nu^{-1}$, this yields the desired estimate on the $Y = (\dot B^{d/2}_{2,1} \cap \dot H^2)_x L^2_p$ norm.

It remains to estimate the $Y$ norm for $t \geq \nu^{-1}$. This is done by breaking into high and low frequencies:

For $t \geq \nu^{-1}$ and $j \geq j_0-1$, we use the exponential decay estimates
\begin{equation}
\begin{aligned}
   \| P_{j \geq j_0} q(\cdot,t) \|_{\dot H^\sigma_xL^2_p}^2 &\les \sum_{j \geq j_0} 2^{2j\sigma} \| \int_0^t S(t-s) \div_p P_j \bm{g} \, ds \|_{L^2_{x,p}}^2 \\
   &\les \nu^{-1} N^2 \int_0^t e^{-2c_0 \nu (t-s)} r^2(\nu s) \, ds \les N^2 \nu^{-2} r^2(\nu t) \, ,
   \end{aligned}
\end{equation}
valid for all $\sigma \in [0,2]$. Interpolating in $\sigma$, we similarly control $\| P_{j \geq j_0} q(\cdot,t) \|_{\dot B^{d/2}_{2,1}}$.

For $t \geq \nu^{-1}$ and $j < j_0$, we have, for any $s_2>s_1$,
\begin{equation}
\begin{aligned}
    &2^{js_2} \| \int_0^t S(t-s) \div_p P_j \bm{g} \, ds \|_{L^2_{x,p}} \\
    &\qquad \les \nu^{-1/2} \left( \int_0^t 2^{2j(s_2-s_1)} e^{-2\delta_0 2^{2j} (t-s)/\nu} 2^{2 js_1} \| P_j \bm{g}(\cdot,s) \|_{L^2_{x,p}}^2 \,ds \right)^{1/2} \, .
    \end{aligned}
\end{equation}
Here, we choose $s_1 = -d/2$ and $s_2 = \sigma \geq d/2$. 
Employing $\nu \leq 1$, we have
\begin{equation}
\begin{aligned}
    &2^{j\sigma} \| \int_0^t S(t-s) \div_p P_j \bm{g} \, ds \|_{L^2_{x,p}} \\
    &\qquad \les \nu^{-1/2} N j^{-3/2} \left( \int_0^t j^3 2^{2j(\sigma+d/2)} e^{-2\delta_0   2^{2j} \la t-s \ra}  r^2(\nu s) \, ds \right)^{1/2} \\
    &\qquad \les \nu^{-1/2} N j^{-3/2} \left( \int_0^t \la t-s \ra^{-(\sigma+\frac{d}{2})}  (\log (1+\la t-s \ra))^3 r^2(\nu s) \, ds \right)^{1/2} \\
    &\qquad \les \nu^{-1} N j^{-3/2} r(\nu t) \, .
    \end{aligned}
\end{equation}
Summing in $j$ with $\sigma=d/2$, we obtain the desired estimate in $(\dot B^{d/2}_{2,1})_x L^2_p$. Notice that the logarithmic factor is lost in the $j^{-3/2}$ coefficient, which is needed to sum the right-hand side. There is an analogous (simpler) estimate in $\dot H^\sigma$, $\sigma \in (d/2,2]$.
\end{proof}

\begin{proof}[Proof of Theorem~\ref{thm:vacuumstab}]
In keeping with the convention that $f$ is the perturbation around the background state $\psi=\overline\psi=0$, we will write $f$ in place of $\psi$.
We suppose the \emph{bootstrap assumption}:
\begin{equation}
    \label{eq:muhbootstarpping}
    \| f \|_{X_T} + \sup_{t \in (0,T)} \| f(\cdot,t) \|_{(\dot B^{-d/2}_{2,\infty})_x L^2_p} \leq C_0 \varepsilon \, ,
\end{equation}
where $\varepsilon_0 > 0$ and $C_0 \geq 10$ will be determined in the course of the proof. We seek to demonstrate that~\eqref{eq:muhbootstarpping} holds with $C_0/2$ on the right-hand side instead of $C_0.$ Again, we estimate Duhamel's formula~\eqref{eq:duhamelsformula} with $B$, $B_1$, and $B_2$ defined by~\eqref{eq:bdef}-\eqref{eq:b2def}, where now $S(t)=e^{tL^{0,\iota}_{\nu,\kappa}}$.

\emph{Step 0. Initial data}. By Lemma~\ref{lem:initialdataest}, we have
\begin{equation}
\label{eq:initdataestyay}
    \| S(\cdot) f^{\rm in} \|_{X_\infty} + \sup_{t \in (0,+\infty)} \| S(t) f^{\rm in} \|_{(\dot B^{-d/2}_{2,\infty})_x L^2_p} \les \| f^{\rm in} \|_Z \les \varepsilon \, .
\end{equation}

\emph{Step 1. $B_1$ term}. Again, by Lemma~\ref{lem:initialdataest}, we have
\begin{equation}
    \label{eq:duhamelforb1}
    \begin{aligned}
        \| B_1(f,f)(\cdot,t) \|_{Y} &\leq  \int_0^t \| S(t-s) \div_x (\bu f)(\cdot,s) \|_{Y} \, ds\\
        &\les \int_0^t  r(\nu (t-s))  \| e^{\kappa(t-s) \Delta_x} \div_x (\bu f)(\cdot,s) \|_{Z} \, ds \, .
        \end{aligned}
    \end{equation}
    We estimate the products $\bu f$ and $\bu \cdot \nabla_x f$. We begin by recording preliminary estimates. Since $\| f(\cdot,t) \|_{L^1_{x,p}} = M$ is conserved, we have
    \begin{equation}
    \| \bu(\cdot,t) \|_{\dot B^{1-d/2}_{2,\infty}} \les M \, .
    \end{equation}
    Additionally, we have
    \begin{equation}
        \| \bu(\cdot,t) \|_{\dot B^{1+d/2}_{2,1} \cap \dot H^3} \les \| f(\cdot,t) \|_{(\dot B^{d/2}_{2,1} \cap \dot H^2)_x L^2_p} \, .
    \end{equation}
    By real interpolation, we have
    \begin{equation}
        \label{eq:buinterpest}
        \| \bu \|_{\dot B^s_{2,1}} \les M^\theta \| f \|_{\dot B^{d/2}_{2,1}}^{1-\theta}\,,
    \end{equation}
    where $s = \theta(1-d/2) + (1-\theta)(1+d/2)$ and $\theta \in (0,1)$. Similarly,
       \begin{equation}
         \label{eq:finterpest}
        \| f \|_{(\dot B^{s'}_{2,1})_x L^2_p} \les \|f \|_{(\dot B^{-d/2}_{2,\infty})_x L^2_p}^\theta \| f \|_{(\dot B^{d/2}_{2,1})_x L^2_p}^{1-\theta}\,,
    \end{equation}
    where $s' = \theta(-d/2) + (1-\theta)(d/2)$ and $\theta \in (0,1)$.
    
    To estimate $\bu f$, we use~\eqref{eq:buinterpest} with $s=d/2$ and $\theta = 1/3$. Then
    \begin{equation}
        \| \bu f(\cdot,t) \|_{(\dot B^{d/2}_{2,1})_x L^2_p} \les M^{1/3} \| f(\cdot,t) \|_{(\dot B^{d/2}_{2,1})_x L^2_p}^{5/3} \, ,
    \end{equation}
    since $\dot B^{d/2}_{2,1}$ is a multiplicative algebra, and
    \begin{equation}
    \begin{aligned}
        \| \bu f (\cdot,t)\|_{\dot H^2_x L^2_p} &\les \| \bu \nabla^2_x f \|_{L^2_{x,p}} + \| \nabla \bu \otimes \nabla_x f \|_{L^2_{x,p}} + \| \nabla^2 \bu f \|_{L^2_{x,p}} \\
        &\les M^{1/3} \| f(\cdot,t) \|_{Y}^{5/3} + \| f(\cdot,t) \|_{Y}^2 \, .
        \end{aligned}
    \end{equation}
    In summary,
    \begin{equation}
        \| \bu f(\cdot,s) \|_{Y} \les M^{1/3} \| f \|_Y^{5/3} + \| f \|_Y^2 \, ,
    \end{equation}
    and we estimate
    \begin{equation}
        \label{eq:muhZest1}
        \| e^{\kappa (t-s) \Delta_x} \div_x (\bu f) \|_{Y} \les (\kappa(t-s))^{-1/2} (M^{1/3} \| f \|_Y^{5/3} + \| f \|_Y^2 ) \, .
    \end{equation}
    
    To estimate $\bu \cdot \nabla_x f$, we use $\theta=5/6$ in~\eqref{eq:buinterpest} and $\theta=1/6$ in~\eqref{eq:finterpest}
    to obtain 
    \begin{equation}
        \| \bu \|_{L^2_x} \les M^{5/6} \| f \|_{Y}^{1/6} \, ,
    \end{equation}
    \begin{equation}
        \| \nabla_x f \|_{L^2_{x,p}} \les \| f \|_{(\dot B^{-d/2}_{2,\infty})_x L^2_p}^{1/6} \| f \|_{Y}^{5/6} \, .
    \end{equation}
    Upon multiplying the two, we have
    \begin{equation}
        \| \bu \cdot \nabla_x f \|_{L^1_x L^2_p} \les M^{5/6} \| f \|_{(\dot B^{-d/2}_{2,\infty})_x L^2_p}^{1/6} \| f \|_Y \, 
    \end{equation}
    and the estimate
   \begin{equation}
    \label{eq:muhZest2}
       \| e^{\kappa (t-s) \Delta_x} \bu \cdot \nabla_x f \|_{(\dot B^{-d/2}_{2,\infty})_xL^2_p} \les M^{5/6} \| f \|_{(\dot B^{-d/2}_{2,\infty})_x L^2_p}^{1/6} \| f \|_Y \, .
   \end{equation}
    
    We combine~\eqref{eq:muhZest1} and~\eqref{eq:muhZest2} to estimate \eqref{eq:duhamelforb1}:
    \begin{equation}
        \label{eq:B1estyay}
        \begin{aligned}
        \| B_1 \|_{X_T}  &\les (\kappa^{-1/2} \nu^{-1/2} + \nu^{-1}) (M + \| f \|_{(\dot B^{-d/2}_{2,\infty})_x L^2_p} + \| f \|_{X_T}) \| f \|_{X_T} \\
        &\les (\kappa^{-1/2} \nu^{-1/2} + \nu^{-1}) C_0^2 \varepsilon^2 \, .
        \end{aligned}
    \end{equation}
    Here, we use crucially that $d=3$ to ensure that the kernel in~\eqref{eq:duhamelforb1} decays faster than $t^{-1}$. The case $d=2$ seems to be critical for treating $\bu \cdot \nabla_x f$ perturbatively. Additionally, for all $t \in (0,T)$, we have
    \begin{equation}
    \begin{aligned}
        \| B_1(\cdot,t) \|_{(\dot B^{-d/2}_{2,\infty})_x L^2_p} &\leq \|  \int_0^t S(t-s) (\bu \cdot \nabla_x f)(\cdot,s) \, ds \|_{(\dot B^{-d/2}_{2,\infty})_x L^2_p} \\
        &\les \int_0^t r(\nu s) \, ds \, M^{5/6} \| f \|_{(\dot B^{-d/2}_{2,\infty})_x L^2_p}^{1/6} \| f \|_{X_T} \les \nu^{-1} C_0^2 \varepsilon^2\,.
        \end{aligned}
    \end{equation}

    \emph{Step 2. $B_2$ term}. For this, we must estimate the product $\nabla \bu f$. Since $f \mapsto \nabla \bu$ is a zeroth-order operator and $Y$ is an algebra, we have
    \begin{equation}
        \| \nabla \bu f \|_Y \les \| f \|_{Y}^2 \, .
    \end{equation}
    For the $\dot B^{-d/2}_{2,\infty}$ part of $Z$, we use that $\| \nabla \bu \|_{\dot B^{-d/2}_{2,\infty}} \les M$ and continuity\footnote{This follows from the characterization $\dot B^{-d/2}_{2,\infty} = (\dot B^{d/2}_{2,1})^*$ and the algebra property of $\dot B^{d/2}_{2,1}$.} of the product $(f,g) \mapsto fg : \dot B^{-d/2}_{2,\infty} \times \dot B^{d/2}_{2,1} \to \dot B^{-d/2}_{2,\infty}$.
    Using these bounds 
    in Lemma~\ref{lem:smoothinginpwholespacest}, we obtain
    \begin{equation}
    \label{eq:B2estyay}
        \| B_2 \|_{X_T} + \sup_{t \in (0,T)} \| B_2(\cdot,t) \|_{(\dot B^{-d/2}_{2,\infty})_x L^2_p} \les \nu^{-1} C_0^2 \varepsilon^2 \, .
    \end{equation}
    
    \emph{3. Conclusion}. Combining the above estimates~\eqref{eq:initdataestyay},~\eqref{eq:B1estyay}, and~\eqref{eq:B2estyay} into Duhamel's formula~\eqref{eq:duhamelsformula}, we have
    \begin{equation}
        \| f \|_{X_T} + \sup_{t \in (0,T)} \| f(\cdot,t) \|_{(\dot B^{-d/2}_{2,\infty})_x L^2_p} \leq C \varepsilon + C (\kappa^{-1/2} \nu^{-1/2} + \nu^{-1})  C_0^2 \varepsilon^2 \, .
    \end{equation}
    To ensure that the right-hand side is bounded above by $C_0 \varepsilon/2$, we choose $C_0 \gg 1$ and $\varepsilon_0 \ll \min(\nu,\kappa)$. This completes the proof.
\end{proof}





\section{Enhanced dissipation}
\subsection{Linear enhancement}

In this section, we consider the linearized system of equations  
\begin{align}
\p_t f +\bp\cdot\nabla f -\nu \Delta_p f - d \overline\psi\nabla\bu:\bp\otimes \bp &=0 \label{eq:linearizedf} \\
- \Delta \bu + \nabla q = \iota \int_{S^{d-1}} \bp\otimes\bp \, \nabla_x f \, d\bp, \quad \div\bu &=0  \label{eq:linearizedu}
\end{align}
and show that solutions decay at an enhanced rate as long as $\overline\psi$ is sufficiently small. 

As in previous sections, it will be convenient to work mode-by-mode in $\bx$. Our results will be stated for each mode $\bk$, but we will often suppress the $\bk$-dependence in our notation.

We begin by considering the case $\overline \psi=0$. 
In this setting, we obtain the following enhancement result.  

\begin{theorem}[Linear enhancement for $\overline\psi=0$]\label{thm:Phi}
Suppose $f(\bx,\bp,t) = f_k(\bp,t) e^{i\bk \cdot \bx}$, $\bk \in \R^d \setminus \{0\}$, and let $k=\abs{\bk}$.
For $0<\nu< k$, there exist constants $\overline a_1$, $\overline a_2$, $\overline a_3>0$ such that if $\nabla_pf^{\rm in}$ and $({\bf I}-\bp\otimes\bp)\nabla_xf^{\rm in}$ are both in $L^2$, then the functional
\begin{equation}\label{eq:Phi_def0}
\begin{aligned}
 \Phi_k(t) &= \frac{1}{2}\norm{f}_{L^2}^2 + \frac{\overline a_1}{2}\nu^{1/2}k^{-1/2}\norm{\nabla_pf}_{L^2}^2 + \overline a_2 k^{-1} \Re \langle ({\bf I}-\bp\otimes\bp)\nabla_xf,\nabla_pf\rangle \\
 &\hspace{6cm}+ \frac{\overline a_3}{2}\nu^{-1/2}k^{-3/2}\norm{({\bf I}-\bp\otimes\bp)\nabla_xf}_{L^2}^2
 \end{aligned}
 \end{equation} 
 satisfies the fast decay estimate 
\begin{equation}\label{eq:Phi_enhance}
\Phi_k(t) \le e^{-\alpha_0 \nu^{1/2}k^{1/2}t}\Phi_k(0), \qquad t\ge 0,
\end{equation}
where $\alpha_0=\overline a_3$ is explicit. \\

For $f^{\rm in}$ in $L^2$ only, the decay rate of \eqref{eq:Phi_enhance} is modified by a logarithmic factor:
\begin{equation}\label{eq:Q_enhance}
\norm{f_k(\cdot,t)}_{L^2} \le \beta_0^{1/2} \norm{f^{\rm in}_k}_{L^2} e^{-\frac{\alpha_0}{2} \lambda_k t}, \qquad \lambda_{\nu,k} = \frac{\nu^{1/2}k^{1/2}}{1+\abs{\log \nu}+\log k}
\end{equation}
where $\beta_0=e( 1 + \frac{3}{4}\alpha_0 \overline a_1 + \frac{3}{2}\overline a_3 )$.
\end{theorem}

In particular, defining the linear operator 
\begin{equation}
    \mc{L}_\nu f :=\bp\cdot\nabla f -\nu \Delta_p f \,,
\end{equation}
we make note of the operator estimate
\begin{equation}\label{eq:semigrp_est}
\norm{e^{\mc{L}_\nu t}}_{L^2\to L^2} \le \beta_0^{1/2} e^{-\frac{\alpha_0}{2} \lambda_{\nu,k} t}.
\end{equation}

As a consequence of Theorem \ref{thm:Phi}, we obtain the following enhancement result for small swimmer concentrations $\overline\psi$, regardless of the sign of the active stress in \eqref{eq:linearizedu}.

\begin{corollary}[Linear enhancement for small $\overline \psi$]\label{cor:psibar}
For $\overline\psi$ satisfying 
\begin{equation}\label{eq:psibar_bd}
\overline\psi \le \lambda_\nu \frac{\alpha_0}{8d\beta_0^{1/2}},
\end{equation}
 the solution $f$ to the full linearized system \eqref{eq:linearizedf}--\eqref{eq:linearizedu} satisfies 
\begin{equation}\label{eq:psibar_enhance}
\norm{f(\cdot,t)}_{L^2} \le 2e^{-\frac{\alpha_0}{4}\lambda_{\nu}t}\norm{f^{\rm in}}_{L^2}\,,
\end{equation}
provided that $\int_xf^{\rm in}\,d\bx=0$.
Here, $\lambda_\nu=\lambda_{\nu,1}=\frac{\nu^{1/2}}{1+\abs{\log\nu}}$.
\end{corollary}

\subsubsection{Enhancement for \texorpdfstring{$\overline\psi=0$}{}}
For Theorem \ref{thm:Phi}, we consider the equation 
\begin{equation}\label{eq:basic_lin_eqn}
\p_t f +\bp\cdot\nabla_x f - \nu\Delta_p f =0
\end{equation}
and proceed via a similar hypocoercivity argument to~\cite{beckwaynebar, michelepoiseuille}. Recall that we work mode-by-mode in $\bx$ but sometimes suppress the $k$-dependence in our notation. 
We consider the functional $\Phi(t)$ defined by \eqref{eq:Phi_def0} with coefficients $a_1$, $a_2$, and $a_3$ yet to be determined:
\begin{equation}\label{eq:Phi_def}
\Phi(t) = \frac{1}{2}\norm{f}_{L^2}^2 + \frac{a_1}{2}\norm{\nabla_pf}_{L^2}^2 + a_2 \Re \langle ({\bf I}-\bp\otimes\bp)\nabla_xf,\nabla_pf\rangle + \frac{a_3}{2}\norm{({\bf I}-\bp\otimes\bp)\nabla_xf}_{L^2}^2\,.
\end{equation}

Below, we will exploit the following extrinsic formulae for the gradient $\nabla_p$ and divergence $\div_p$ on the sphere. For a smooth function $h(\bp)$ on $S^{d-1}$, $\nabla_p h$ can be computed by extending $h$ arbitrarily to a smooth function $\wt h$ in a neighborhood of the sphere, computing the flat gradient, and projecting back to the tangent space:
\begin{equation}
  \label{eq:flatgradient}
\nabla_p h=({\bf I}-\bp\otimes\bp)\nabla\wt h \, .
\end{equation}
Similarly, given a smooth vector field $V(\bp)$ on $S^{d-1}$, one may arbitrarily extend $V$ to a smooth vector field $\tilde{V}$ in a neighborhood of the sphere, compute the flat gradient, and take its tangential trace:
\begin{equation}
  \label{eq:divcalc}
\div_p V = \nabla \tilde{V} : ({\bf I} - \bp \otimes \bp) \, .
\end{equation}
The extensions can be done homogeneously, for example.

To justify~\eqref{eq:divcalc}, it is sufficient to compute $(\div_p V)(\be_1)$, say, in spherical coordinates $(\phi,\theta)$, where $\bp(\phi,\theta) = (\sin \phi \cos \theta, \sin \phi \sin \theta, \cos \phi)$, $g = \sin^2 \phi \, d\phi^2 + d\theta^2$, and $\div V = (\sqrt{|g|})^{-1} \p_i (\sqrt{|g|} V^i)$ ($|g| = \det g$). The two-dimensional case is simple. As a consequence, we have
\begin{equation}
\nabla_p p_k = \be_k - \bp p_k \, , \quad \div_p \bp = d-1 \, .
\end{equation}

\begin{proof}[Proof of Theorem \ref{thm:Phi}]
Our goal is to choose $a_1$, $a_2$, and $a_3$ such that the bound \eqref{eq:Phi_enhance} holds. 

We first notice that 
\begin{equation}\label{eq:middle_term}
\abs{a_2\langle ({\bf I}-\bp\otimes\bp)\nabla_xf,\nabla_pf\rangle} 
\le \frac{a_2}{4\delta}\norm{({\bf I}-\bp\otimes\bp)\nabla_xf}_{L^2}^2 + \delta a_2\norm{\nabla_pf}_{L^2}^2
\end{equation}
for $\delta>0$; in particular, as long as 
\begin{equation}\label{eq:coeff_cond1}
\frac{a_2}{\delta} \le a_1, \quad 
\delta a_2 \le \frac{a_3}{4},
\end{equation}
for some choice of $\delta$, we have that $\Phi$ satisfies
\begin{align}
\Phi(t) &\ge \frac{1}{2}\norm{f}_{L^2}^2 + \frac{a_1}{4}\norm{\nabla_pf}_{L^2}^2 + \frac{a_3}{4}\norm{({\bf I}-\bp\otimes\bp)\nabla_xf}_{L^2}^2  \label{eq:Phi_lower}\\
\Phi(t) &\le \frac{1}{2}\norm{f}_{L^2}^2 + \frac{3a_1}{4}\norm{\nabla_pf}_{L^2}^2+  \frac{3a_3}{4}\norm{({\bf I}-\bp\otimes\bp)\nabla_xf}_{L^2}^2\,.  \label{eq:Phi_upper}
\end{align} 
We aim to use \eqref{eq:Phi_lower} and \eqref{eq:Phi_upper} to bound $\p_t\Phi(t)$ in terms of $\Phi(t)$.

The time derivatives of each of the four terms of $\Phi$ can be shown to satisfy the following equations: 
\begin{align}
\frac{1}{2}\p_t\norm{f}_{L^2}^2 &= -\nu\norm{\nabla_p f}_{L^2}^2 \label{eq:pt_est1}\\
\frac{1}{2}\p_t\norm{\nabla_p f}_{L^2}^2 &= - \nu\norm{\Delta_pf}_{L^2}^2 - \Re \langle ({\bf I}-\bp\otimes\bp)\nabla_x f,\nabla_p f\rangle \label{eq:pt_est2} \\
\p_t\langle ({\bf I}-\bp\otimes\bp)\nabla_x f,\nabla_p f\rangle 
&= -\norm{({\bf I}-\bp\otimes\bp)\nabla_xf}_{L^2}^2  \label{eq:pt_est3} \\
&\qquad -2\nu \Re \langle ({\bf I}-\bp\otimes\bp)\nabla_x\cdot(\nabla_p f),\Delta_p f\rangle \nonumber \\
&\qquad + (d-1)\nu \langle \bp\cdot\nabla_x f,\Delta_p f\rangle \nonumber \\
\frac{1}{2}\p_t \norm{({\bf I}-\bp\otimes\bp)\nabla_x f}_{L^2}^2 &= -\nu\norm{({\bf I}-\bp\otimes\bp)\nabla_x\nabla_p f}_{L^2}^2  \label{eq:pt_est4} \\
&\qquad +\nu \Re \langle \bp\cdot\nabla_xf,({\bf I}-\bp\otimes\bp)\nabla_x\cdot(\nabla_p f) \rangle \nonumber\\
&\qquad
+ \nu \Re \langle ({\bf I}-\bp\otimes\bp)\nabla_x f,\bp\cdot\nabla_x(\nabla_p f) \rangle\,. \nonumber
\end{align}
The first equation \eqref{eq:pt_est1} follows immediately from \eqref{eq:basic_lin_eqn}, while the second equation \eqref{eq:pt_est2} relies on the commutator 
\begin{equation}
\begin{aligned}
\relax
[\nabla_p,\bp\cdot\nabla_x] &= ({\bf I}-\bp\otimes\bp)\nabla_x\,.
\end{aligned}
\end{equation}

\begin{proof}[Verification of~\eqref{eq:pt_est3},~\eqref{eq:pt_est4}]
We verify~\eqref{eq:pt_est3} and \eqref{eq:pt_est4} mode-by-mode in $\bx$. Without loss of generality, $f = f_k e^{i\bk \cdot \bx}$ and $\bk=k\be_1$.

We show~\eqref{eq:pt_est3} by calculating
\begin{equation}
\begin{aligned}
&\p_t\langle (\be_1-\bp p_1)ik f_k,\nabla_pf_k\rangle \\
&\qquad = -\langle (\be_1-\bp p_1)ik(ikp_1 f_k),\nabla_p f_k\rangle 
+\nu\langle (\be_1-\bp p_1)ik(\Delta_p f_k),\nabla_p f_k\rangle \\
&\qquad \quad
- \langle (\be_1-\bp p_1)ik f_k,\nabla_p(ik p_1 f_k)\rangle
+\nu\langle (\be_1-\bp p_1)ik f_k,\nabla_p(\Delta_p f_k)\rangle \\
&\qquad \overset{\ast}{=} - k^2 \norm{(\be_1-\bp p_1) f_k}_{L^2}^2 
+\nu\langle (\be_1-\bp p_1)ik(\Delta_p f_k),\nabla_p f_k\rangle \\
&\qquad\quad
+\nu\langle (\be_1-\bp p_1)ik f_k,\nabla_p(\Delta_p f_k)\rangle \\
&\quad\;\, \overset{\eqref{eq:divofthethingcalc}}{=} - k^2 \norm{(\be_1-\bp p_1) f_k}_{L^2}^2 
+\nu(d-1)\langle p_1ik f_k,\Delta_p f_k\rangle \\
&\qquad\quad
+\nu\langle (\be_1-\bp p_1)ik(\Delta_p f_k),\nabla_p f_k\rangle-\nu\langle (\be_1-\bp p_1)ik \cdot\nabla_p f_k,\Delta_p f_k\rangle \\
%
%
%
%
&\qquad= - k^2 \norm{(\be_1-\bp p_1) f_k}_{L^2}^2 
+\nu(d-1)\langle p_1ik f_k,\Delta_p f_k\rangle \\
&\qquad\quad
-2\nu \Re\langle (\be_1-\bp p_1)ik \cdot\nabla_p f_k,\Delta_p f_k\rangle \,,
\end{aligned}
\end{equation}
where in the $\ast$ step we simplify the two terms without $\nu$, and
\begin{equation}
  \label{eq:divofthethingcalc}
\div_p(\be_1-\bp p_1) = -(d-1)p_1\,.
\end{equation}

Furthermore, we may calculate \eqref{eq:pt_est4} by
\begin{equation}
\begin{aligned}
&\frac{1}{2}\p_t \norm{(\be_1-\bp p_1)ik f_k}_{L^2}^2 \\
&\qquad = \underbrace{- \Re \langle (\be_1-\bp p_1)ik f_k, ik p_1(\be_1-\bp p_1)ik f_k \rangle}_{ = 0} \\
&\qquad\quad+\nu \Re \langle (\be_1-\bp p_1)ik f_k,(\be_1-\bp p_1)ik(\Delta_p f_k) \rangle \\
&\qquad = \nu k^2 \Re \langle (1-p_1^2) f_k,\Delta_p f_k \rangle \\
&\qquad= -\nu k^2 \Re \langle \nabla_p((1-p_1^2) f_k),\nabla_p f_k \rangle \\
&\qquad = -\nu k^2 \Re \langle (1-p_1^2)\nabla_p f_k,\nabla_p f_k \rangle + 2\nu k^2 \Re \langle (\be_1-\bp p_1) p_1 f_k,\nabla_p f_k \rangle\\
&\qquad = -\nu k^2 \norm{(\be_1-\bp p_1)\nabla_p f_k}_{L^2}^2 +\nu k^2 \Re \langle p_1 f_k,(\be_1-\bp p_1)\cdot \nabla_p f_k \rangle \\
&\qquad\quad
+ \nu k^2 \Re \langle (\be_1-\bp p_1)f_k,p_1\nabla_p f_k \rangle\, ,
\end{aligned}
\end{equation}
where we used that $| \be_1 - \bp p_1|^2 = 1-p_1^2$.
\end{proof}


Combining equations \eqref{eq:pt_est1}--\eqref{eq:pt_est4}, $\p_t\Phi$ then satisfies 
\begin{equation}\label{eq:pt_Phi0}
\begin{aligned}
\p_t\Phi(t) 
&\le -\nu\norm{\nabla_p f}_{L^2}^2 - \nu a_1\norm{\Delta_p f}_{L^2}^2 -a_2\norm{({\bf I}-\bp\otimes\bp)\nabla_xf}_{L^2}^2 \\
&\quad -\nu a_3 \norm{({\bf I}-\bp\otimes\bp)\nabla_x\nabla_pf}_{L^2}^2
 +a_1\norm{ ({\bf I}-\bp\otimes\bp)\nabla_xf}_{L^2}\norm{\nabla_pf}_{L^2} \\
 &\quad 
 +2\nu a_2\norm{ ({\bf I}-\bp\otimes\bp)\nabla_x\cdot(\nabla_pf)}_{L^2}\norm{\Delta_pf}_{L^2} \\
 &\quad +(d-1)\nu a_2 \norm{ \bp\cdot\nabla_xf}_{L^2}\norm{\Delta_pf}_{L^2} \\
 &\quad +\nu a_3 \norm{\bp\cdot\nabla_xf}_{L^2}\norm{({\bf I}-\bp\otimes\bp)\nabla_x\cdot(\nabla_pf)}_{L^2} \\
 &\quad + \nu a_3\norm{ ({\bf I}-\bp\otimes\bp)\nabla_xf}_{L^2}\norm{\bp\cdot\nabla_x(\nabla_pf) }_{L^2}\\
 &\le -(\nu-\delta_0a_1)\norm{\nabla_pf}_{L^2}^2 - \nu\bigg( a_1-a_2\delta_1 - (d-1) a_2 \delta_2\bigg)\norm{\Delta_pf}_{L^2}^2  \\
&\quad -\bigg(a_2-\frac{a_1}{4\delta_0}- \nu \frac{a_3}{4\delta_4} \bigg)\norm{ ({\bf I}-\bp\otimes\bp)\nabla_xf}_{L^2}^2 \\
&\quad -\nu\bigg( a_3- \delta_3 a_3-  \frac{a_2}{\delta_1} \bigg) \norm{({\bf I}-\bp\otimes\bp)\nabla_x\nabla_pf}_{L^2}^2 \\
&\quad + \nu\bigg( (d-1) \frac{a_2}{4\delta_2} +\frac{a_3}{4\delta_3} \bigg) \norm{\bp\cdot\nabla_xf}_{L^2}^2 + \nu \delta_4 a_3\norm{\bp\cdot\nabla_x(\nabla_pf) }_{L^2}^2\,,
\end{aligned}
\end{equation}
where each $\delta_j>0$ is also yet to be determined.  

From here, it will become crucial that we work mode-by-mode in $\bx$. 
The chosen constants $a_j$ and $\delta_j$ will depend on both $\nu$ and $k$ in such a way that the right-hand side of \eqref{eq:pt_Phi0} is bounded by a multiple of $\Phi(t)$. 

We make use of the following Poincar\'e-type inequality on $\T^d$, cf. \cite[Lemma 3.8]{beckwaynebar}.
\begin{lemma}\label{lem:p_poincare}
For $f(\bp,\bx)=f_k(\bp)e^{i\bk\cdot\bx}$ with $\nabla_p f\in L^2$ and $({\bf I}-\bp\otimes\bp)\nabla_x f\in L^2$,  whenever $0<\nu<k$, we have
\begin{equation}\label{eq:p_poincare}
\norm{\nabla_x f}_{L^2}^2 \les k^{3/2}\nu^{1/2}\norm{\nabla_p f}_{L^2}^2 + k^{1/2}\nu^{-1/2}\norm{({\bf I}-\bp\otimes\bp)\nabla_x f}_{L^2}^2\,.
\end{equation}
\end{lemma}
The proof of Lemma \ref{lem:p_poincare} appears in Appendix~\ref{app:p_poincare}. 

To show the bound \eqref{eq:Phi_enhance}, we need each term of $\Phi(t)$, defined by \eqref{eq:Phi_def}, to appear on the right-hand side of the inequality \eqref{eq:pt_Phi0} with an appropriate sign. Note that we are still missing a term proportional to $\norm{f}_{L^2}^2$.
We may use Lemma \ref{lem:p_poincare} to insert this term: in particular, for $a_4>0$, we have, for an appropriate constant $c_0$, 
\begin{equation}\label{eq:fL2_poin}
\begin{aligned}
&-a_4 k^2\norm{f}_{L^2}^2 + a_4c_0\bigg(k^{3/2}\nu^{1/2}\norm{\nabla_p f}_{L^2}^2 \\
&\hspace{5cm}+ k^{1/2}\nu^{-1/2}\norm{({\bf I}-\bp\otimes\bp)\nabla_x f}_{L^2}^2\bigg) \ge 0\,.
\end{aligned}
\end{equation}

Choosing the coefficients $a_j$ and $\delta_j$ in \eqref{eq:pt_Phi0} and \eqref{eq:fL2_poin} to satisfy 
\begin{equation}\label{eq:coeffs0}
\begin{aligned}
a_1&=\overline a_1\nu^{1/2}k^{-1/2}, \quad a_2=\overline a_2 k^{-1}, \quad a_3=\overline a_3\nu^{-1/2}k^{-3/2}, \\
a_4 &= \overline a_4 \nu^{1/2}k^{-3/2}, \quad 
\delta_0= \overline\delta_0\nu^{1/2}k^{1/2}, \quad \delta_1=\overline\delta_1\nu^{1/2}k^{1/2}, \\
\delta_2&=\overline\delta_2\nu^{1/2}k^{1/2}, \quad \delta_3=\overline\delta_3, \quad \delta_4=\overline \delta_4 \nu^{1/2}k^{-1/2}\,,
\end{aligned}
\end{equation}
and using Lemma \ref{lem:p_poincare} again to absorb the final two terms on the right-hand side of \eqref{eq:pt_Phi0} into previous terms, we obtain (increasing $c_0$ if necessary)
\begin{equation}\label{eq:pt_Phi1} 
\begin{aligned}
&\p_t\Phi(t)\\
&\le -\nu^{1/2}k^{1/2}\overline a_4\norm{f}_{L^2}^2 -\nu \bigg(1-\overline a_1\overline \delta_0 -(d-1) c_0 \frac{\overline a_2}{4\overline \delta_2} - c_0 \frac{\overline a_3}{4\overline \delta_3} -c_0\overline a_4\bigg)\norm{\nabla_p f}_{L^2}^2  \\
&\quad - \nu^{3/2}k^{-1/2}\bigg( \overline a_1- \overline a_2\overline \delta_1 - (d-1) \overline a_2 \overline \delta_2 - c_0\overline\delta_4 \overline a_3 \bigg)\norm{\Delta_pf}_{L^2}^2  \\
&\quad -k^{-1}\bigg(\overline a_2- \frac{\overline a_1}{4\overline\delta_0}- \frac{\overline a_3}{4\overline \delta_4} - (d-1)c_0\frac{\overline a_2}{4\overline \delta_2} - c_0\frac{\overline a_3}{4\overline \delta_3} -c_0\overline a_4\bigg)\norm{({\bf I}-\bp\otimes\bp)\nabla_x f}_{L^2}^2\\
&\quad  -\nu^{1/2}k^{-3/2}\bigg(\overline a_3- \overline \delta_3 \overline a_3-  \frac{\overline a_2}{\overline \delta_1}  - c_0\overline\delta_4 \overline a_3\bigg) \norm{({\bf I}-\bp\otimes\bp)\nabla_x\nabla_pf}_{L^2}^2\,.
\end{aligned}
\end{equation}

We now choose $\overline a_j$ and $\overline \delta_j$ such that the coefficients of each of the terms in \eqref{eq:pt_Phi1} are strictly negative. Recalling that we also need to satisfy the condition \eqref{eq:coeff_cond1}, we take
\begin{equation}\label{eq:coeffs1}
\begin{aligned}
\overline a_1 &= \frac{c_0}{64c_0^2+1},\quad 
\overline a_2 = \frac{\overline a_1}{16c_0} = \frac{1}{16(64c_0^2+1)}, \\
\overline a_3 &= \frac{\overline a_1}{64c_0^2} =\frac{1}{64c_0(64c_0^2+1)}, \quad
\overline a_4 = \frac{\overline a_1}{128c_0^2} = \frac{1}{128c_0(64c_0^2+1)},\\
\overline \delta_0&=64c_0,\quad   \overline \delta_1= 10c_0 , \quad \overline \delta_2 = (d-1)c_0, \quad \overline \delta_3 = \frac{1}{4}, \quad \overline \delta_4 = \frac{1}{4c_0}\,.
\end{aligned}
\end{equation}
Note that \eqref{eq:coeff_cond1} holds for $\delta=\frac{1}{16c_0}\nu^{-1/2}k^{-1/2}$. We thus have that $\p_t\Phi$ satisfies
\begin{equation}\label{eq:pt_Phi2}
\begin{aligned}
\p_t\Phi(t)&\le -\nu^{1/2}k^{1/2}\overline a_4\norm{f}_{L^2}^2 -\nu \frac{123}{128c_0}\overline a_1\norm{\nabla_p f}_{L^2}^2 \\
&\quad -k^{-1}\frac{c_0}{4}\overline a_3\norm{({\bf I}-\bp\otimes\bp)\nabla_x f}_{L^2}^2 -\nu^{3/2}k^{-1/2}7c_0^2\overline a_3 \norm{\Delta_pf}_{L^2}^2  \\
&\quad 
 -\nu^{1/2}k^{-3/2}  \frac{\overline a_3}{10} \norm{({\bf I}-\bp\otimes\bp)\nabla_x\nabla_pf}_{L^2}^2 \\
&\le -2\overline a_4\nu^{1/2}k^{1/2}\Phi(t)\,.
\end{aligned}
\end{equation}
Here we have used the upper bound \eqref{eq:Phi_upper} for $\Phi$ to bound the first three terms. Noting that $2\overline a_4=\overline a_3=\frac{1}{64c_0(64c_0^2+1)}$, we obtain the estimate \eqref{eq:Phi_enhance} of Theorem \ref{thm:Phi}.\\

We next show that the estimate \eqref{eq:Q_enhance} holds for $L^2$ initial data. The strategy of the proof can be found in, e.g.,~\cite{michelepoiseuille}. 

We begin by noting that, from \eqref{eq:pt_est1}, $\norm{f}_{L^2}$ satisfies 
\begin{equation}\label{eq:finL2_eqn}
\norm{f(\cdot,t)}_{L^2}^2 = \norm{f^{\rm in}}_{L^2}^2 - 2\nu\int_0^t\norm{\nabla_p f}_{L^2}^2\,,
\end{equation}
and, in particular, 
\begin{equation}\label{eq:fin_est1}
\norm{f(\cdot,t)}_{L^2}^2 \le \norm{f^{\rm in}}_{L^2}^2\,.
\end{equation}
Notice that for $0\le t\le T_{\nu,k} := \frac{1+\abs{\log\nu}+\log k}{\alpha_0 \nu^{1/2}k^{1/2}}$, \eqref{eq:fin_est1} automatically implies \eqref{eq:Q_enhance}. 

For $t\ge T_{\nu,k}$, we first note that by \eqref{eq:finL2_eqn} and the mean value theorem, there exists $t^*\in \big(0,\frac{1}{\alpha_0\nu^{1/2}k^{1/2}}\big)$ such that
\begin{equation}\label{eq:nablapf_MVT}
2\nu\norm{\nabla_pf(\cdot,t^*)}_{L^2}^2 =\alpha_0 \nu^{1/2}k^{1/2}\left(\norm{f^{\rm in}}_{L^2}^2 - \norm{f(\cdot,t)}_{L^2}^2\right) \le \alpha_0 \nu^{1/2}k^{1/2}\norm{f^{\rm in}}_{L^2}^2\,.
\end{equation}
Furthermore, using \eqref{eq:coeffs0}, we may rewrite the upper bound \eqref{eq:Phi_upper} as 
\begin{equation}\label{eq:Phi_upper2}
\Phi(t) \le \frac{1}{2}\norm{f}_{L^2}^2 + \frac{3}{4}\overline a_1 \nu^{1/2}k^{-1/2}\norm{\nabla_pf}_{L^2}^2+ \frac{3}{4}\overline a_3 \nu^{-1/2}k^{-3/2}\norm{({\bf I}-\bp\otimes\bp)\nabla_xf}_{L^2}^2\,.
\end{equation}
Using \eqref{eq:nablapf_MVT} in \eqref{eq:Phi_upper2}, we then have
\begin{equation}\label{eq:Phi_star_bd}
\begin{aligned}
\Phi(t^*)&\le \frac{1}{2}\norm{f(\cdot,t^*)}_{L^2}^2 
+ \frac{3}{8} \alpha_0 \overline a_1 \norm{f^{\rm in}}_{L^2}^2
+ \frac{3}{4}\overline a_3 \nu^{-1/2}k^{-3/2}\norm{({\bf I}-\bp\otimes\bp)\nabla_xf(\cdot,t^*)}_{L^2}^2\\
&\le \frac{1}{2}\norm{f(\cdot,t^*)}_{L^2}^2 
+ \frac{3}{8} \alpha_0 \overline a_1 \norm{f^{\rm in}}_{L^2}^2
+ \frac{3}{4}\overline a_3 \nu^{-1/2}k^{1/2}\norm{f(\cdot,t^*)}_{L^2}^2 \\
&\le \bigg( \frac{1}{2} + \frac{3}{8}\alpha_0 \overline a_1 
+ \frac{3}{4}\overline a_3\bigg) \nu^{-1/2}k^{1/2}\norm{f^{\rm in}}_{L^2}\,,
\end{aligned}
\end{equation}
by \eqref{eq:fin_est1}. For $t\ge T_{\nu, k}$, we thus obtain
\begin{equation}\label{eq:fin_est2}
\begin{aligned}
\norm{f(\cdot,t)}_{L^2}^2 &\le 2\Phi(t) \\
&\le 2e^{-\alpha_0\nu^{1/2}k^{1/2}(t-t^*)}\Phi(t^*)\\
&\le \bigg( 1 + \frac{3}{4} \alpha_0\overline a_1 + \frac{3}{2}\overline a_3 \bigg)e^{\alpha_0 \nu^{1/2}k^{1/2}t^*}\nu^{-1/2}k^{1/2}e^{-\alpha_0 \nu^{1/2}k^{1/2}t}\norm{f^{\rm in}}_{L^2}^2 \\
&\le e\bigg( 1 + \frac{3}{4}\alpha_0 \overline a_1 + \frac{3}{2}\overline a_3 \bigg)\nu^{-1/2}k^{1/2}e^{-\alpha_0 \nu^{1/2}k^{1/2}t}\norm{f^{\rm in}}_{L^2}^2\,,
\end{aligned}
\end{equation}
where we have used that $t^*\le \frac{1}{\alpha_0 \nu^{1/2}k^{1/2}}$. Now, since $t\ge T_{\nu,k}$, the following bound holds:
\begin{equation}\label{eq:log_ID}
\nu^{-1/2}k^{1/2}e^{-\alpha_0 \nu^{1/2}k^{1/2}t} \le e^{-\frac{\alpha_0\nu^{1/2}k^{1/2}}{1+\abs{\log\nu}+\log k}t}\,.
\end{equation}
Inserting \eqref{eq:log_ID} in \eqref{eq:fin_est2} yields
\begin{equation}
\norm{f(\cdot,t)}_{L^2}^2 \le e\bigg( 1 + \frac{3}{4}\alpha_0 \overline a_1 + \frac{3}{2}\overline a_3 \bigg) \norm{f^{\rm in}}_{L^2}^2 e^{-\alpha_0 \frac{\nu^{1/2}k^{1/2}}{1+\abs{\log \nu}+\log k} t}\,,
\end{equation}
and, defining $\beta_0=e( 1 + \frac{3}{4}\alpha_0 \overline a_1 + \frac{3}{2}\overline a_3 )$, we obtain \eqref{eq:Q_enhance}. 
\end{proof}

\subsubsection{Enhancement for small \texorpdfstring{$\overline\psi$}{}}
We now consider the full linearized system \eqref{eq:linearizedf}--\eqref{eq:linearizedu} with small $\overline\psi>0$. 

\begin{proof}[Proof of Corollary \ref{cor:psibar}]
We begin by bounding $\nabla_x\bu$ in terms of $f$. Upon multiplying \eqref{eq:linearizedu} and integrating by parts, we have 
 \begin{equation}\label{eq:gradu_bd0}
 \norm{\nabla_x\bu}_{L^2}^2 = -\iota \langle \nabla_x\bu, \bp\otimes\bp f\rangle \le \norm{\nabla_x\bu}_{L^2}\norm{f}_{L^2} \le \frac{1}{2}\norm{\nabla_x\bu}_{L^2}^2 + \frac{1}{2}\norm{f}_{L^2}^2\,,
 \end{equation}
from which we obtain the bound
 \begin{equation}\label{eq:gradu_bd}
 \norm{\nabla_x\bu}_{L^2} \le \norm{f}_{L^2}\,.
 \end{equation}

We write $f$ satisfying \eqref{eq:linearizedf}--\eqref{eq:linearizedu} using Duhamel's formula: 
\begin{equation}\label{eq:f_duhamel}
f(\cdot,t) = e^{\mc{L}_\nu t}f^{\rm in} + d\overline\psi \int_0^t e^{\mc{L}_\nu(t-s)}\nabla_x\bu:\bp\otimes \bp\, ds\,.
\end{equation}
Note that, using \eqref{eq:gradu_bd} along with the semigroup estimate \eqref{eq:semigrp_est}, we have
\begin{equation}\label{eq:L2integral_est}
\norm{\int_0^t e^{\mc{L}_\nu (t-s)}\nabla_x\bu:\bp\otimes \bp\, ds}_{L^2} 
\le \beta_0^{1/2}\int_0^t e^{-\frac{\alpha_0}{2}\lambda_{\nu}(t-s)}\norm{f(\cdot,t)}_{L^2}\, ds \,.
\end{equation}
Then, multiplying \eqref{eq:f_duhamel} by the time weight $e^{\frac{\alpha_0}{4}\lambda_{\nu}t}$ and taking the $L^2$ norm in $\bx$ and $\bp$, we may estimate the integral term as 
\begin{equation}\label{eq:fL2_timewt}
\begin{aligned}
e^{\frac{\alpha_0}{4}\lambda_{\nu}t}\norm{f(\cdot,t)}_{L^2} &\le
e^{-\frac{\alpha_0}{4}\lambda_{\nu} t}\norm{f^{\rm in}}_{L^2} + d\overline\psi \beta_0^{1/2}  \int_0^t e^{-\frac{\alpha_0}{4}\lambda_{\nu}(t-s)} e^{\frac{\alpha_0}{4}\lambda_{\nu}s}\norm{f(\cdot,s)}_{L^2}\, ds \\
&\le
e^{-\frac{\alpha_0}{4}\lambda_{\nu} t}\norm{f^{\rm in}}_{L^2} +  \frac{\overline \psi}{\lambda_{\nu}} \frac{4d\beta_0^{1/2}}{\alpha_0} \sup_{t\in(0,T)}\left(e^{\frac{\alpha_0}{4}\lambda_{\nu}t}\norm{f(\cdot,t)}_{L^2} \right) \,.
\end{aligned}
\end{equation}
Thus, if we take $\overline\psi$ to satisfy
\begin{equation}\label{eq:psibar_bd0}
\overline\psi \le \lambda_{\nu}\frac{\alpha_0}{8d\beta_0^{1/2}}\,,
\end{equation}
we may absorb the $\norm{f(\cdot,t)}_{L^2}$ term on the right-hand side of \eqref{eq:fL2_timewt} into the left-hand side to obtain 
\begin{equation}\label{eq:fL2_timewt2}
\sup_{t\in(0,T)}\left(e^{\frac{\alpha_0}{4}\lambda_{\nu}t}\norm{f(\cdot,t)}_{L^2} \right) \le 2\sup_{t\in(0,T)}\left(e^{-\frac{\alpha_0}{4}\lambda_{\nu} t}\right)\norm{f^{\rm in}}_{L^2} \le 2\norm{f^{\rm in}}_{L^2}\,.
\end{equation}
From this we obtain \eqref{eq:psibar_enhance}. 
\end{proof}

Finally, we demonstrate a smoothing estimate which will be necessary in the nonlinear argument:
\begin{lemma}[Smoothing-in-$\bp$]\label{lem:smoothp}
(i) Let $f$ be the solution of the linearized PDE on $\T^d \times S^{d-1}$ with $\overline{\psi} \ll \lambda_\nu$ and $f(\cdot,0) = f^{\rm in} \in L^2$. Then
\begin{equation}
    \nu \int_0^{+\infty} e^{2c \lambda_\nu t} \int |\nabla_p f|^2 \, d\bx \, d\bp \, dt \les \| f^{\rm in} \|_{L^2}^2 \, .
\end{equation}
(ii) Let $h$ be the solution of the linearized PDE
\begin{equation}
    \label{eq:hpde}
    \p_t h + \bp \cdot \nabla_x h - d \overline{\psi} \nabla \bu : \bp \otimes \bp = \nu \Delta_p h + \kappa \Delta_x h + \div_p \bm{g}
\end{equation}
on $\T^d \times S^{d-1}$ for some $\bm{g} \in L^2(\T^d \times S^{d-1} \times \R_+)$, with $\overline{\psi} \ll \lambda_\nu$ and $h(\cdot,0) = 0$. Then
\begin{equation}
    \| h(\cdot,t) \|_{L^2}^2 \les \nu^{-1} \int_0^t e^{-2c\lambda_\nu (t-s)} \int |\bm{g}|^2 \, d\bx \, d\bp \, ds \, .
\end{equation}
\end{lemma}
\begin{proof}
To prove (i), we extract the following from the basic energy estimate:
\begin{equation}
\begin{aligned}
    \nu \int_{t_1}^{t_2} \int |\nabla_p f|^2 \, d\bx \, d\bp \,dt &\leq \frac{1}{2} \| f(\cdot,t_1) \|_{L^2}^2 + C \overline{\psi} \int_{t_1}^{t_2} \int |f|^2 \, d\bx \, d\bp \,dt \, \\
    &\leq C e^{-2c\lambda_\nu t_1}  \|f^{\rm in}\|_{L^2}^2  + C \overline{\psi} \int_{t_1}^{+\infty} e^{-2c\lambda_\nu t} \, dt\,  \| f^{\rm in} \|_{L^2}^2  \\
    &\leq C e^{-2c\lambda_\nu t_1}  \|f^{\rm in}\|_{L^2}^2 \, ,
    \end{aligned}
\end{equation}
since $\overline{\psi} \ll \lambda_\nu$. We choose $t_2 = t_1 + \lambda_\nu^{-1}$ and sum the estimate over $t_1 = k \lambda_\nu^{-1}$, $k = 0, 1, \hdots$. Finally, (ii) follows from a duality argument as in Corollary~\ref{cor:smoothingestf}. \end{proof}



\subsection{Nonlinear enhancement}





Finally, we consider the full nonlinear equation \eqref{eq:smoluchowski} for $\psi=\overline\psi+f$ on $\T^d$. Recalling the definition \eqref{eq:Lnukappa} of the linear operator $L_{\nu,\kappa}^{\overline\psi,\iota}$, we write \eqref{eq:smoluchowski} as an equation for the perturbation $f$: 
\begin{equation}\label{eq:nonlin_eq0}
\p_t f -L_{\nu,\kappa}^{\overline\psi,\iota}f +\bu\cdot\nabla_x f + \divp[({\bf I}-\bp\otimes\bp)(\nabla\bu[f] \bp)f ] = 0 \, .
\end{equation} 
We show that for sufficiently small $f^{\rm in}$ and $\overline\psi$, the nonlinear evolution of $f$ satisfies the same enhanced dissipation as the linear evolution (Corollary \ref{cor:psibar}). \\

Noting that the $k=0$ spatial mode is not enhanced, we must treat the $k=0$ and $k\neq0$ modes of $f$ separately. We define $\PP_0$ to be the projection onto the $k=0$ spatial mode,
\begin{equation}
\PP_0 h(\bx,\bp,t) = \int_{\T^d} h(\bx,\bp,t) \, d\bx \, ,
\end{equation}
and define $\PP_{\neq}={\rm Id}-\PP_0$. We consider separately the evolution of  
\begin{equation}
f_0:= \PP_0 f \quad \text{and} \quad f_{\neq}:=\PP_{\neq}f\, .
\end{equation}

Note that $\PP_0$ commutes with $L_{\nu,\kappa}^{\overline\psi,\iota}$ and that $L_{\nu,\kappa}^{\overline\psi,\iota}f_0=-\nu\Delta_p f_0$. Further note that $\PP_0\big(\bu\cdot\nabla_x f\big) = \PP_0\big(\div_x(\bu f)\big)=0$ and that, since $\PP_0\nabla_x\bu=0$, the only contribution to the $k=0$ mode from the $\divp$ term in \eqref{eq:nonlin_eq0} is due to interactions between $\bu[f_{\neq}]$ and $f_{\neq}$. The evolution of the zero mode $f_0(\bp,t)$ thus satisfies a forced heat equation in $\bp$ whose forcing depends only on nonzero modes: 
\begin{equation}\label{eq:0mode1}
\p_t f_0 -\nu\Delta_p f_0 + \PP_0\,\divp[({\bf I}-\bp\otimes\bp)(\nabla_x\bu[f_{\neq}] \bp)f_{\neq} ] = 0.
\end{equation} 
Using that $\bu[f_0]=0$, the nonzero modes $f_{\neq}(\bx,\bp,t)$ evolve via
\begin{equation}\label{eq:neqf1}
\begin{aligned}
\p_t f_{\neq} -L_{\nu,\kappa}^{\overline\psi,\iota}f_{\neq}+\bu[f_{\neq}]\cdot\nabla_x f_{\neq} &+ \PP_{\neq}\divp[({\bf I}-\bp\otimes\bp)(\nabla_x\bu[f_{\neq}] \bp)f_{\neq} ]\\
& + \PP_{\neq}\divp[({\bf I}-\bp\otimes\bp)(\nabla_x\bu[f_{\neq}] \bp)f_0 ]= 0.
\end{aligned}
\end{equation}

We are now equipped to prove the nonlinear enhancement result of Theorem \ref{thm:nonlin_enhance}.



\begin{proof}[Proof of Theorem \ref{thm:nonlin_enhance}]
We begin by defining our function spaces. For the zero mode evolution, we will consider the space $X_T=\{g\in C_tH^2_xL^2_p \, : \, \norm{g}_{X_T}<\infty\}$; for the nonzero modes we consider $Y_T=\{g\in C_tH^2_xL^2_p \, : \, \norm{g}_{Y_T}<\infty\}$,
where the norms $\norm{\cdot}_{X_T}$ and $\norm{\cdot}_{Y_T}$ are given by
\begin{equation}
\norm{\cdot}_{X_T} = \sup_{t\in[0,T]}\, e^{\delta_0\nu t}\norm{\cdot}_{H^2_xL^2_p}, \qquad  \norm{\cdot}_{Y_T} = \sup_{t\in[0,T]}\,e^{\delta_{\neq}\lambda_\nu t}\norm{\cdot}_{H^2_xL^2_p} 
\end{equation}
for constants $0<\delta_0<c_0$ and $0<\delta_{\neq}<c_{\neq}$, where $c_0$ (defined in \eqref{eq:heatp_ests}) and $c_{\neq}=\alpha_0/4$ (defined in Corollary \ref{cor:psibar}) are decay rates from the linear theory.

We again use a bootstrap argument to show Theorem \ref{thm:nonlin_enhance}. For some $\varepsilon\le \varepsilon_0$ and $C_0>2$ to be determined, we assume that $\norm{f_{\neq}}_{Y_T} \le C_0\varepsilon$, and aim to show that 
\begin{equation}\label{eq:bootstrap_tbd}
 \norm{f_{\neq}}_{Y_T} \le \frac{C_0}{2}\varepsilon\,.
 \end{equation} 


To show \eqref{eq:bootstrap_tbd}, we begin by bounding the evolution of $f_0$ in \eqref{eq:0mode1}, since $f_{\neq}$ depends on $f_0$ via \eqref{eq:neqf1}. Using Duhamel's formula, we may write $f_0$ as
\begin{equation}\label{eq:f0_duhamel}
\begin{aligned}
f_0(\bp,t) &= e^{t\nu\Delta_p }f^{\rm in}_0(\bp) + B_0(f_{\neq},f_{\neq})(\bp,t)\,, \\
B_0(f,g)(\bp,t) &= -\PP_0\int_0^t e^{(t-s)\nu\Delta_p} \divp[({\bf I}-\bp\otimes\bp)(\nabla_x\bu[g] \bp)f ](\bx,\bp,s)\, ds\,.
\end{aligned}
\end{equation}
Note that by standard estimates for the heat equation on $S^{d-1}$, we have
\begin{equation}\label{eq:heatp_ests}
\norm{e^{t\nu\Delta_p }}_{L^2_p\to L^2_p} \lesssim e^{-c_0\nu t}\,, \qquad \norm{e^{t\nu\Delta_p }\divp}_{L^2_p\to L^2_p} \lesssim (\nu t)^{-1/2}e^{-c_0\nu t}
\end{equation}
for some $c_0>0$. Using \eqref{eq:heatp_ests} and that $H^2_x$ is an algebra, we may estimate the bilinear term $B_0$ in \eqref{eq:f0_duhamel} as
\begin{equation}
\begin{aligned}
&\norm{B_0(f_{\neq},f_{\neq})(\cdot,t)}_{H^2_xL^2_p} \\
&\qquad \lesssim \int_0^t(\nu (t-s))^{-1/2}e^{-c_0\nu(t-s)}\norm{({\bf I}-\bp\otimes\bp)(\nabla_x\bu_{\neq} \bp)f_{\neq} }_{H^2_xL^2_p} \, ds \\
&\qquad \lesssim \int_0^t (\nu(t-s))^{-1/2} e^{-c_0\nu(t-s)}e^{-2\delta_{\neq}\lambda_\nu s} \, ds\,  \norm{f_{\neq}}_{Y_T}^2 \\
&\qquad  \lesssim e^{-c_0\nu t}\nu^{-1}\int_0^{\nu t} (\nu t-s)^{-1/2} \underbrace{e^{s}e^{-c\lambda_\nu\nu^{-1} s}}_{\leq e^{-c\lambda_\nu\nu^{-1} s/2}} \, ds\,  \norm{f_{\neq}}_{Y_T}^2 
%
\end{aligned}
\end{equation}
for $\nu$ sufficiently small. Then, using the estimate (see Appendix \ref{app:p_poincare})
\begin{equation}
\label{eq:thebthing}
\int_0^t (t-s)^{-a} e^{-bs} \, ds \les_{a} b^{a-1} \, , \quad\quad \forall a \in (0,1) \, ,   b > 0 \, ,  t > 0 \, ,
\end{equation}
with $a=1/2$, $b = c\lambda_\nu \nu^{-1}/2$, we have 
\begin{equation}
\norm{B_0(f_{\neq},f_{\neq})(\cdot,t)}_{H^2_xL^2_p} \lesssim e^{-c_0\nu t}\nu^{-1/2}\lambda_\nu^{-1/2}\,\norm{f_{\neq}}_{Y_T}^2\,.
\end{equation}
In particular, again using \eqref{eq:heatp_ests}, $f_0$ satisfies
\begin{equation}\label{eq:f0X0est}
\norm{f_0}_{X_T} \lesssim \norm{f^{\rm in}_0}_{L^2_p}+\nu^{-1/2}\lambda_\nu^{-1/2}C_0^2\varepsilon^2 \, .
\end{equation}

We next consider the evolution of $f_{\neq}$. Recalling the definition of $S(t)$ in \eqref{eq:semigrp} as the semigroup associated with   $L_{\nu,\kappa}^{\overline\psi,\iota}$ from \eqref{eq:Lnukappa}, we may write $f_{\neq}$ using Duhamel's formula as
\begin{equation}\label{eq:fneq_duhamel}
f_{\neq}(\cdot,t) = S(t)f_{\neq}^{\rm in}(\cdot) + B_1(f_{\neq},f_{\neq})(\cdot,t)+B_2(f_{\neq},f_{\neq})(\cdot,t) + B_2(f_0,f_{\neq})(\cdot,t)
\end{equation}
where
\begin{align}
B_1(f,g)(\cdot,t)&= -\int_0^t S(t-s)\, \div_x(\bu[g]f)(\cdot,s)\, ds \\
B_2(f,g)(\cdot,t) &= -\PP_{\neq}\int_0^t S(t-s)\,\divp [({\bf I}-\bp\otimes\bp)(\nabla_x\bu[g] \bp)f ](\cdot, s) \,ds\,.
\end{align}

We proceed to estimate each term in \eqref{eq:fneq_duhamel}. In addition to the smoothing-in-$\bp$ estimate of Lemma \ref{lem:smoothp}, we will make use of the following more standard semigroup estimates, which follow directly from Corollary \ref{cor:psibar} with $c_{\neq}=\frac{\alpha_0}{4}$: 
\begin{equation}\label{eq:St_ests}
\|S(\cdot,t)\|_{L^2_{x,p}\to L^2_{x,p}} \lesssim e^{-c_{\neq}\lambda_\nu t}\,, \quad 
\|S(\cdot,t)\div_x\|_{L^2_{x,p}\to L^2_{x,p}} \lesssim (\kappa t)^{-1/2}e^{-c_{\neq}\kappa t}e^{-c_{\neq}\lambda_\nu t}\,.
\end{equation}
Note that the second bound in \eqref{eq:St_ests} uses the decomposition \eqref{eq:add_in_kappa}.

Using \eqref{eq:St_ests}, we have 
\begin{equation}\label{eq:B1fneq}
\begin{aligned}
\norm{B_1(f_{\neq},f_{\neq})(\cdot,t)}_{H^2_xL^2_p}
&\le \int_0^t \norm{S(t-s)\, \div_x(\bu[f_{\neq}]f_{\neq})(\cdot,s)}_{H^2_xL^2_p}\, ds \\
&\le \int_0^t \kappa^{-1/2}(t-s)^{-1/2}e^{-c_{\neq}\lambda_\nu(t-s)} \norm{\bu[f_{\neq}]f_{\neq}}_{H^2_xL^2_p} \, ds \\
&\lesssim \kappa^{-1/2}\int_0^t (t-s)^{-1/2}e^{-c_{\neq}\lambda_\nu (t-s)}e^{-2\delta_{\neq}\lambda_\nu s} \, ds \, \norm{f_{\neq}}_{Y_T}^2 \\
&\lesssim \kappa^{-1/2}\lambda_\nu^{-1/2}e^{-\delta_{\neq}\lambda_\nu t}(C_0\varepsilon)^2,
\end{aligned}
\end{equation}
where we have also used  $\norm{\bu[f_{\neq}]}_{H^3_xL^2_p}\les\norm{f_{\neq}}_{H^2_xL^2_p}$. 

Next, applying Lemma \ref{lem:smoothp} with $\bm{g}=[({\bf I}-\bp\otimes\bp)(\nabla_x\bu[f_{\neq}] \bp)f_{\neq} ]$, we obtain 
\begin{equation}\label{eq:B2fneq}
\begin{aligned}
\norm{B_2(f_{\neq},f_{\neq})(\cdot,t)}_{H^2_xL^2_p}^2  
&\lesssim \nu^{-1}\int_0^t e^{-2c_{\neq}\lambda_\nu(t-s)} \norm{\nabla_x\bu[f_{\neq}] \,f_{\neq}}_{H^2_xL^2_p}^2 \, ds \\
&\lesssim \nu^{-1}\int_0^t e^{-2c_{\neq}\lambda_\nu(t-s)} e^{-4\delta_{\neq}\lambda_\nu s} \, ds \norm{f_{\neq}}_{Y_T}^4\\
&\lesssim \nu^{-1}\lambda_\nu^{-1}e^{-2\delta_{\neq}\lambda_\nu t}(C_0\varepsilon)^4\,.
\end{aligned}
\end{equation}

Finally, using Lemma \ref{lem:smoothp} along with the estimate \eqref{eq:f0X0est}, we have 
\begin{equation}\label{eq:B2f0}
\begin{aligned}
\norm{B_2(f_0,f_{\neq})(\cdot,t)}_{H^2_xL^2_p}^2 
&\lesssim \nu^{-1}\int_0^t e^{-2c_{\neq}\lambda_\nu(t-s)} \norm{\nabla_x\bu[f_{\neq}] \,f_0}_{H^2_xL^2_p}^2 \, ds \\
&\lesssim \nu^{-1}\int_0^t e^{-2c_{\neq}\lambda_\nu(t-s)} e^{-2\delta_{\neq}\lambda_\nu s}e^{-2\delta_0\nu s} \, ds \norm{f_{\neq}}_{Y_T}^2\norm{f_0}_{X_T}^2\\
&\lesssim \nu^{-1}\lambda_\nu^{-1}e^{-2\delta_{\neq}\lambda_\nu t}(C_0\varepsilon)^2(\norm{f^{\rm in}_0}_{L^2_p}+ \nu^{-1/2}\lambda_\nu^{-1/2}C_0^2\varepsilon^2  )^2 \, .
\end{aligned}
\end{equation}

Combining the bounds \eqref{eq:B1fneq}, \eqref{eq:B2fneq}, and \eqref{eq:B2f0}, and using \eqref{eq:St_ests} to bound the evolution of $f_{\neq}^{\rm in}$, we thus obtain 
\begin{equation}
\begin{aligned}
\norm{f_{\neq}}_{Y_T} &\le \norm{f_{\neq}^{\rm in}}_{H^2_xL^2_p} + \norm{B_1(f_{\neq},f_{\neq})}_{Y_T}+ \norm{B_2(f_{\neq},f_{\neq})}_{Y_T} + \norm{B_2(f_0,f_{\neq})}_{Y_T} \\
&\lesssim \big[1 + \kappa^{-1/2}\lambda_\nu^{-1/2}C_0^2\varepsilon+ \nu^{-1/2}\lambda_\nu^{-1/2}C_0^2\varepsilon \\
&\qquad + \nu^{-1/2}\lambda_\nu^{-1/2}C_0\big(\norm{f^{\rm in}_0}_{L^2_p}+ \nu^{-1/2}\lambda_\nu^{-1/2}C_0^2\varepsilon^2 \big)\big]\, \varepsilon \, .
\end{aligned}
\end{equation}
To close the bootstrap argument \eqref{eq:bootstrap_tbd}, we need 
\begin{equation}
\begin{aligned}
&C_0^{-1} + \kappa^{-1/2}\lambda_\nu^{-1/2}C_0\varepsilon+ \nu^{-1/2}\lambda_\nu^{-1/2}C_0\varepsilon \\
&\qquad + \nu^{-1/2}\lambda_\nu^{-1/2}\big(\norm{f^{\rm in}_0}_{L^2_p}+\nu^{-1/2}\lambda_\nu^{-1/2}C_0^2\varepsilon^2 \big) \ll 1\,,
\end{aligned}
\end{equation}
which can be satisfied if $C_0\gg1$, $\varepsilon\ll \min(\kappa^{1/2},\nu^{1/2})\lambda_\nu^{1/2}$, and $\norm{f^{\rm in}_0}_{L^2_p}\ll \nu^{1/2}\lambda_\nu^{1/2}$. 
\end{proof}

\appendix
\numberwithin{theorem}{section}

\section{Nondimensionalization}\label{app:nondim}
To facilitate comparison with results from the applied and computational literature, we comment here on our choice of nondimensionalization for equations \eqref{eq:smoluchowski}-\eqref{eq:activestress}. 
The fully dimensional version of the kinetic model is given by 
\begin{align}
\p_t \psi +V_0\bp\cdot\nabla\psi +\bu\cdot\nabla\psi +\div_p[({\bf I}-\bp\otimes\bp)\nabla\bu \bp \, \psi] &= d_{\rm r} \Delta_p\psi + d_{\rm t} \Delta_x\psi   \label{eq:dimeq1}  \\
-\mu \Delta \bu + \nabla q = \div {\bm \Sigma}\,, \; \div\bu &=0
\label{eq:dimeq2}
\\
 {\bm \Sigma}=\sigma_0\int_{S^{d-1}} \psi(\bx,\bp,t) \bp\otimes\bp &\, d\bp\,,   \label{eq:dimeq3}
\end{align}
where $V_0$ is the average swimming speed of the particles, $d_{\rm r}$ and $d_{\rm t}$ are the rotational and translational diffusion coefficients, $\mu$ is the fluid viscosity, and $\sigma_0$ is the (signed) active stress magnitude. The above system is considered on a $d$-dimensional periodic box with side length $L$, that is, $\R^d/L\Z^d$.

We nondimensionalize \eqref{eq:dimeq1}-\eqref{eq:dimeq3} according to 
\begin{equation}\label{eq:nondim}
 \bu^* = \frac{1}{V_0}\bu\,, \quad \bx^*= \frac{2\pi }{L}\bx\,, \quad t^*=\frac{2\pi V_0}{L}t\,, \quad \psi^* = \frac{L\abs{\sigma_0}}{2\pi \mu V_0}\psi \,.
 \end{equation}
Here, the (nondimensionalized) particle number density
\begin{equation}\label{eq:barpsidef} 
\overline\psi = \frac{1}{L^d}\int_{\T^d}\int_{S^{d-1}}\psi^* \, d\bp \,d\bx
\end{equation}
is a free parameter. 
Note that we may write
\begin{equation}
\overline\psi = \frac{L\abs{\sigma_0}n_\psi}{2\pi \mu V_0}\, ,
\end{equation}
where $n_\psi$ is the dimensional particle number density, given by \eqref{eq:barpsidef} with $\psi$ in place of $\psi^*$. The dimensionless rotational and translational diffusion coefficients $\nu$, $\kappa$ in \eqref{eq:smoluchowski} are given by 
\begin{equation}\label{eq:nu}
    \nu = \frac{d_{\rm r}L}{2\pi V_0}\,, \quad \kappa = \frac{d_{\rm t}2\pi}{L V_0}\,.
\end{equation}


\section{Strong solution theory}
\label{app:strongsols}

Let $T>0$, $\Omega^d = \T^d$ or $\R^d$, and $d=2,3$. Let $\nu > 0$ and $\kappa \geq 0$.
\begin{definition}[Strong solution]
\label{def:strongsol}
A non-negative function $\psi : \Omega^d \times S^{d-1} \times [0,T] \to [0,+\infty)$ is a \emph{strong solution} to~\eqref{eq:smoluchowski}-\eqref{eq:activestress} on $ \Omega^d \times S^{d-1} \times (0,T)$ with initial data $\psi^{\rm in} \in H^2_x L^2_p \cap L^1_x L^1_p(\Omega^d \times S^{d-1})$ if the following requirements are satisfied. Namely,
\begin{equation}
    \text{(i)} \quad \psi \in C([0,T];H^2_x L^2_p) \cap L^2_t H^2_x H^1_p \cap C([0,T];L^1_x L^1_p) \, ,
\end{equation}
(ii) the PDE \eqref{eq:smoluchowski}-\eqref{eq:activestress} is satisfied in the sense of distributions on $ \Omega^d \times S^{d-1} \times (0,T)$, and (iii)
$\| \psi(\cdot,t) - \psi^{\rm in} \|_{L^2} \to 0$ as $t \to 0^+$.
\end{definition}

\begin{theorem}[Strong solution theory]
\label{thm:strongsol}
\emph{(Existence)}. For all $0 \leq \psi^{\rm in} \in H^2_x L^2_p \cap L^1_x L^1_p$, there exists $\bar{T} = \bar{T}(\| \psi^{\rm in} \|_{H^2_x L^2_p},\nu) > 0$ (the \emph{guaranteed existence time}) and a strong solution $\psi$ to~\eqref{eq:smoluchowski}-\eqref{eq:activestress} on $\Omega^d \times S^{d-1} \times (0,\bar{T})$ with initial data $\psi^{\rm in}$.

\emph{(Uniqueness)}. If $\psi, \tilde{\psi}$ are two strong solutions on $\Omega^d \times S^{d-1} \times (0,T)$, then $\psi \equiv \tilde{\psi}$.
\end{theorem}

\begin{proof}
\emph{(Existence)}. Let $\varepsilon > 0$. For $0 \leq \varphi \in C^\infty_0(\R^d)$ with $\int \varphi = 1$, define the mollification $(g)_\varepsilon := \varepsilon^{-d} g \ast  \varphi(\cdot/\varepsilon)$.

Consider the \emph{mollified equations}
\begin{equation}
    \p_t \psi + \bp\cdot\nabla_x\psi + (\bu)_\varepsilon \cdot\nabla_x\psi + \div_p\left[({\bf I}-\bp\otimes\bp)(\nabla(\bu)_{\varepsilon} \bp) \psi \right] = \nu \Delta_p\psi +\kappa\Delta_x\psi \, ,
\end{equation}
subject to the constitutive law~\eqref{eq:stokes}-\eqref{eq:activestress} for $\bu$. We also mollify the initial condition $\psi(\cdot,0) = (\psi^{\rm in})_\varepsilon$.

Define the \emph{Picard iterates} $\psi_n$ inductively: $\psi_{-1}=0$ and $\psi_n$, $n=0,1,2,\dots$, is the solution to the following linear advection-diffusion equation with spatially smooth coefficients:
\begin{equation}
\p_t \psi_n+\bp\cdot\nabla_x \psi_n+\bu[\psi_{n-1}]\cdot\nabla_x\psi_n + \div_p[({\bf I}-\bp\otimes\bp)\nabla\bu[\psi_{n-1}]\bp \psi_n] = \nu\Delta_p\psi_n +\kappa\Delta_x\psi_n\, ,
\end{equation}
with smooth initial condition $\psi_n(\cdot,0) = (\psi^{\rm in})_\varepsilon$. Clearly, $\psi_n \geq 0$ and mass is conserved, i.e., $\int \psi_n(\bx,\bp,t) \, d\bx \, d\bp = \int \psi^{\rm in}(\bx,\bp) \, d\bx \, d\bp$ for all $t > 0$.

\emph{1. \emph{A priori} estimates.} To begin, we record the following \emph{a priori} energy estimates for a smooth solution $\psi$ to~\eqref{eq:smoluchowski}:
\begin{equation}\label{eq:psi1}
\begin{aligned}
&\frac{1}{2}\frac{d}{dt}\norm{\psi}_{L^2_{x,p}}^2 +\frac{1}{2}\int_{x,p}\div_p[({\bf I}-\bp\otimes\bp)(\nabla\bu\bp)]\,\psi^2\,d\bx d\bp \\
&\qquad= -\nu\norm{\nabla_p\psi}_{L^2_{x,p}}^2 -\kappa\norm{\nabla_x\psi}_{L^2_{x,p}}^2\,,
\end{aligned}
\end{equation}
\begin{equation}\label{eq:psi2}
\begin{aligned}
&\frac{1}{2}\frac{d}{dt}\norm{\nabla_x\psi}_{L^2_{x,p}}^2 + \int_{x,p} (\nabla\bu\nabla_x\psi)\cdot\nabla_x\psi \, d\bx d\bp \\
&\qquad
-\int_{x,p} ({\bf I}-\bp\otimes\bp)(\nabla^2\bu \bp \psi):(\nabla_p\nabla_x\psi)\,d\bx d\bp \\
&\qquad 
+\frac{1}{2}\int_{x,p}\div_p[({\bf I}-\bp\otimes\bp)(\nabla\bu\bp)]\,\abs{\nabla_x\psi}^2\,d\bx d\bp\\
&\qquad\qquad = -\nu\norm{\nabla_p\nabla_x\psi}_{L^2_{x,p}}^2 -\kappa\norm{\nabla_x^2\psi}_{L^2_{x,p}}^2 \,,
\end{aligned}
\end{equation}
\begin{equation}\label{eq:psi3}
\begin{aligned}
&\frac{1}{2}\frac{d}{dt}\norm{\nabla_x^2\psi}_{L^2_{x,p}}^2 + \int_{x,p}(\nabla^2\bu\cdot\nabla_x\psi+\nabla\bu\cdot\nabla_x^2\psi):\nabla_x^2\psi\, d\bx d\bp \\
&\qquad -\int_{x,p} ({\bf I}-\bp\otimes\bp)[\nabla^3\bu \bp \psi+2\nabla^2\bu \bp \otimes\nabla_x\psi]:(\nabla_p\nabla_x^2\psi)\,d\bx d\bp \\
&\qquad 
+\frac{1}{2}\int_{x,p}\div_p[({\bf I}-\bp\otimes\bp)(\nabla\bu\bp)]\,\abs{\nabla_x^2\psi}^2\,d\bx d\bp\\
&\qquad\qquad = -\nu\norm{\nabla_p\nabla_x^2\psi}_{L^2_{x,p}}^2 -\kappa\norm{\nabla_x^3\psi}_{L^2_{x,p}}^2\,.
\end{aligned}
\end{equation}
The final term on the left-hand side of each of \eqref{eq:psi1}-\eqref{eq:psi3} arises from integration by parts twice in $\bp$.

From \eqref{eq:psi1}-\eqref{eq:psi3}, we obtain the following bounds:
\begin{equation}
\frac{d}{dt}\norm{\psi}_{L^2_{x,p}}^2 \lesssim \norm{\nabla\bu}_{L^\infty_x}\norm{\psi}_{L^2_{x,p}}^2 -\nu\norm{\nabla_p\psi}_{L^2_{x,p}}^2 -\kappa\norm{\nabla_x\psi}_{L^2_{x,p}}^2\,,
\end{equation}
\begin{equation}
\begin{aligned}
\frac{d}{dt}\norm{\nabla_x\psi}_{L^2_{x,p}}^2 &\lesssim \norm{\nabla\bu}_{L^\infty_x}\norm{\nabla_x\psi}_{L^2_{x,p}}^2 +
\norm{\nabla^2\bu}_{L^2_x}\norm{\psi}_{L^\infty_xL^2_p}\norm{\nabla_p\nabla_x\psi}_{L^2_{x,p}} \\
&\quad 
  -\nu\norm{\nabla_p\nabla_x\psi}_{L^2_{x,p}}^2 -\kappa\norm{\nabla_x^2\psi}_{L^2_{x,p}}^2\,,
\end{aligned}
\end{equation}
\begin{equation}
\begin{aligned}
\frac{d}{dt}\norm{\nabla_x^2\psi}_{L^2_{x,p}}^2 &\lesssim \norm{\nabla^2\bu}_{L^4_x}\norm{\nabla_x\psi}_{L^4_xL^2_p}\norm{\nabla_x^2\psi}_{L^2_{x,p}} 
 +\norm{\nabla^3\bu}_{L^2_x}\norm{\psi}_{L^\infty_xL^2_p}\norm{\nabla_p\nabla_x^2\psi}_{L^2_{x,p}} \\ 
&\quad +\norm{\nabla\bu}_{L^\infty_x}\norm{\nabla_x^2\psi}_{L^2_{x,p}}^2 
+\norm{\nabla^2\bu}_{L^4_x}\norm{\nabla_x\psi}_{L^4_xL^2_p}\norm{\nabla_p\nabla_x^2\psi}_{L^2_{x,p}} \\
&\quad
 -\nu\norm{\nabla_p\nabla_x^2\psi}_{L^2_{x,p}}^2 -\kappa\norm{\nabla_x^3\psi}_{L^2_{x,p}}^2\,.
\end{aligned}
\end{equation}
Furthermore, we note that $\nabla\bu$ satisfies 
 \begin{equation}\label{eq:gradu_bd_psi}
 \norm{\nabla\bu}_{H^2_x} \les \norm{\psi}_{H^2_xL^2_p}\, .
 \end{equation}
 Using Sobolev embedding and the bound \eqref{eq:gradu_bd_psi} for $\bu$, and using Young's inequality to absorb the terms $\norm{\nabla_p\nabla_x^2\psi}_{L^2_{x,p}}$ into the rotational diffusion term, we obtain 
\begin{equation}
\begin{aligned}
\frac{d}{dt}\norm{\psi}_{H^2_xL^2_p}^2 + \nu \| \nabla_p \psi \|_{H^2_xL^2_p}^2  &\lesssim_\nu \norm{\psi}_{H^2_xL^2_p}^3 + \norm{\psi}_{H^2_xL^2_p}^4  \,.
\end{aligned}
\end{equation}
By Gronwall's inequality, there exists $\bar{T} = \bar{T}(\| \psi^{\rm in} \|_{H^2_x L^2_p}, \nu) > 0$ with the property 
\begin{equation}
    \| \psi \|_{L^\infty_t H^2_x L^2_p(\Omega^d \times S^{d-1} \times (0,\bar{T}))}^2 + \nu \| \nabla_p \psi \|_{L^2_t H^2_x L^2_p(\Omega^d \times S^{d-1} \times (0,\bar{T}))}^2 \les_\nu \| \psi^{\rm in} \|_{H^2_x L^2_p}^2 \, 
\end{equation}
While the above \emph{a priori} estimates were for smooth $\bu$, analogous computations for the Picard iterates $\psi_n$ to the mollified equations produce the inequality
\begin{equation}
\begin{aligned}
\frac{d}{dt}\norm{\psi_n}_{H^2_xL^2_p}^2 + \nu \| \nabla_p \psi_n \|_{H^2_xL^2_p}^2  &\lesssim_\nu (\norm{\psi_{n-1}}_{H^2_xL^2_p} + \norm{\psi_{n-1}}_{H^2_xL^2_p}^2 ) \norm{\psi_n}_{H^2_xL^2_p}^2
\end{aligned}
\end{equation}
and, consequently, the \emph{a priori} estimates
\begin{equation}
    \label{eq:psiH2xL2p}
    \| \psi_n \|_{L^\infty_t H^2_x L^2_p(\Omega^d \times S^{d-1} \times (0,\bar{T}))}^2 + \nu \| \nabla_p \psi_n \|_{L^2_t H^2_x L^2_p(\Omega^d \times S^{d-1} \times (0,\bar{T}))}^2 \les_\nu \| \psi^{\rm in} \|_{H^2_x L^2_p}^2 \, .
\end{equation}

\emph{2. Contraction}. Letting $\bu_n=\bu[\psi_n]$, the difference in Picard iterates $\psi_n-\psi_{n-1}$ satisfies 
\begin{equation}
\begin{aligned}
\p_t(&\psi_n-\psi_{n-1})+\bp\cdot\nabla_x(\psi_n-\psi_{n-1})
+(\bu_{n-1}-\bu_{n-2})\cdot\nabla_x\psi_n \\ &+\bu_{n-2}\cdot\nabla_x(\psi_n-\psi_{n-1})
+ \div_p[({\bf I}-\bp\otimes\bp)\nabla(\bu_{n-1}-\bu_{n-2})\bp \psi_n] \\
& + \div_p[({\bf I}-\bp\otimes\bp)\nabla\bu_{n-2}\bp (\psi_n-\psi_{n-1})] \\
 &= \nu\Delta_p(\psi_n-\psi_{n-1}) +\kappa\Delta_x(\psi_n-\psi_{n-1})\,.
\end{aligned}
\end{equation}
Multiplying by $\psi_n-\psi_{n-1}$ and integrating by parts, we have
\begin{equation}
\begin{aligned}
&\frac{d}{dt}\norm{\psi_n-\psi_{n-1}}_{L^2_{x,p}}^2 \lesssim
\left| \int_{x,p} (\bu_{n-1}-\bu_{n-2})\cdot\nabla_x\psi_n (\psi_n - \psi_{n-1}) \, d\bx \,  d\bp \right| \\
&\quad+\norm{\nabla\bu_{n-2}}_{L^\infty_x}\norm{\psi_n-\psi_{n-1}}_{L^2_{x,p}}^2\\
 &\quad +\norm{\nabla(\bu_{n-1}-\bu_{n-2})}_{L^2_x}\norm{\psi_n}_{L^\infty_xL^2_p}\norm{\nabla_p(\psi_n-\psi_{n-1})}_{L^2_{x,p}} \\
 &\quad -\nu\norm{\nabla_p(\psi_n-\psi_{n-1})}_{L^2_{x,p}}^2 -\kappa\norm{\nabla_x(\psi_n-\psi_{n-1})}_{L^2_{x,p}}^2 \,.
\end{aligned}
\end{equation}
In dimension three and on $\T^2$, we have
\begin{equation}
\begin{aligned}
    &\left| \int_{x,p} (\bu_{n-1}-\bu_{n-2})\cdot\nabla_x\psi_n (\psi_n - \psi_{n-1}) \, d\bx \,  d\bp \right| \\
    &\qquad\les \| \bu_{n-1}-\bu_{n-2} \|_{L^6_x} \| \nabla_x \psi_n \|_{L^3} \| \psi_n - \psi_{n-1} \|_{L^2} \, .
    \end{aligned}
\end{equation}
The approach on $\R^2$ is more subtle, and we return to it later. By the \emph{a priori} bound \eqref{eq:psiH2xL2p}, we have that $\norm{\psi_n}_{H^2_xL^2_p}$ is controlled for sufficiently small time. Furthermore, by \eqref{eq:gradu_bd_psi}, we may bound $\norm{\nabla(\bu_{n-1}-\bu_{n-2})}_{L^2_x}$ by $\norm{\psi_{n-1}-\psi_{n-2}}_{L^2_x}$. Using Young's inequality to absorb the $\norm{\nabla_p(\psi_n-\psi_{n-1})}_{L^2_{x,p}}$ term into the rotational diffusion term, we obtain 
\begin{equation}\label{eq:diff_ineq}
\frac{d}{dt}\norm{\psi_n-\psi_{n-1}}_{L^2_{x,p}}^2 \lesssim_{\nu}
\norm{\psi_{n-1}-\psi_{n-2}}_{L^2_{x,p}}^2 + \norm{\psi_n-\psi_{n-1}}_{L^2_{x,p}}^2\,.
\end{equation}
This is a differential inequality of the form
\begin{equation}
\frac{d}{dt}A_n \le M A_{n-1}+MA_n\,,
\end{equation}
where $M>0$ is independent of $n$ and each $A_n(t) \geq 0$. By Gr\"onwall's inequality,
\begin{equation}
A_n(t)\le M\int_0^te^{M(t-s)}A_{n-1}(s)\, ds\, .
\end{equation}
In particular, for sufficiently small $T$, we have $\sup_{t \in (0,T)} A_n\le \frac{1}{2} \sup_{t \in (0,T)} A_{n-1}$. 
Thus, we have strong convergence of the Picard iterates $\psi_n \to \psi_\infty$ in $C_t L^2_{x,p}$ for short times. By lower semicontinuity, the \emph{a priori} estimates~\eqref{eq:psiH2xL2p} persist as $n \to +\infty$. Then we allow the mollification parameter $\varepsilon \to 0^+$. That the solutions belong to $C([0,T];H^2_x L^2_p \cap L^1_x L^1_p)$ can be justified after the fact via the linear theory. Looking ahead, once uniqueness is known, we can extend the solution to its maximal time of existence, for which a lower bound is $\bar{T}$ from Step 1.

\emph{2'. Contraction in $\R^2$}. We now address the two-dimensional setting. We consider also estimates satisfied by $\| \psi_n - \psi_{n-1} \|_{L^1}$, namely (with $q=2+\varepsilon$, $0 < \varepsilon \ll 1$),
\begin{equation}
\begin{aligned}
    &\frac{d}{dt} \| \psi_n - \psi_{n-1} \|_{L^1} \\
    &\quad \leq \| \bu_{n-1} - \bu_{n-2} \|_{L^{q}} \| \nabla_x \psi_n \|_{L^{q'}} + \| \div_p [(\bI - \bp \otimes \bp) \nabla (\bu_{n-1} - \bu_{n-2}) \bp \psi_n] \|_{L^1} \\
    &\quad \les (\| \psi_{n-1} - \psi_{n-2} \|_{L^1} + \| \psi_{n-1} - \psi_{n-2} \|_{L^2}) \| \psi_n \|_{L^1_x L^1_p \cap H^2_x L^2_p} \\
    &\quad\quad+ \| \psi_{n-1} - \psi_{n-2} \|_{L^2} \| \psi_n \|_{L^2_x H^1_p}  \, .
    \end{aligned}
\end{equation}
To complete the $L^2$ estimate, we write
\begin{equation}
\begin{aligned}
    &\left| \int_{x,p} (\bu_{n-1}-\bu_{n-2})\cdot\nabla_x\psi_n (\psi_n - \psi_{n-1}) \, d\bx \,  d\bp \right| \\
    &\quad \les \| \bu_{n-1}-\bu_{n-2} \|_{L^4_x} \| \nabla_x \psi_n \|_{L^4} \| \psi_n - \psi_{n-1} \|_{L^2} \\
    &\quad \les \| \nabla_x \psi_n \|_{L^4} ( \| \psi_{n-1} - \psi_{n-2} \|_{L^1} +  \| \psi_{n-1} - \psi_{n-2} \|_{L^2} )\| \psi_n - \psi_{n-1} \|_{L^2} \, .
    \end{aligned}
\end{equation}
We now conclude via the differential inequality for $\| \psi_n - \psi_{n-1} \|_{L^1}^2 + \| \psi_n - \psi_{n-1} \|_{L^2}^2$.

\emph{(Uniqueness)}. Given two strong solutions $\psi$, $\wt\psi$ on $\Omega^d \times S^{d-1} \times (0,T)$, we consider again the differential inequality satisfied by $\| \psi - \wt\psi \|_{L^2}^2$ (or, when $\Omega^d = \R^2$, $\| \psi - \wt\psi \|_{L^1}^2 + \| \psi - \wt\psi \|_{L^2}^2$). Estimates analogous to~\eqref{eq:diff_ineq} grant uniqueness.
\end{proof}


\section{Proof of auxiliary lemmas}\label{app:p_poincare}

First, we prove the Poincar\'e-type inequality~\eqref{eq:p_poincare}.

\begin{proof}[Proof of Lemma~\ref{lem:p_poincare}]
Without loss of generality, we may choose $\bk=k\be_1$ with $k = |\bk|$ and define $\bp=p_1\be_1+p_2\be_2+p_3\be_3$ with $p_1^2+p_2^2+p_3^2=1$.

Away from $p_1^2=1$, we may bound the full spatial gradient $k\norm{f_k}_{L^2}$ in terms of the spatial gradient projected off the unit sphere, i.e., $k\norm{(\be_1-p_1\bp)f_k}_{L^2}=k\norm{\sqrt{1-p_1^2}f_k}_{L^2}$. 
Near $p_1^2=1$, we must instead control $k\norm{f_k}_{L^2}$ using the orientational gradient $\norm{\nabla_p f_k}_{L^2}$. 

We thus define a cutoff function on the unit sphere:
\begin{equation}\label{eq:cutoff_sph}
\varphi_\delta = \begin{cases}
1, & 1-p_1^2 < \delta \\
0, & 1-p_1^2 > 2\delta
\end{cases}
\end{equation}
for some $0<\delta<\frac{1}{2}$, with smooth decay between. Then, away from $p_1^2=1$, we have 
\begin{equation}\label{eq:bd_p1small}
\norm{f_k(1-\varphi_\delta)}_{L^2}^2 \le \frac{1}{\delta}\norm{\sqrt{1-p_1^2} f_k}_{L^2}^2\,.
\end{equation}
Near $p_1^2=1$, we have
\begin{equation}\label{eq:bd_p1big} 
\begin{aligned}
\norm{f_k\varphi_\delta}_{L^2}^2 &\le c\delta\norm{\nabla_p(f_k\varphi_\delta)}_{L^2}^2 
\le c\delta \bigg(\norm{(\nabla_p f_k)\varphi_\delta}_{L^2}^2 +\norm{ f_k(\nabla_p\varphi_\delta)}_{L^2}^2  \bigg)\\
&\le c\delta \bigg(\norm{\nabla_p f_k}_{L^2}^2 + \frac{1}{\delta^2}\norm{\sqrt{1-p_1^2} f_k}_{L^2}^2  \bigg) \,,
\end{aligned}
\end{equation}
where we have used that $|\nabla_p\varphi_\delta| \les \frac{1}{\delta^{1/2}}$ is supported within the strip $\delta\le 1-p_1^2 \le 2\delta$. 
Together, we obtain
\begin{equation}\label{eq:p_poincare0}
k^2\norm{f_k}_{L^2}^2 \les \frac{k^2}{\delta}\norm{\sqrt{1-p_1^2} f_k}_{L^2}^2 + \delta k^2 \norm{\nabla_p f_k}_{L^2}^2\,,
\end{equation}
and, choosing $\delta=\frac{1}{4}\nu^{1/2}k^{-1/2}$, we obtain Lemma \ref{lem:p_poincare}. 
\end{proof}

Second, we justify~\eqref{eq:thebthing}.

\begin{proof}[Proof of~\eqref{eq:thebthing}]
We have
\begin{equation}
\begin{aligned}
\left( \int_0^{t/2} + \int_{t/2}^t \right) (t-s)^{-a} e^{-bs} \, ds &\les t^{-a} \int_0^{t/2} e^{-bs} \, ds + \int_{t/2}^t (t-s)^{-a} \, ds \,  e^{-bt/2}  \\
& \les_a t^{-a} b^{-1} (1-e^{-bt/2}) + t^{1-a} e^{-bt/2} \, .
\end{aligned}
\end{equation}
For $t \leq b^{-1}$, we have
\begin{equation}
  t^{-a} b^{-1} (1-e^{-bt/2}) \les t^{1-a} \les b^{a-1} \, ,
\end{equation}
since $1-e^{-bt/2} \les bt$, and, since $e^{-bt/2} \leq 1$,
\begin{equation}
t^{1-a} e^{-bt/2} \les b^{a-1} \, .
\end{equation}
For $t \geq b^{-1}$, we have
\begin{equation}
t^{-a} b^{-1} (1-e^{-bt/2}) \les b^{a-1} \, ,
\end{equation}
since $1-e^{-bt/2} \leq 1$, and
\begin{equation}
t^{1-a} e^{-bt/2} \les b^{a-1} (bt)^{1-a} e^{-bt/2} \leq b^{a-1} \, ,
\end{equation}
since $(bt)^{1-a} e^{-bt/2} \les 1$.
\end{proof}

\subsection*{Acknowledgments}
 DA was supported by NSF Postdoctoral Fellowship Grant No.\ 2002023 and Simons Foundation Grant No.\ 816048. LO acknowledges support from NSF Postdoctoral Fellowship DMS-2001959. We thank the anonymous referees for their valuable work.

\bibliographystyle{siamplain}
\bibliography{bibliography}

\begin{thebibliography}{10}

\bibitem{Albritton2022}
{\sc D.~Albritton, R.~Beekie, and M.~Novack}, {\em Enhanced dissipation and
  {H}\"{o}rmander's hypoellipticity}, Journal of Functional Analysis,  (2022),
  p.~109522, \url{https://doi.org/10.1016/j.jfa.2022.109522},
  \url{https://doi.org/10.1016/j.jfa.2022.109522}.

\bibitem{aris1956dispersion}
{\sc R.~Aris}, {\em On the dispersion of a solute in a fluid flowing through a
  tube}, Proceedings of the Royal Society of London. Series A. Mathematical and
  Physical Sciences, 235 (1956), pp.~67--77.

\bibitem{BaeTrivisa}
{\sc H.~Bae and K.~Trivisa}, {\em On the {D}oi model for the suspensions of
  rod-like molecules: global-in-time existence}, Commun. Math. Sci., 11 (2013),
  pp.~831--850, \url{https://doi.org/10.4310/CMS.2013.v11.n3.a8},
  \url{https://doi.org/10.4310/CMS.2013.v11.n3.a8}.

\bibitem{bcd}
{\sc H.~Bahouri, J.-Y. Chemin, and R.~Danchin}, {\em Fourier analysis and
  nonlinear partial differential equations}, vol.~343 of Grundlehren der
  Mathematischen Wissenschaften [Fundamental Principles of Mathematical
  Sciences], Springer, Heidelberg, 2011,
  \url{https://doi.org/10.1007/978-3-642-16830-7},
  \url{https://doi.org/10.1007/978-3-642-16830-7}.

\bibitem{beckchaudharywayne}
{\sc M.~Beck, O.~Chaudhary, and C.~E. Wayne}, {\em Rigorous justification of
  {T}aylor dispersion via center manifolds and hypocoercivity}, Arch. Ration.
  Mech. Anal., 235 (2020), pp.~1105--1149,
  \url{https://doi.org/10.1007/s00205-019-01440-2},
  \url{https://doi.org/10.1007/s00205-019-01440-2}.

\bibitem{beckwaynebar}
{\sc M.~Beck and C.~E. Wayne}, {\em Metastability and rapid convergence to
  quasi-stationary bar states for the two-dimensional {N}avier-{S}tokes
  equations}, Proc. Roy. Soc. Edinburgh Sect. A, 143 (2013), pp.~905--927,
  \url{https://doi.org/10.1017/S0308210511001478},
  \url{https://doi.org/10.1017/S0308210511001478}.

\bibitem{jacobmicheleshear}
{\sc J.~Bedrossian and M.~Coti~Zelati}, {\em Enhanced dissipation,
  hypoellipticity, and anomalous small noise inviscid limits in shear flows},
  Arch. Ration. Mech. Anal., 224 (2017), pp.~1161--1204,
  \url{https://doi.org/10.1007/s00205-017-1099-y},
  \url{https://doi.org/10.1007/s00205-017-1099-y}.

\bibitem{bedrossianmasmoudiinviscid}
{\sc J.~Bedrossian and N.~Masmoudi}, {\em Inviscid damping and the asymptotic
  stability of planar shear flows in the 2{D} {E}uler equations}, Publ. Math.
  Inst. Hautes \'{E}tudes Sci., 122 (2015), pp.~195--300,
  \url{https://doi.org/10.1007/s10240-015-0070-4},
  \url{https://doi.org/10.1007/s10240-015-0070-4}.

\bibitem{bedrossianmasmoudimouhot}
{\sc J.~Bedrossian, N.~Masmoudi, and C.~Mouhot}, {\em Landau damping:
  paraproducts and {G}evrey regularity}, Ann. PDE, 2 (2016), pp.~Art. 4, 71,
  \url{https://doi.org/10.1007/s40818-016-0008-2},
  \url{https://doi.org/10.1007/s40818-016-0008-2}.

\bibitem{bmv1}
{\sc J.~Bedrossian, N.~Masmoudi, and V.~Vicol}, {\em Enhanced dissipation and
  inviscid damping in the inviscid limit of the {N}avier-{S}tokes equations
  near the two dimensional {C}ouette flow}, Arch. Ration. Mech. Anal., 219
  (2016), pp.~1087--1159, \url{https://doi.org/10.1007/s00205-015-0917-3},
  \url{https://doi.org/10.1007/s00205-015-0917-3}.

\bibitem{BedrossianWang2019}
{\sc J.~Bedrossian and F.~Wang}, {\em The linearized {V}lasov and
  {V}lasov{\textendash}{F}okker{\textendash}{P}lanck equations in a uniform
  magnetic field}, Journal of Statistical Physics, 178 (2019), pp.~552--594,
  \url{https://doi.org/10.1007/s10955-019-02441-x},
  \url{https://doi.org/10.1007/s10955-019-02441-x}.

\bibitem{chaturvedi2021vlasov}
{\sc S.~Chaturvedi, J.~Luk, and T.~T. Nguyen}, {\em The
  {V}lasov--{P}oisson--{L}andau system in the weakly collisional regime}, arXiv
  preprint arXiv:2104.05692,  (2021).

\bibitem{chen2014existence}
{\sc X.~Chen, X.~Li, and J.-G. Liu}, {\em Existence and uniqueness of global
  weak solution to a kinetic model for the sedimentation of rod-like
  particles}, Communications in Mathematical Sciences, 12 (2014),
  pp.~1579--1601.

\bibitem{chen2013global}
{\sc X.~Chen and J.-G. Liu}, {\em Global weak entropy solution to
  {D}oi--{S}aintillan--{S}helley model for active and passive rod-like and
  ellipsoidal particle suspensions}, Journal of Differential Equations, 254
  (2013), pp.~2764--2802.

\bibitem{constantin2007regularity}
{\sc P.~Constantin, C.~Fefferman, E.~Titi, and A.~Zarnescu}, {\em Regularity of
  coupled two-dimensional nonlinear {F}okker-{P}lanck and {N}avier-{S}tokes
  systems}, Communications in mathematical physics, 270 (2007), pp.~789--811.

\bibitem{constantin2008global}
{\sc P.~Constantin and N.~Masmoudi}, {\em Global well-posedness for a
  {S}moluchowski equation coupled with {N}avier-{S}tokes equations in 2d},
  Communications in mathematical physics, 278 (2008), pp.~179--191.

\bibitem{constantin2010global}
{\sc P.~Constantin and G.~Seregin}, {\em Global regularity of solutions of
  coupled {N}avier-{S}tokes equations and nonlinear {F}okker {P}lanck
  equations}, Discrete and Continuous Dynamical Systems, 26 (2010),
  pp.~1185--1196.

\bibitem{zelati2022orientation}
{\sc M.~Coti~Zelati, H.~Dietert, and D.~G{\'e}rard-Varet}, {\em Orientation
  mixing in active suspensions}, arXiv preprint arXiv:2207.08431,  (2022).

\bibitem{michelepoiseuille}
{\sc M.~Coti~Zelati, T.~M. Elgindi, and K.~Widmayer}, {\em Enhanced dissipation
  in the {N}avier-{S}tokes equations near the {P}oiseuille flow}, Comm. Math.
  Phys., 378 (2020), pp.~987--1010,
  \url{https://doi.org/10.1007/s00220-020-03814-0},
  \url{https://doi.org/10.1007/s00220-020-03814-0}.

\bibitem{doi1981molecular}
{\sc M.~Doi}, {\em Molecular dynamics and rheological properties of
  concentrated solutions of rodlike polymers in isotropic and liquid
  crystalline phases}, J. Polym. Sci. B, 19 (1981), pp.~229--243.

\bibitem{doi1986theory}
{\sc M.~Doi and S.~F. Edwards}, {\em The Theory of Polymer Dynamics}, vol.~73,
  Oxford University Press, 1986.

\bibitem{frankel1989foundations}
{\sc I.~Frankel and H.~Brenner}, {\em On the foundations of generalized
  {T}aylor dispersion theory}, Journal of Fluid Mechanics, 204 (1989),
  pp.~97--119.

\bibitem{grenier2020landau}
{\sc E.~Grenier, T.~T. Nguyen, and I.~Rodnianski}, {\em Landau damping for
  analytic and {G}evrey data}, arXiv preprint arXiv:2004.05979,  (2020).

\bibitem{guolandau}
{\sc Y.~Guo}, {\em The {L}andau equation in a periodic box}, Comm. Math. Phys.,
  231 (2002), pp.~391--434, \url{https://doi.org/10.1007/s00220-002-0729-9},
  \url{https://doi.org/10.1007/s00220-002-0729-9}.

\bibitem{hohenegger2010stability}
{\sc C.~Hohenegger and M.~J. Shelley}, {\em Stability of active suspensions},
  Physical Review E, 81 (2010), p.~046311.

\bibitem{jeffery1922motion}
{\sc G.~B. Jeffery}, {\em The motion of ellipsoidal particles immersed in a
  viscous fluid}, Proceedings of the Royal Society of London. Series A,
  Containing papers of a mathematical and physical character, 102 (1922),
  pp.~161--179.

\bibitem{jiang2020coupled}
{\sc N.~Jiang, Y.-L. Luo, and T.-F. Zhang}, {\em Coupled self-organized
  hydrodynamics and {N}avier--{S}tokes models: local well-posedness and the
  limit from the self-organized kinetic-fluid models}, Archive for Rational
  Mechanics and Analysis, 236 (2020), pp.~329--387.

\bibitem{kanzler2021kinetic}
{\sc L.~Kanzler and C.~Schmeiser}, {\em Kinetic model for myxobacteria with
  directional diffusion}, arXiv preprint arXiv:2109.13184,  (2021).

\bibitem{koch2011collective}
{\sc D.~L. Koch and G.~Subramanian}, {\em Collective hydrodynamics of swimming
  microorganisms: living fluids}, Annual Review of Fluid Mechanics, 43 (2011),
  pp.~637--659.

\bibitem{la2019global}
{\sc J.~La}, {\em Global well-posedness of strong solutions of {D}oi model with
  large viscous stress}, Journal of Nonlinear Science, 29 (2019),
  pp.~1891--1917.

\bibitem{la2020diffusive}
{\sc J.~La}, {\em On diffusive 2d {F}okker--{P}lanck--{N}avier--{S}tokes
  systems}, Archive for Rational Mechanics and Analysis, 235 (2020),
  pp.~1531--1588.

\bibitem{lauga2020fluid}
{\sc E.~Lauga}, {\em The fluid dynamics of cell motility}, vol.~62, Cambridge
  University Press, 2020.

\bibitem{lions2007global}
{\sc P.-L. Lions and N.~Masmoudi}, {\em Global existence of weak solutions to
  some micro-macro models}, Comptes Rendus Mathematique, 345 (2007),
  pp.~15--20.

\bibitem{manela2003generalized}
{\sc A.~Manela and I.~Frankel}, {\em Generalized {T}aylor dispersion in
  suspensions of gyrotactic swimming micro-organisms}, Journal of Fluid
  Mechanics, 490 (2003), pp.~99--127.

\bibitem{masmoudi2018equations}
{\sc N.~Masmoudi}, {\em Equations for polymeric materials}, in Handbook of
  Mathematical Analysis in Mechanics of Viscous Fluids, Springer International
  Publishing, 2018, pp.~973--1005.

\bibitem{masmoudi2008global}
{\sc N.~Masmoudi, P.~Zhang, and Z.~Zhang}, {\em Global well-posedness for 2d
  polymeric fluid models and growth estimate}, Physica D: Nonlinear Phenomena,
  237 (2008), pp.~1663--1675.

\bibitem{mouhotvillani}
{\sc C.~Mouhot and C.~Villani}, {\em On {L}andau damping}, Acta Math., 207
  (2011), pp.~29--201, \url{https://doi.org/10.1007/s11511-011-0068-9},
  \url{https://doi.org/10.1007/s11511-011-0068-9}.

\bibitem{ohm2022weakly}
{\sc L.~Ohm and M.~J. Shelley}, {\em Weakly nonlinear analysis of pattern
  formation in active suspensions}, Journal of Fluid Mechanics, 942 (2022),
  \url{https://doi.org/10.1017/jfm.2022.392},
  \url{https://doi.org/10.1017/jfm.2022.392}.

\bibitem{otto2008continuity}
{\sc F.~Otto and A.~E. Tzavaras}, {\em Continuity of velocity gradients in
  suspensions of rod--like molecules}, Communications in Mathematical Physics,
  277 (2008), pp.~729--758.

\bibitem{RudinRealComplex}
{\sc W.~Rudin}, {\em Real and complex analysis}, McGraw-Hill Book Co., New
  York, third~ed., 1987.

\bibitem{saintillan2007orientational}
{\sc D.~Saintillan and M.~J. Shelley}, {\em Orientational order and
  instabilities in suspensions of self-locomoting rods}, Physical review
  letters, 99 (2007), p.~058102.

\bibitem{saintillan2008instabilitiesPRL}
{\sc D.~Saintillan and M.~J. Shelley}, {\em Instabilities and pattern formation
  in active particle suspensions: kinetic theory and continuum simulations},
  Physical Review Letters, 100 (2008), p.~178103.

\bibitem{saintillan2008instabilitiesPOF}
{\sc D.~Saintillan and M.~J. Shelley}, {\em Instabilities, pattern formation,
  and mixing in active suspensions}, Physics of Fluids, 20 (2008), p.~123304.

\bibitem{saintillan2012emergence}
{\sc D.~Saintillan and M.~J. Shelley}, {\em Emergence of coherent structures
  and large-scale flows in motile suspensions}, Journal of the Royal Society
  Interface, 9 (2012), pp.~571--585.

\bibitem{saintillan2013active}
{\sc D.~Saintillan and M.~J. Shelley}, {\em Active suspensions and their
  nonlinear models}, Comptes Rendus Physique, 14 (2013), pp.~497--517.

\bibitem{saintillan2015theory}
{\sc D.~Saintillan and M.~J. Shelley}, {\em Theory of active suspensions}, in
  Complex Fluids in biological systems, Springer, 2015, pp.~319--355.

\bibitem{sieber2020existence}
{\sc O.~Sieber}, {\em Existence of global weak solutions to an inhomogeneous
  {D}oi model for active liquid crystals}, arXiv preprint arXiv:2006.16832,
  (2020).

\bibitem{vskultety2020swimming}
{\sc V.~{\v{S}}kult{\'e}ty, C.~Nardini, J.~Stenhammar, D.~Marenduzzo, and
  A.~Morozov}, {\em Swimming suppresses correlations in dilute suspensions of
  pusher microorganisms}, Physical Review X, 10 (2020), p.~031059.

\bibitem{stein1993harmonic}
{\sc E.~M. Stein and T.~S. Murphy}, {\em Harmonic analysis: real-variable
  methods, orthogonality, and oscillatory integrals}, vol.~3, Princeton
  University Press, 1993.

\bibitem{subramanian2009critical}
{\sc G.~Subramanian and D.~L. Koch}, {\em Critical bacterial concentration for
  the onset of collective swimming}, Journal of Fluid Mechanics, 632 (2009),
  pp.~359--400.

\bibitem{subramanian2011stability}
{\sc G.~Subramanian, D.~L. Koch, and S.~R. Fitzgibbon}, {\em The stability of a
  homogeneous suspension of chemotactic bacteria}, Physics of Fluids, 23
  (2011), p.~041901.

\bibitem{sun2011global}
{\sc Y.~Sun and Z.~Zhang}, {\em Global well-posedness for the 2d micro-macro
  models in the bounded domain}, Communications in mathematical physics, 303
  (2011), pp.~361--383.

\bibitem{taylor1954dispersion}
{\sc G.~I. Taylor}, {\em The dispersion of matter in turbulent flow through a
  pipe}, Proceedings of the Royal Society of London. Series A. Mathematical and
  Physical Sciences, 223 (1954), pp.~446--468.

\bibitem{villani2010landau}
{\sc C.~Villani et~al.}, {\em Landau damping}, Notes de cours, CEMRACS,
  (2010).

\bibitem{dongyiweiresolvent}
{\sc D.~Wei}, {\em Diffusion and mixing in fluid flow via the resolvent
  estimate}, Sci. China Math., 64 (2021), pp.~507--518,
  \url{https://doi.org/10.1007/s11425-018-9461-8},
  \url{https://doi.org/10.1007/s11425-018-9461-8}.

\bibitem{kolmogorovhypocoercivitywei}
{\sc D.~Wei and Z.~Zhang}, {\em Enhanced dissipation for the {K}olmogorov flow
  via the hypocoercivity method}, Sci. China Math., 62 (2019), pp.~1219--1232,
  \url{https://doi.org/10.1007/s11425-018-9508-5},
  \url{https://doi.org/10.1007/s11425-018-9508-5}.

\bibitem{drivascotizelati}
{\sc M.~C. Zelati and T.~D. Drivas}, {\em A stochastic approach to enhanced
  diffusion}, {ANNALI} {SCUOLA} {NORMALE} {SUPERIORE} - {CLASSE} {DI}
  {SCIENZE},  (2021), pp.~811--834,
  \url{https://doi.org/10.2422/2036-2145.201911_013},
  \url{https://doi.org/10.2422/2036-2145.201911_013}.

\bibitem{gallay2021enhanced}
{\sc M.~C. Zelati and T.~Gallay}, {\em Enhanced dissipation and taylor
  dispersion in higher-dimensional parallel shear flows}, Journal of the London
  Mathematical Society,  (2023), \url{https://doi.org/10.1112/jlms.12782},
  \url{https://doi.org/10.1112/jlms.12782}.

\bibitem{zhang2008new}
{\sc H.~Zhang and P.~Zhang}, {\em On the new multiscale rodlike model of
  polymeric fluids}, SIAM journal on mathematical analysis, 40 (2008),
  pp.~1246--1271.

\end{thebibliography}
\end{document}